\documentclass[reqno, 11pt]{amsart}

\usepackage[all]{xy}
\usepackage{amsfonts}
\usepackage[mathcal]{eucal}
\usepackage{eufrak}
\usepackage{amssymb}
\usepackage{amsmath}
\usepackage{amscd}
\usepackage{color}
\usepackage[pagebackref,colorlinks]{hyperref}
\usepackage{enumerate}

\newcommand{\mi}{\setminus}
\newcommand{\lra}{\longrightarrow}
\newcommand{\ov}{\overline}

\newcommand{\jdotfont}{}
\font\jdotfont lcircle10  scaled 913 
\newcommand{\osbullet}{\jdotfont\char113} 
\newlength{\sbwd} 
\newcommand{\csbullet}{\kern.5\sbwd\osbullet\kern-.5\sbwd}

\newcommand{\id}{\operatorname{id}}

\newcommand{\bsi}{\boldsymbol\sigma}
\newcommand{\bb}{\mathbf b}
\newcommand{\bu}{\mathbf u}
\newcommand{\bone}{{\boldsymbol 1}}

\newcommand{\sC}{\mathsf{C}}
\newcommand{\sD}{\mathsf{D}}
\newcommand{\sE}{\mathsf{E}}
\newcommand{\sEz}{\mathsf{E}_0}
\newcommand{\sF}{\mathsf{F}}
\newcommand{\sT}{\mathsf{T}}
\newcommand{\sTz}{{\mathsf{T}_0}}
\newcommand{\sR}{\mathsf{R}}
\newcommand{\sRz}{{\mathsf{R}_0}}
\newcommand{\sB}{\mathsf{B}}
\newcommand{\sL}{\mathsf{L}}
\newcommand{\sM}{\mathsf{M}}
\newcommand{\sN}{\mathsf{N}}
\newcommand{\sI}{\mathsf{I}}
\newcommand{\sJ}{\mathsf{J}}
\newcommand{\sQ}{\mathsf{Q}}
\newcommand{\sS}{\mathsf{S}}
\newcommand{\sA}{\mathsf{A}}
\newcommand{\sV}{\mathsf{V}}
\newcommand{\sZ}{\mathsf{Z}}
\def\sTR{\sT/\sR}

\newcommand{\qq}{\mathbb{Q}}
\newcommand{\zz}{\mathbb{Z}}
\newcommand{\rr}{\mathbb{R}}

\newcommand{\Ga}{\Gamma}
\newcommand{\ga}{\gamma}
\newcommand{\gade}{{\gamma+\delta}}
\newcommand{\al}{\alpha}
\newcommand{\be}{\beta}
\newcommand{\de}{\delta}

\newcommand{\si}{\sigma}
\newcommand{\ep}{\varepsilon}

\newcommand{\wi}{\widetilde}
\newcommand{\wh}{\widehat}
\def\hsp{\mspace{1mu}}
\def\ccdot{\!\cdot \!}

\def\col{\! :\!}
\def\lan{\langle}
\def\ran{\rangle}
\def\op{^\text{op}}
\def\xg{x_\ga}
\def\xd{x_\de}
\def\xs{x_{\ga+\de}}
\def\ThetaE{\Theta_\sE}
\def\GaE{\Ga_\sE}
\def\quat#1#2#3{\big(\frac{#1,#2}{#3}\big)}

\newcommand\SK[1][1]{\operatorname{SK}_{#1}}
\newcommand{\Nrd}[1][{}]{{\operatorname{Nrd}_{#1}}}

\newcommand{\res}{\operatorname{res}}
\newcommand{\Aut}{\operatorname{Aut}}
\newcommand{\Hom}{\operatorname{Hom}}
\newcommand{\cors}{\operatorname{cor}}
\newcommand{\Dec}{\operatorname{Dec}}
\newcommand{\Ind}{\operatorname{Ind}} 
\newcommand{\intt}{\operatorname{int}}

\newcommand{\chr}{\operatorname{char}}
\newcommand{\ind}{\operatorname{ind}}

\newcommand{\End}{\operatorname{End}}
\newcommand{\sEnd}{\mathsf{End}}
\newcommand{\Gal}{\operatorname{Gal}}
\newcommand{\sGal}{\mathsf{Gal}}

\newcommand{\Br}{\operatorname{Br}}
\newcommand{\sBr}{\mathsf{Br}}
\newcommand{\modd}{\operatorname{mod}}

\newcommand{\gr}{\operatorname{{\sf gr}}}
\newcommand{\im}{\operatorname{im}}

\newcommand{\inv}{^{-1}}

\def\cupp{{\raise 0.9pt\hbox{$\,\scriptstyle\cup\,$}}}
\newcommand{\uij}{u_{ij}}

\theoremstyle{plain}
\newtheorem {lemma}{Lemma}[section]
\newtheorem {theorem}[lemma]{Theorem}

\newtheorem {corollary}[lemma]{Corollary}

\newtheorem {proposition}[lemma]{Proposition}
\newtheorem {prop}[lemma]{Proposition}

\theoremstyle{remark}
\newtheorem{remark}[lemma]{Remark}

\newtheorem {example}[lemma]{Example}

\theoremstyle{definition}

\numberwithin{equation}{section}

\title[Unitary  $\SK$ of semiramified division algebras]
{Unitary  $\boldsymbol{\SK}$ of semiramified graded and 
valued division algebras}

\begin{thanks}
{
The  author would like to thank 
 R. Hazrat and the Queen's University, Belfast 
and J.-P. Tignol and the Universit\'e Catholique de Louvain
for their 
hospitality while some of the research for this paper was carried out.
}
\end{thanks}

\author{A. R. Wadsworth}
\address{
Department of Mathematics\\
University of California at San Diego\\
La Jolla, California 92093-0112\\
U.S.A.}
 \email{arwadsworth@ucsd.edu}

\begin{document} 
\maketitle

\section{Introduction}

Let $D$ be a division algebra finite-dimensional over its center $K$.  Then,
$$
\SK(D) \ = \ \{ d\in D^* \mid\Nrd_D(d) = 1\} \, \big / \, [D^*,D^*], 
$$
where $\Nrd_D$ denotes the reduced norm and $[D^*,D^*]$ is the 
 commutator group of the group of units $D^*$ of~$D$.
If $D$ has a unitary involution $\tau$ (i.e., an involution $\tau$ on 
$D$ with $\tau|_K \ne \id$), then the unitary $\SK$ for $\tau$~on~$D$~
is
\begin{equation}\label{unitarydef} 
\SK(D, \tau) \ = \ \Sigma'_\tau(D) \, \big/ \, \Sigma_\tau(D), 
\end{equation}
where 
$$
\Sigma _\tau'(D) \ = \ \{d\in D^*\mid \Nrd_D(d) = \tau(\Nrd_D(d))\}
\qquad \text{and} \qquad \Sigma_\tau(D) \ = \ \big\lan \{ d\in D^*\mid
d = \tau(d)\}\big\ran.
$$
The groups $\SK(D)$ and $\SK(D,\tau)$ are of considerable interest 
as subtle invariants of $D$, and as reduced Whitehead groups for certain 
algebraic groups (cf.~\cite{tits}, \cite{platsurvey}, \cite{gille}).

In this paper we will prove formulas for $\SK(\sE)$ and $\SK(\sE, \tau)$ 
for $\sE$ a semiramified graded division algebra $\sE$ of finite rank over
its center.  In view of the isomorphisms in \cite[Th.~4.8]{hazwadsworth} and 
\cite[Th.~3.5]{I}, the formulas for $\sE$ imply analogous formulas for 
$\SK$ and unitary $\SK$ for a tame semiramified division algebra $D$ over a 
Henselian valued field $K$.  The formulas thus obtained in the Henselian case
generalize ones given by Platonov for $\SK(D)$ and Yanchevski\u\i\  
for  $\SK(D,\tau)$ for bicyclic decomposably 
semiramified division algebras over iterated Laurent fields. 
  Most of our work will be in the unitary setting, 
which is not as well developed as the nonunitary setting.  

Ever since Platonov  gave   examples of division algebras 
with  $\SK(D)$ nontrivial there has been ongoing interest in 
$\SK$.  Platonov  showed in \cite{platrat} that 
nontriviality of $\SK(D)$ implies that the algebraic group
group $\text{SL}_1(D)$ (with $K$-points $\{d\in D\mid \Nrd_D(d) = 1\}$) 
is not a rational variety.  Also, Voskresenski\u\i\  showed in 
\cite{V1} and \cite[Th., p.~186]{V2} that $\SK(D)\cong \text{SL}_1(D)/R$,  
the group of \mbox{$R$-equivalence} classes of the variety $\text{SL}_1(D)$.  
The corresponding unitary result, $\SK(D, \tau) \cong 
\text{SU}_1(D, \tau)/R$ was given in \cite[Remark, p.~537]{yy} and 
\cite[Th.~5.4]{cm}. More recently, Suslin in \cite{suslin1} and \cite{suslin2}
has related $\SK(D)$ to certain $4$-th cohomology groups associated to 
$D$, and has conjectured that whenever the Schur index~$\ind(D)$ 
is not square-free then $\SK(D\otimes_KL)$ is nontrivial for some 
field $L\supseteq K$. (This has been proved by Merkurjev 
in \cite{merkconj1} and \cite{merkconj2}  if $4|\ind(D)$, but remains 
open otherwise.) Nonetheless, explicit computable formulas for $\SK(D)$
and $\SK(D,\tau)$ have remained elusive, and are principally available, 
when $\ind(D) >4$, only for algebras over Henselian fields  
(cf.~\cite{ershov} and \cite[Th.~3.4]{hazwadsworth}) and quotients of 
iterated twisted polynomial algebras (cf.~\cite[Th.~5.7]{hazwadsworth}).   

Platonov's original examples with nontrivial $\SK$ 
in  \cite{platTannaka}  
and \cite{platonov} were division 
algebras $D$ over a twice iterated
Laurent power series field $K = k(((x))((y))$, where $k$ is a 
local or global field
or an infinite algebraic extension of such a  field.  His $K$ has a 
naturally associated rank $2$ Henselian valuation which extends uniquely 
to a valuation on $D$.  With respect to this valuation, his
$D$ is tame and \lq\lq decomposably semiramified"  and, in addition,  its residue division algebra
$\ov D$ is a field with $\ov D = L_1\otimes _kL_2$, where each $L_i$ is 
cyclic Galois over $k$.  His basic formula for such $D$ is:
\begin{equation}\label{platformula}
\SK(D) \, \cong\, \Br(\ov D/k) \big/ \big[\Br(L_1/k)\cdot \Br(L_2/k)\big],
\end{equation}
where $k$ is any field, 
$\Br(k)$ is the Brauer group of $k$, and for a field $M\supseteq k$, 
$\Br(M/k)$ denotes the relative Brauer
 group 
$\ker(\Br(k) \to \Br(M))$, a subgroup of $\Br(k)$.   
That $D$ is  tame and semiramified means that\break 
${[\ov D \col \ov K] = | \Ga_D \col \Ga_K| = \sqrt{[D\col K]}}$ and 
$\ov D$ is a field separable (hence abelian Galois) over $\ov K$,
where $\Ga_D$ is the value group the valuation on $D$.   We say that $D$ is 
{\it decomposably semiramified} (abbreviated~DSR) if $D$~is a tensor product
of cyclic tame and semiramified division algebras.   
Using \eqref{platformula}  with $k$ a global field or an algebraic 
extension of a
global field, Platonov   showed in \cite{plat76} that every finite
abelian group   and some infinite  abelian groups 
of bounded torsion appear as $\SK(D)$ for suitable~$D$.

Shortly after Platonov's work, Yanchevski\u\i\ obtained in \cite{yannounc},
\cite{y}, \cite{yinverse} 
similar results for 
the unitary $\SK$ for similar types of division algebras, namely 
$D$ decomposably semiramified over $K = k((x))((y))$, with $k$ any  
field, given that  $D$ has a unitary involution $\tau$ with fixed field $K^\tau
= \ell((x))((y))$ for some field $\ell \subseteq k$ with 
$[k\col \ell] = 2$.  Yanchevski\u\i's key formula (when $\ov D = 
L_1\otimes_k L_2$ as above) is:
\begin{equation}\label{yanchformula}  
\SK(D,\tau) \cong \Br(\ov D/k;\ell) \big/ \big[\Br(L_1/k;\ell)\cdot 
\Br(L_2/k;\ell)\big],
\end{equation}
where for a field $M\supseteq k$, $\Br(M/k;\ell) = 
\ker\big(\text{cor}_{k\to \ell}\colon
\Br(M/k) \to \Br(\ell)\big)$; this is the subgroup of $\Br(k)$ consisting
of the classes of central simple $k$-algebras split by $M$ and having a unitary involution 
$\tau$ with fixed field ${k^\tau = \ell}$.        
He used this in \cite{yinverse}  with $k$~and~$\ell$ global fields 
to show that any finite abelian group is realizable as $\SK(D,\tau)$.
He obtained remarkably similar analogues for the unitary $\SK$ to 
other results of Platonov for the nonunitary $\SK$, but generally with 
substantially more difficult and intricate proofs.

Ershov showed in \cite{e} and \cite{ershov} that the natural setting 
for viewing Platonov's examples of nontrivial $\SK(D)$ is that of 
tame division algebras $D$ over a Henselian valued field $K$.
(Platonov considered his $K$ 
in a somewhat cumbersome way as a field with complete discrete
valuation with residue field which also has a complete discrete valuation.)
The Henselian valuation on $K$ has a unique extension to a valuation
on~$D$, and Ershov gave exact sequences that describe $\SK(D)$ in terms 
of various data related to the residue division ring $\ov D$. In particular
he showed (combining \cite[p.~69, (6) and Cor.~(b)]{ershov}) that if 
$D$ is DSR (with $K$ Henselian), then 
\begin{equation}\label{DSRSK1} 
\SK(D)  \, \cong\,  \wh H\inv(\Gal(\ov D/\ov K), \ov D^*).   
\end{equation}

More recently, there has been work on associated graded rings of 
valued division algebras, see especially \cite{hwcor}, \cite{mounirh},
\cite{tignolwadsworth}. The tenor of this work has been that for 
a tame division algebra $D$ over a Henselian valued field, most
of the structure of $D$ is inherited by its associated graded ring 
$\gr(D)$, while $\gr(D)$ is often much easier to work with 
than $D$ itself.  This theme was applied quite recently by 
R.~Hazrat and the author in \cite{hazwadsworth} and \cite{I} to 
calculations of $\SK$ and unitary $\SK$. It was shown in 
\cite[Th.~4.8]{hazwadsworth} that if $D$ is tame over $K$ with respect to
a Henselian valuation, then $\SK(D) \cong\SK(\gr(D))$; the corresponding 
result for unitary $\SK$ was proved in \cite[Th.~3.5]{I}.  Calculations of 
$\SK$ in the graded setting are significantly easier and more transparent 
than in the original ungraded setting, allowing almost effortless recovery
of Ershov's exact sequences, with some worthwhile improvements.  Notably, 
it was shown in \cite[Cor.~3.6(iii)]{hazwadsworth} that if
$K$ is Henselian and 
$D$ is tame and semiramified (but not necessarily DSR), then there is an 
exact sequence 
\begin{equation}\label{hwsemiram}
H\wedge H \, \lra\,  \wh H\inv(H, \ov D^*) \, \lra\,  \SK(D) 
\, \lra\,  1, \qquad
\text{where}\qquad H \,=\, \Gal(\ov D/\ov K) \, \cong\, 
\Ga_D/\Ga_K.
\end{equation}
    When $D$ is DSR, 
the  image of $H\wedge H$ in $\wh H\inv(H, \ov D^*)$ is trivial, yielding
\eqref{DSRSK1}.  Then, Platonov's formula
 \eqref{platformula} is obtained from \eqref{DSRSK1} via the following isomorphism:  For a field 
$M = L_1\otimes_k L_2$ where each $L_i$ is cyclic Galois over $k$,
\begin{equation}\label{bicycliccohom}
\wh H\inv(\Gal(M/k), M^*) \ \cong \ \Br(M/k)/\big[\Br(L_1/k)
\cdot \Br(L_2/k) \big].
\end{equation}
See \eqref{bicyclicbrel}--\eqref{njnj} below for a short proof of 
\eqref{bicycliccohom} using facts about abelian crossed products.

When $D$ is semiramified but not DSR, the contribution of the first 
term in \eqref{hwsemiram} can be better understood in terms of the 
$I\otimes N$ decomposition of $D$:  Our semiramified $D$ is equivalent 
in $\Br(K)$ to $I\otimes_K N$, where $I$~is inertial (= unramified) over
$K$ and $N$ is DSR, so $\ov N \cong \ov D$ and $\Ga_N = \Ga_D$.  Thus,
the $\wh H\inv$ term in~\eqref{hwsemiram} coincides with 
$\SK(N)$. We will show  in Cor.~\ref{henselcor}(i) below that the image of 
$H\wedge H$ in $\wh H\inv(H, \ov D^*)$
is expressible in terms of parameters describing the residue algebra
$\ov I$ of $I$, which is central simple over $\ov K$ and 
split by the field~$\ov D$. This $\ov I$ does not show up within 
$D$ or $\ov D$, but nonetheless has significant influence on the 
structure of $D$. (For example, it determines whether $D$ can be a crossed product
or nontrivially decomposable---see \cite[pp.~162--166, Remarks~5.16]{jw}.
In \cite{jw} DSR algebras were called \lq\lq nicely semiramified," and abbreviated
NSR.  We prefer the more descriptive term decomposably semiramified.) 
  Also, 
$\ov I$~is not uniquely determined by~$D$, but determined only modulo the 
group $\Dec(\ov D/\ov K)$ of simple $\ov K$-algebras which 
\lq\lq decompose according to~$\ov D$\,"---see \S3 below for the definition
of $\Dec(\ov D/\ov K)$.  In the bicyclic  case where $D$ is semiramified and 
$K$ Henselian and $\ov D \cong L_1\otimes _{\ov K} L_2$ with each 
$L_i$ cyclic Galois over~$\ov K$, we will show in Cor.~\ref{henselcor}(ii)
that 
\begin{equation}\label{ubicyclicsemiram}
\SK(D) \ \cong  \ \Br(\ov D/\ov K) \big/ \big[\Br(L_1/\ov K)
\cdot \Br(L_2/\ov K)\cdot\lan[\ov I]\ran\big],
\end{equation}
which is a natural generalization of Platonov's formula \eqref{platformula}.

The principal aim of this paper is to prove unitary versions of the
results described above for nonunitary~$\SK$, especially \eqref{DSRSK1}, 
\eqref{bicycliccohom}, and \eqref{ubicyclicsemiram}.  The
unitary versions of these are, respectively, Th.~\ref{unitaryDSR}(i),
Prop.~\ref{unitarybicyclic}, and Th.~\ref{main}(ii). Along the way, it 
will be necessary to develop a unitary version of the $I\otimes N$
decomposition for semiramified division algebras.  This is given in Prop.~
\ref{uINdecomp}.  In the final section we will apply some of these formulas 
to give an example where the natural map $\SK(D,\tau) \to \SK(D)$ is 
not injective.  

This paper is a sequel to \cite{I}, which describes the equivalence of the
graded setting and the Henselian valued setting for computing
unitary $\SK$, and has calculations of $\SK(D,\tau)$ for several cases other 
than the semiramified one considered here.  However, the present paper can be read 
independently of \cite{I}.
We will work here primarily with 
graded division algebras, where the calculations are more transparent
than for valued algebras.
Some basic background on the graded objects is given in \S\ref{graded}.
But we reiterate  that by \cite[Th.~3.5]{I} every result in the graded 
setting yields a corresponding result for tame division algebras over 
Henselian valued fields. While what is proved here is for a rather
specialized type of algebra, we note that detailed knowledge of $\SK$ in special cases 
sometimes has wider consequences.  See, e.g., the paper~\cite{rty} 
where Suslin's conjecture is reduced to the case of cyclic algebras.   
See also \cite[Th.~4.11]{w}, where 
the proof of nontriviality of a cohomological invariant of Kahn 
uses a careful analysis of $\SK(D)$ for the $D$ in Platonov's 
original example.   

From the perspective of algebraic groups, it is perhaps unsurprising that 
there should be  results for the unitary $\SK$ similar to those in the  
nonunitary case.   For, $\text{SL}_1(D)$ is a group of inner 
type $A_{n-1}$ where $n = \deg(D)$, and $\text{SU}_1(D, \tau)$ is a group 
of outer type $A_{n-1}$ (cf.~\cite[Th.~(26.9)]{kmrt}).  
Nonetheless, the similarities in 
formulas for $\SK(D,\tau)$   given in Yanchevski\u\i's work 
and in \cite{I} and here to those for $\SK(D)$ seem quite striking.   
Likewise, the results by 
Rost on $\SK(D)$ for biquaternion algebras (see \cite[\S17A]{kmrt}) and 
by Merkurjev in \cite{merkurjev} 
for arbitrary algebras of degree $4$, have a unitary analogue 
proved by Merkurjev in \cite{merk2}.
This suggests that a further 
analysis of the unitary $\SK$  would be worthwhile, notably to 
investigate whether there are unitary versions of the deep results   
by Suslin \cite{suslin2} and
Kahn \cite{kahn} relating $\SK(D)$ to higher \'etale cohomology groups.


\section{Graded division algebras and simple algebras}\label{graded}

We will be working throughout with graded algebras 
graded by a torsion-free abelian group.  We now set 
up the terminology for such algebras and recall some 
of the basic facts we will use frequently.  

Let $\Ga$ be a torsion-free abelian group, and let 
$\sR$ be a ring graded by $\Ga$, i.e., $\sR = \bigoplus
_{\ga\in \Ga}\sR_\ga$, where each $\sR_\ga$ is an 
additive subgroup of $\sR$ and $\sR_\ga \cdot \sR_\de
\subseteq \sR_{\ga +\de}$ for all $\ga, \de\in \Ga$.  
The homogeneous elements of $\sR$ are those lying in
$\bigcup_{\ga\in \Ga}\sR_\ga$. If $r\in \sR_\ga$, $r\ne 0$, 
then we write $\deg(r) = \ga$.  The grade set of 
$\sR$ is $\Ga_\sR = \{ \ga \in \Ga\mid 
\sR_\ga \ne \{0\}\}$.  (We work only with gradings by 
torsion-free abelian groups because we are interested 
in the associated  graded rings determined by 
valuations on division algebras; for such rings the 
grading is indexed by the value group of the valuation, which is 
torsion-free abelian.)  If $\sR' = \bigoplus_{\ga\in \Ga}
\sR'_\ga$ is another graded ring, a {\it graded ring    
homomorphism} $\varphi\colon \sR\to \sR'$ is a 
ring homomorphism such that $\varphi(\sR_\ga) \subseteq 
\sR'_\ga$ for all $\ga\in \Ga$.  If $\varphi$~is an 
isomorphism, we say that $\sR$ and $\sR'$ are graded 
ring isomorphic, and write $\sR \cong_g \sR'$.
For example, if $a\in \sR$ is homogeneous and $a\in \sR^*$, 
the group of units of $\sR$, then the map $\intt(a)\colon
\sR\to\sR$ given by $r\mapsto ara\inv$ is a graded ring 
automorphism of $\sR$. 

A graded ring $\sE = \bigoplus_{\ga \in \Ga} \sE_\ga$ 
is said to be a {\it graded division ring} if every 
nonzero homogeneous element of $\sE$ lies in the 
multiplicative group $\sE^*$ of units of $\sE$.  See
\cite{hwcor} for background on graded division ring and 
proofs of the properties mentioned here.  
Notably (as $\Ga$ is torsion-free abelian), $\sE$ has no 
zero divisors, $\sE^*$~consists entirely of homogeneous 
elements, $\Ga _\sE$ is a subgroup of $\Ga$,  \,
$\sEz$ is a division ring, and each nonzero 
homogeneous component
$\sE_\ga$ of $\sE$ is a $1$-dimensional left and right 
$\sEz$-vector space.  Furthermore, if $\sM$ is any left graded 
$\sE$- module (i.e., an $\sE$-module such that 
$\sM = \bigoplus_{\ga \in \Ga}\sM_\ga$ with 
$\sE_\ga \!\cdot\! \sM_\de \subseteq \sM_{\ga +\de}$
for all~$\ga, \de \in \Ga$), then $\sM$ is a free 
$\sE$-module with a homogeneous base, and any two such bases
have the same cardinality; this cardinality is called the 
dimension of $\sM$ and denoted $\dim_\sE(\sM)$.
Any such $\sM$ is therefore called a left graded  $\sE$-vector 
space. 

A commutative graded division ring $\sT = \bigoplus
_{\ga \in \Ga}\sT_\ga$ is called a {\it graded field}.  
Such a $\sT$  is an integral domain;
let $q(\sT)$ denote the quotient field of $\sT$.
A graded ring $\sA$ which is a 
$\sT$-algebra is called a {\it graded $\sT$-algebra} if the module 
action of $\sT$ on $\sA$ makes $\sA$ into a graded 
$\sT$-module.  When this occurs, $\sT$ is graded isomorphic 
to a graded subring of the center of $\sA$, which is denoted 
$Z(\sA)$. { \it All graded $\sT$-algebras considered in this
paper are assumed to be finite-dimensional graded 
$\sT$-vector spaces.}  Note that if $\sA$ is  a graded 
\mbox{$\sT$-algebra},
then $\sA \otimes _\sT q(\sT)$ is a $q(\sT)$-algebra of the same dimension.
That is, $[\sA\col\sT] = [\sA \otimes _\sT q(\sT)\col q(\sT)]$, 
where  $[\sA\col\sT]$~denotes $\dim_\sT(\sA)$ and 
$[\sA \otimes _\sT q(\sT)\col q(\sT)] = 
\dim_{q(\sT)}(\sA \otimes _\sT q(\sT))$.

Note that if $\sA$ and $\sB$ are graded algebras over a graded
field $\sT$ then $\sA \otimes _\sT \sB$ is also a graded
$\sT$-algebra with $(\sA \otimes _\sT \sB)_\ga
= \sum _{\de\in \Ga} \sA_\de \otimes_\sTz \sB_{\ga - \de}$ 
for all $\ga \in \Ga$.  Clearly, $\Ga_{\sA \otimes _\sT \sB}
= \Ga_\sA + \Ga_\sB$. Also, if  $C$ is a finite-dimensional 
 $\sTz$-algebra, then $C\otimes _\sTz \sA$ is a graded 
$\sT$-algebra with $(C\otimes _\sTz \sA)_\ga = C\otimes _\sTz 
\sA_\ga$ for all $\ga \in \Ga$, and $\Ga_{C\otimes _\sTz \sA}
= \Ga_\sA$.

A graded $\sT$-algebra $\sA$ is said to 
be {\it simple} if it has no homogeneous two-sided ideals
except $\sA$ and~$\{0\}$.  $\sA$ is called a 
{\it central simple} $\sT$-algebra if in addition its center
$Z(\sA)$ is $\sT$.  
The theory of  simple graded algebras 
is analogous to the usual theory of finite-dimensional simple 
algebras.  This is described in \cite[\S 1]{hwcor}, where proofs of 
the following facts can be found.  There is a graded
Wedderburn Theorem for simple graded algebras:  Any such
$\sA$ is graded isomorphic to $\sEnd_\sE(\sM)$ for some 
finite-dimensional graded vector space $\sM$ over a graded 
division algebra $\sE$, and $\sE$ is unique up to graded 
isomorphism.  Also, while $\sA_0$ need not be simple, it is 
always semisimple, and ${\sA_0\cong \prod_{j= 1}^s M_{\ell_j}
(\sEz)}$ for some $\ell_j\times \ell_j$ matrix rings
over $\sEz$
(see the proof of Lemma~\ref{zerosimple} below).
We write $[\sA]$ for the equivalence class of~$\sA$ under 
the equivalence relation $\sim_g$ given by:
$\sA \sim_g \sA'$ iff 
$\sA\cong_g \sEnd_\sE(\sM)$ and ${\sA' \cong_g
\sEnd_\sE(\sM')}$ for the same graded division algebra~
$\sE$.  The Brauer group (of graded algebras) 
for $\sT$ is 
$$
{\sBr(\sT) = \{[\sA]\mid \sA \text{ is
a graded central simple $\sT$-algebra}\}},
$$
\vfill\eject
\noindent with the 
well-defined group operation $[\sA]\cdot [\sA'] = 
[\sA \otimes_\sT \sA']$.  When $\sA \cong_g \sEnd_\sE(\sM)$ as
above, then $[\sA] = [\sE]$, and up to graded isomorphism
$\sE$ is the only graded division algebra
with $\sA \sim_g \sE$. 
There is a graded version 
of the Double Centralizer Theorem, see \cite[Prop.~1.5]{hwcor}
and also the Skolem-Noether Theorem, see \cite[Prop.~1.6]{hwcor}.
We recall the latter, since it has an added condition not appearing
in the ungraded version.

\begin{proposition}[{\cite[Prop.~1.6(b),(c)]{hwcor}}]\label{grSN}
Let $\sA$ be a central simple graded algebra over the graded field~$\sT$,
and let $\sB$ and $\sB'$ be simple graded $\sT$-subalgebras of $\sA$. Let
$\sC = C_\sA(\sB)$, the centralizer of $\sB$ in $\sA$,  
   and let $\sZ = Z(\sC)  = Z(\sB)$ and 
$\sC' = C_\sA(\sB')$.  Let $\alpha\colon
\sB \to \sB'$ be a graded $\sT$-algebra isomorphism. Then there is 
a homogeneous $a\in A^*$ such that $\alpha(b) = aba\inv$ for all $b\in \sB$
if and only if there is a graded $\sT$-algebra isomorphism 
$\gamma\colon \sC \to \sC'$ such that  $\gamma|_Z = \alpha|_\sZ$.
Such a $\gamma$ exists whenever $\sC_0$ is a division ring.
\end{proposition}

If $\sE$ is a graded division algebra over a graded field
$\sT$, we write $[\sE\col\sT]$ for $\dim_\sT(\sE)$.  
A basic fact is the Fundamental Equality
\begin{equation}\label{fun}
[\sE\col \sT] \ = \ [\sEz\col\sTz] \, |\Ga_\sE\col\Ga_\sT|,
\end{equation}  
where $|\Ga_\sE\col\Ga_\sT|$ denotes the index in 
$\Ga_\sE$ of its subgroup $\Ga_\sT$.  Also,
it is known that $Z(\sEz)$ is abelian Galois over
$\sTz$, and there is a well-defined group 
epimorphism 
\begin{equation}\label{Theta}
\Theta_\sE \colon\Ga_\sE \to \Gal(Z(\sEz)/\sTz)
\quad\text{ given by } \Theta_\sE(\ga)(z) = aza\inv
\text { for any $z\in Z(\sEz)$ and $a\in \sE_\ga\setminus
\{0\}$}.
\end{equation}
Clearly, $\Ga_\sT\subseteq \ker(\Theta_\sE)$, so 
$\Theta_\sE$ induces an epimorphism of finite groups 
$\ov \Theta_\sE\colon\Ga_\sE/\Ga_\sT \to \Gal(Z(\sEz)/\sEz)$.

The terminology for different cases in \eqref{fun} is carried 
over from valuation theory:  We say that a graded 
field $\sS\supseteq \sT$ is {\it inertial over $\sT$}
if $[\sS_0\col\sTz] = [\sS\col\sT]<\infty$ and the 
field $\sS_0$ is separable over $\sTz$.  When this occurs,
$\Ga_\sS = \Ga_\sT$, and 
the graded monomorphism $\sS_0\otimes _\sTz \sT \to 
\sS$ given by multiplication in $\sS$ is surjective 
by dimension count; so $\sS\cong_g \sS_0\otimes _\sTz\sT$.
At the other extreme, we say that a graded field 
$\sJ\supseteq \sT$
is {\it totally ramified over $\sT$} if 
$|\Ga_\sJ\col\Ga_\sT|= [\sJ\col\sT] <\infty$.  
When this occurs, $\sJ_0 = \sTz$ and, more generally, 
for any $\ga \in \Ga_\sT$, we have $\sJ_\ga = \sT_\ga$
since $\dim_\sTz(\sJ_\ga) = \dim_{\sJ_0}(\sJ_\ga) = 1
= \dim_\sTz(\sT_\ga)$.

There is an extensive theory of finite-degree graded field  
extensions;  \cite{hwalg} is a good reference for what we 
need here.  Notably, there is a version of Galois theory:
  For graded fields $\sT \subseteq \sF$,
with $[\sF\col\sT] <\infty$, the (graded) Galois group 
of $\sF$ over $\sT$ is defined to be:
\begin{equation*}
\sGal(\sF/\sT) \ =  \ \{\psi\colon \sF \to \sF\mid
\psi \text{ is a graded field automorphism of $\sF$ and \  
$\psi|_\sT = \id$}\}.
\end{equation*}
Galois theory for graded fields follows easily from the 
classical ungraded theory since for the quotient fields
of $\sF$ and $\sT$ 
we have $q(\sF) \cong \sF \otimes _\sT q(\sT)$, 
so $[q(\sF)\col q(\sT)] = [\sF\col \sT]$, 
and there is a canonical isomorphism
${\sGal(\sF/\sT) \to \Gal(q(\sF)/q(\sT))}$ (the usual Galois
group) given by $\psi \mapsto \psi \otimes \id_{q(\sT)}$
(see \cite[Cor.~2.5(d), Th.~3.11 ]{hwalg}).   Thus, $\sF$~is 
Galois over $\sT$ iff $q(\sF)$ is Galois over $q(\sT)$,
iff $|\sGal(\sF/\sT)| = [\sF\col\sT]$, iff $\sT$
is the fixed ring of $\sGal(\sF/\sT)$.  This will arise 
here primarily in the inertial case: Suppose $\sS$ is a graded
field which contains and is inertial over~$\sT$, with 
$[\sS\col\sT]<\infty$.  For any $\psi\in \sGal(\sS/\sT)$ 
clearly the restriction $\psi|_{\sS_0}$ lies in 
$\Gal(\sS_0/\sTz)$.  Moreover, as 
 $\sS \cong_g \sS_0 \otimes_\sTz \sT$, for any 
$\rho \in \Gal(\sS_0/\sTz)$ we have $\rho \otimes \id_\sT
\in \sGal(\sS/\sT)$.  Thus, the restriction map 
$\psi \mapsto \psi|_{\sS_0}$ yields a canonical 
isomorphism $\sGal(\sS/\sT) \to \Gal(\sS_0/\sTz)$.  Hence,
as $[\sS\col\sT] = [\sS_0\col\sTz]$, $\sS$ is Galois over 
$\sT$ iff $\sS_0$ is Galois over $\sTz$. 

Just as in the ungraded case, we can use Galois graded 
field extensions to build central simple graded  algebras.
If $\sF$ is a Galois graded field extension of 
$\sT$, set $G = \sGal(\sF/\sT)$ and take any $2$-cocycle
$f\in Z^2(G, \sF^*)$.  Then we can build a crossed product
graded algebra $\sB = (\sF/\sT, G, f) = \bigoplus_{\si \in G}
\sF x_\si$ with multiplication given by $(ax_\si)(bx_\rho)
= a \si(b) f(\si,\rho) x_{\si\rho}$ for all $a,b \in \sF$, 
$\si, \rho\in G$. The grading is given by viewing 
$\sB$ as a left graded $\sF$-vector space with 
$(x_\si)_{\si\in G}$ as a homogeneous base with 
$\deg(x_\si) = 
\frac 1{|G|}\sum _{\rho\in G}\deg(f(\si,\rho))$.
A short calculation shows that $\deg(f(\sigma,\tau)\,x_{\sigma\tau})
= \deg(x_\sigma) + \deg(x_\tau)$ for all $\sigma,\tau\in G$; it follows
easily that $\sB$ is a graded $\sT$-algebra.
Indeed, $\sB$ is a simple graded algebra with $Z(\sB) 
\cong_g \sT$.  Conversely, if $\sA$~is any central 
 simple graded  
$\sT$-algebra containing~$\sF$ as a strictly maximal graded
subfield (i.e., ${[\sF\col\sT] = \deg(\sA) \,(= 
\sqrt{\dim_\sT(\sA)}\,\,)}$, then by the graded Double 
Centralizer Theorem
$C_\sA(\sF)=\sF= Z(\sF)$;
 so the graded Skolem-Noether
Theorem, Prop.~\ref{grSN} above, applies to the graded
isomorphisms in $G$, which yields that 
  $\sA\cong_g (\sF/\sT,G,f)$
for some $f\in Z^2(G, \sF^*)$.  From this one deduces, as
in the ungraded case,
that $\sBr(\sF/\sT) \cong H^2(G,\sF^*)$, where 
$\sBr(\sF/\sT)$ denotes the kernel of the canonical
map $\sBr(\sT) \to \sBr(\sF)$ given by $[\sA] \mapsto
[\sA\otimes_\sT \sF]$.  In particular, if $\sGal(\sF/\sT)$,
is cyclic, say with generator $\si$, then for any 
$b\in \sT^*$ we have the graded cyclic algebra
$\sC = (\sF/\sT, \si, b) = \bigoplus_{i=0}^{r-1}\sF y^i$, 
in which $ya = \si(a) y$ for all $a\in \sF$ and $y^r = b$, 
where $r = [\sF\col\sT]$. For the grading, we view $\sC$ as a 
left graded  $\sF$-vector space with homogeneous base
$(1, y, y^2, \ldots, y^{r-1})$ with 
$\deg(y^i) = \frac ir \deg(b)$.  Then $\sC$ is a 
 central  simple graded  $\sT$-algebra.

There are also norm maps in the graded setting:  If $\sT \subseteq \sF$
are graded fields with $[\sF\col\sT] <\infty$, then because 
$\sF$ is a free module the norm $N_{\sF/\sT}\colon \sF \to \sT$
can be defined by $c\mapsto \det(\lambda_c)$, where for 
$c\in \sF$, $\lambda _c\in \Hom_\sT(\sF,\sF)$ is the map $a\mapsto ca$.
Clearly, $N_{\sF/\sT}(c) = N_{q(\sF)/q(\sT)}(c)$, where $N_{q(\sF)/q(\sT)}$ 
is the usual norm for the quotient fields.  Also, if $c\in \sF$ is 
homogeneous, say $c\in \sF_\gamma$, then $N_{\sF/\sT}(c) \in 
\sT_{[\sF:\sT]\gamma}$.  Likewise, if $\sB$ is a central simple
graded $\sT$-algebra, then it is known that $\sB$ is an Azumaya algebra 
of constant rank $[\sB\col\sT]$ over~$\sT$;  hence there is a reduced
norm map $\Nrd_{\sB}\colon \sB \to \sT$.  It is easy to see that 
for the central ring of quotients $q(\sB)= \sB\otimes _\sT q(\sT)$ of $\sB$, we have 
$q(\sB)$ is a central simple algebra over the field
$q(\sT)$, and it is known 
(see \cite[proof of Prop.~3.2(i)]{hazwadsworth}) that for any $b\in \sB$, 
$\Nrd_{\sB}(b) = \Nrd_{q(\sB)}(b)$, where  $\Nrd_{q(\sB)}\colon q(\sB)
\to q(\sT)$ is the  reduced norm for $q(\sB)$. As usual, $b\in \sB^*$ iff 
$\Nrd_\sB(b)\in \sT^*$.  Also, if $b\in \sB_\gamma$, then  
$\Nrd_\sB(b) \in \sT_{\deg(\sB) \gamma}$.
Now assume further that $\sB$ is a graded division algebra, 
so that all its units are homogeneous.  Then
for the commutator group $[\sB^*, \sB^*]$ of $\sB$, 
we have
$[\sB^*, \sB^*] \subseteq \{b \in \sB\mid \Nrd_\sB(b) = 1\}
\subseteq \sB_0^*$.  We define
\begin{equation}\label{grsk} 
\SK(\sB) \ = \ \{b \in \sB\mid \Nrd_\sB(b) = 1\}\,\big/\, [\sB^*,\sB^*].
\end{equation} 
The fact that both terms in the right quotient lie in $\sB_0^*$ often
makes that calculation of $\SK(\sB)$ much more tractable in this graded 
setting than for ungraded division algebras.   

We need  terminology for some types of  simple graded  algebras 
and graded division algebras over a graded field~$\sT$.  
A central simple graded  $\sT$-algebra $\sI$  is said to be 
{\it inertial} (or unramified) if $[\sI_0\col\sT_0] = 
[\sI\col\sT]$.  When this occurs, the injective
graded $\sT$-algebra homomorphism
$\sI_0 \otimes _\sTz \sT \to \sI$ is surjective by 
dimension count.  So, $\Ga_\sI = \Ga_\sT$ and
$\sI\cong_g \sI_0\otimes _\sTz \sT$.  
Hence, $\sI_0$ must be a central simple 
$\sTz$-algebra.  Moreover, if we let $D$ be the
$\sTz$-central division algebra with $\sI_0\cong
M_\ell(D)$, then $D\otimes_\sTz \sT$ is clearly a graded 
division algebra over $\sT$ which is also inertial over $\sT$,
and $D\otimes_\sTz \!\sT\sim_g \sI$ (see Lemma~\ref{zerosimple}
below). 

The principal focus of this paper is on calculating 
$\SK$ and unitary $\SK$ for semiramified 
graded division algebras. 
Let $\sE$ be a central graded division algebra over 
a graded field $\sT$.  This $\sE$ is said to be 
{\it semiramified} if ${[\sE_0\col\sTz] = 
|\Gamma_\sE\col\Gamma_\sT| = \deg(\sE)}$ 
and $\sE_0$ is a field.  Since $\sEz = Z(\sEz)$, $\sEz$ is 
abelian Galois over~$\sTz$ and the epimorphism
$\ov\Theta_\sE\colon \Gamma_\sE/\Gamma_\sT\to 
\Gal(\sEz/\sTz)$ (see \eqref{Theta}) must be an isomorphism as 
${|\Ga_\sE/\Ga_\sT| = [\sEz\col \sTz] = |\Gal(\sEz/\sTz)|}$.
Furthermore, $\sE$ has the graded subfield 
$\sEz\sT \cong_g\sEz\otimes _\sTz \!\sT$, which is inertial and 
Galois over $\sT$ with $\sGal(\sEz\sT/\sT) 
\cong\Gal(\sEz/\sTz)$.  Because $[\sEz\sT \col \sT] = 
\deg(\sE)$, the graded Double Centralizer Theorem 
\cite[Prop.~1.5]{hwcor} shows that $C_\sE(\sEz\sT) = \sEz\sT$, 
and hence $\sEz\sT$ is a maximal graded subfield of $\sE$;
thus, $\sE$ is a graded abelian crossed product, as will be 
discussed in \S 3.

There is a significant special class of semiramified graded 
division algebras which are building blocks for  all semiramified algebras.  
We say that a $\sT$-central graded division algebra $\sN$
is {\it decomposably semiramified} (abbreviated DSR)
if $\sN$ has a  maximal graded subfield $\sS$ which is 
inertial over $\sT$ and another  maximal  graded subfield ~
$\sJ$ which is totally ramified over $\sT$.  The 
graded Double Centralizer Theorem yields that 
${[\sS\col \sT] = [\sJ\col\sT] = \deg(\sN)}$.  We thus have
\begin{equation}\label{DSRineqs}
\deg(\sN)\, = \, [\sJ\col\sT] \, = \, |\Ga_\sJ\col \Ga_\sT|
\, \le \, |\Ga_\sN\col \Ga_\sT| \quad \text{ and }
\quad \deg(\sN) \,=\, [\sS\col\sT]\, = \,[\sS_0\col\sTz]
\,\le [\sN_0\col\sTz].
\end{equation}
Since $|\Ga_\sN\col\Ga_\sT| \,[\sN_0\col\sTz] = [\sN\col\sT]
= \deg(\sN)^2$, the inequalities in \eqref{DSRineqs} must
be equalities, showing that $\sN_0 = \sS_0$ and $\Ga_\sN=
\Ga_\sJ$, hence $\sN$ is semiramified.  We call such an $\sN$
decomposably semiramified because it is always decomposable 
into a tensor product of cyclic semiramified graded division 
algebras (see Prop.~\ref{DSRdecomp} below for the unitary 
analogue to this).  The older term for such algebras 
is nicely semiramified (NSR).  

While our focus in this paper is on central graded division 
algebras we will often take tensor products of such algebras, 
obtaining simple graded  algebras which may have zero divisors.
The next lemma  allows us to recover information 
about the graded division algebra Brauer equivalent to such a 
tensor product.

\begin{lemma}\label{zerosimple}
Let $\sB$ be a central simple graded algebra over
the graded field $\sT$.  Let $\sD$ be the graded division 
algebra Brauer equivalent to $\sB$.  Suppose $\sB_0$ is a 
simple ring.  Then,
\begin{enumerate}[\upshape (i)]
\item
$\sB \!\cong_g\! M_\ell(\sD)$ for some $\ell$, where the matrix 
ring $M_\ell(\sD)$ is given the standard grading in 
which\break ${\big(M_\ell(\sD)\big)_\ga = M_\ell(\sD _\ga)}$ for 
all $\ga \in \Ga_\sD$.  Hence, $\sB_0\cong M_\ell(\sD_0)$,
$\Ga_\sB =\Ga'_\sB= \Ga_\sD$, and  $\Theta_\sB = \Theta_\sD$, 
where ${\Ga_\sB' = \{\deg(b) \mid b\in B^* \text{ and } b 
\text{ is homogeneous}\}}$, and
\begin{equation}\label{simpleTheta}
 \Theta_\sB\colon 
\Ga_\sB' \to \Gal(Z(\sB_0)/\sTz) \text{ is 
given by } \deg(b) \mapsto \intt(b)|_{Z(\sB_0)},
\text{ for any homogeneous $b\in \sB^*$}
\end{equation} 
where $\intt(b)$ denotes conjugation by $b$.
\item
$\sB$ is a graded division algebra if and only if $\sB_0$
is a division ring.
\end{enumerate}
\end{lemma}

\begin{proof}
(i)  By the graded Wedderburn Theorem \cite[Prop.~1.3]{hwcor},
$\sB\cong_g \sEnd_\sD(\sV)$ for some right graded  
vector space~$\sV$ of $\sD$.  The grading on 
$\sEnd_\sD(\sV)$ is given by 
$$
\big(\sEnd_\sD(\sV)\big)_\ep
\ = \  \{f\in \sEnd_\sD(\sV) \mid f(\sV_\delta) \subseteq
\sV_{\ep+\delta} \text{ for all }\delta \in \Ga_\sV\}.
$$
Take a homogeneous
$\sD$-base $(v_1, \ldots, v_\ell)$ of $\sV$, and let
$\gamma_i = \deg(v_i)$, for $1\le i\le \ell$; then, 
$\Ga_\sV = \bigcup_{i = 1}^\ell \, \ga_i + \Ga_\sD$.
Let $\delta_1 +\Ga_\sD, \ldots, \de_s+\Ga_\sD$ be the 
distinct cosets of $\Ga_\sD$ appearing in $\Ga_\sV$,
and let $t_j$ be the number of $i$ with $\gamma_i\in
\de_j+ \Ga_\sD$. So, $t_1 + \ldots + t_s = \ell$.
By replacing each  $v_i$ by a 
$\sD^*$-multiple of it,  we may assume that 
$\deg(v_i) = \de_j$ whenever $\ga_i\in \de_j+\Ga_\sD$.
Then, we can reindex $(v_1, \ldots, v_\ell) = 
(v_{1\hsp 1}, \ldots, v_{1\hsp t_1}, \ldots, v_{s\hsp 1}, 
\ldots, v_{s \hsp t_s})$ with $\deg(v_{jk}) =
\de_j$ for all~$j,k$.  Then,
$\sV_{\de_j} = \sD_0\text{-span}(v_{j1}, \ldots, v_{jt_j})$ 
for $j = 1, 2, \ldots, s$, and 
\begin{align*}
\textstyle
\big(\sEnd_\sD(\sV)\big)_0\, &= \, 
\big\{f\in \sEnd_\sD(\sV) \mid f(\sV_\ep)
\subseteq \sV_\ep\text{ for all }\ep \in \Ga_\sV\big\} \\ 
&\cong \ 
\textstyle\prod\limits_{j=1}^s\End_{\sD_0}(\sD_0\text{-span}
(v_{j\hsp 1}, \ldots, v_{j\hsp t_j})) \ 
\cong  \ \prod\limits_{j=1}^s M_{t_j}(\sD_0).
\end{align*}
This is a direct product of $s$ simple algebras.  Since
we have assumed that $\sB_0$ is simple, we must have $s=1$,
i.e., all the $v_i$ have degree $\de_1$.  It is then clear that 
when we use the base $(v_1, \ldots , v_\ell)$ for the 
isomorphism $\sEnd_\sD(\sV) \cong M_\ell(\sD)$, the grading
on $M_\ell(\sD)$ induced by the isomorphism is the standard
grading.  Thus, $\sB \cong_g M_\ell(\sD)$ and hence
$\sB_0\cong M_\ell(\sD_0)$ and $\Ga_\sB = \Ga_\sD$.  
Then, $\Ga'_\sB = \Ga'_{M_\ell(\sD)} = \Ga_\sD$ and,
when we identify $Z(\sB_0)$ with $Z(M_\ell(\sD_0))$
and with $Z(\sD_0)$, clearly $\Theta_\sB = \Theta_{M_\ell(\sD)}
= \Theta_\sD$.

(ii) If $\sB$ is a graded division algebra, then every nonzero
homogeneous element of $\sB$ lies in $\sB^*$. In particular,
$\sB_0\setminus\{0\} \subseteq \sB^*$, so $\sB_0$ is a division
 ring. Conversely, suppose $\sB_0$ is a division ring.
Since $\sB_0$ is then simple,  part (i) applies,  showing
 that for some graded division algebra $\sD$, we have 
 $\sB\cong_g M_\ell(\sD)$ where $\sB_0 \cong M_\ell(\sD_0)$.
 Necessarily $\ell =1$, as $\sB_0$ is a division ring.
\end{proof}

\begin{corollary}\label{ItensorE}
Let $\sI$ and $\sE$ be central graded division algebras over
a graded field $\sT$, with $\sI$ inertial, and let~$\sD$~be 
the graded division algebra  with $\sD \sim_g \sI \otimes_\sT
\sE$.  Then, $\sD_0 \sim \sI_0\otimes _\sTz \sEz$, $Z(\sD_0)
\cong Z(\sEz)$, $\Ga_\sD = \Ga_\sE$, and $\Theta_\sD = 
\Theta_\sE$.
\end{corollary}

\begin{proof}
Let $\sB = \sI\otimes_\sT \sE$.  Since $\sI \cong_g \sI_0
\otimes _\sTz \!\sT$, we have $\sB \cong_g \sI_0 \otimes
_\sTz \sE$.  Hence, 
${\sB_0\cong \sI_0 \otimes _\sTz \sEz}$,\break 
${Z(\sB_0) \cong Z(\sI_0) \otimes _\sTz Z(\sEz) \cong Z(\sEz)}$, 
and $\Ga_\sB = \Ga_\sE$.  Moreover, $\sB_0$ is simple
as $\sI_0$ is central 
simple over~$\sTz$,  so 
Lemma~\ref{zerosimple} applies to $\sB$.  In particular, 
$\Ga'_\sB = \Ga_\sB$ and $\Theta_\sB = \Theta_\sE$.
Since $\sD$ is the graded division algebra with 
$\sD \sim_g \sB$, the Lemma yields $\sD_0 \sim \sB_0 \cong 
\sI_0\otimes_\sTz \sEz$, so $Z(\sD_0) \cong Z(\sB_0) \cong
Z(\sEz)$, and  $\Ga_\sD = \Ga_\sB = \Ga_\sE$, and 
$\Theta_\sD = \Theta_\sB = \Theta_\sE$.  
\end{proof}


\section{Abelian crossed products and nonunitary $\SK$ 
for semiramified algebras}
\label{abcp}

Let $M$ be a finite degree abelian Galois extension of a field $K$, and let 
$H = \Gal(M/K)$.  Let ${X(M/K) = \text{Hom}(H, \qq/\zz)}$, the character group 
of $H$.  Take any cyclic decomposition $H = \langle
\sigma_1 \rangle \times \ldots \times \langle \sigma_k \rangle$, 
and let $r_i$ be the order of $\sigma_i$ in $H$. Let $(\chi_1 , \ldots, \chi_k)$
be the base of $X(M/K)$ dual to $(\si_1, \ldots, \si_k)$; so $\chi_i(\si_j) = 
\delta_{ij}/r_i +\zz$, where $\de_{ij} = 1$ if $j = i$ and $= 0$ if $j\ne i$. 
Let $L_i$ be the fixed field of $\ker(\chi_i)$.  So, 
$M = L_1 \otimes_K \ldots \otimes_K L_k$, and for each $i$, $L_i$ is cyclic 
Galois over $K$ with $[L_i\col K] = r_i$ and $ \Gal(L_i/K) = 
\langle \sigma_i|_K\rangle$.  Let $A$ be any central simple $K$-algebra 
containing $M$ as a strictly maximal subfield (i.e., $M$ is a 
maximal subfield of $A$  with $  [M\col K] = \deg(A)$).  By the Double Centralizer
Theorem, the centralizer $C_A(M)$ is $M$.
Recall that
every algebra class in $\Br(M/K)$ is represented by a unique 
such $A$.  By 
Skolem-Noether, for each $i$ there is $z_i \in A^*$ with 
$\intt(z_i)|_M = \si_i$, where 
$\intt(z_i)$ denotes conjugation by $z_i$.  Set
$$
u_{ij}  \, =  \,z_iz_jz_i^{-1}z_j^{-1}\ \ \text{and} \ \ b_i  \, = \,  z_i^{r_i}
.
$$
Since $\intt(u_{ij})|_M = \si_i\si_j\si_i\inv\si_j\inv = \id_M$ and 
$\intt(b_i)|_M = \si^{r_i} = \id_M$,  all the $u_{ij}$ and $b_i$ 
lie in $C_A(M)^* = M^*$.
Take the index set $\mathcal I=\prod_{i=1}^k \{0,1,2,\dots,r_i-1\}
\subseteq \zz^k$.  For $\bold i = (i_1, \ldots , i_k)\in \mathcal I$, 
set $\si^{\bold i} = \si_1^{i_1} \ldots \si_k^{i_k}$ and 
$z^{\bold i} = z_1^{i_1}\ldots z_k^{i_k}$.  So, 
$\intt(z^{\bold i})|_M = \si^{\bold i}$ and, as the map
$\bold i \mapsto \si^{\bold i}$ is a bijection 
$\mathcal I \to H$, we 
have the crossed product decomposition
$$
A \ =  \ \textstyle \bigoplus\limits_{\bold i\, \in \,  \mathcal I} \, 
Mz^{\bold i}. 
$$
For $\bold i , \bold j \in \mathcal I$, if we set 
$\bold i 
\bold *
\bold j$ to be the element of $\mathcal I$ congruent to 
$\bold i + \bold j$ mod $r_1\zz \times \ldots \times r_k\zz$ in 
$\zz^k$, and set  
$$
f(\si ^{\bold i}, \si ^{\bold j})  \ = \ z^{\bold i} z^{\bold j}
 (z^{ \bold i \bold * \bold j})\inv \ \in \ M^*,
$$
then $f\in Z^2(H, M^*)$ and the multiplication in $A$ is given by 
$$
a  z^{\bold i} \,\cdot \, c z^{\bold j} \ = 
\ a\si^{\bold i}(c)           
f(\si ^{\bold i}, \si ^{\bold j})\, z^{ \bold i \bold *
\bold j}, \ \ \text{ for all}\ \ a,c \in M \ \text { and }
\bold i, \bold j
\in \mathcal I.
$$
Since each $f(\si ^{\bold i}, \si ^{\bold j})$ is expressible as a computable 
product of the $u_{ij}$ and the $b_i$ and their images under~$H$, 
the multiplication for $A$ is completely determined by $M$, $H$, and the 
$u_{ij}$ and $b_i$.  Thus, we write 
${A = A(M/K, \bsi, \bu,\bb)}$, where $\bsi = (\si_1, \ldots, \si_k)$, 
$\bu=  (u_{ij})_{i=1, \,j=1}^{k\ \ \ \,k}$, and $\bb = 
(b_1, \ldots, b_k)$.

It is easy to check (cf.~ \cite[Lemma~1.2]{as} or \cite [p.~423]{tignol})   
that the $u_{ij}$ and the $b_i$ satisfy the 
following relations, for all $i,j,\ell$,
\begin{equation}\label{urels}
u_{ii}  \, = \,  1,  \ \ u_{ji}  \, = \,  u_{ij}\inv,  \ \ 
\si_i(u_{j\ell}) \si_j(u_{\ell i}) \si_{\ell}(u_{ij})  \ = \ 
u_{j\ell} u_{\ell i} u_{ij}  
\end{equation}
and 
\begin{equation}\label{brel}
N_{M/M^{\langle \si_i\rangle}}(u_{ij})  \ = \  b_i/\si_j(b_i),
\end{equation}
where $M^{\langle \si_i\rangle}$ is the fixed field of 
$M$ under $\langle \si_i\rangle$.
It is known (cf.~\cite[Th.~1.3]{as})  that  
for any family of 
$u_{ij}$ and $b_i$ in $M^*$ satisfying \eqref{urels} 
and \eqref{brel} there is 
a central simple $K$-algebra $A(M/K, \bsi, \bu, \bb)$.

\begin{lemma}\label{onen}
Let $A=A(M/K,\bsi,\bu,\bb)$ as above, and let 
$B=A(M/K,\bsi,\bf v, \bf c)$. 
Then, there is a well-defined abelian crossed product
$A(M/K,\bsi,\bf w, \bf d)$ where 
$w_{ij}=u_{ij}v_{ij}$ and $d_i=b_ic_i$ for all $i,j$.
Moreover,
$A\otimes_K B \sim A(M/K,\bsi,\bf w, \bf d)$ 
$($Brauer equivalent$)$.
\end{lemma} 

\begin{proof}
Because the $u_{ij}$ and $b_i$ satisfy \eqref{urels} and \eqref{brel} as do the $v_{ij}$
and $c_i$, and the $\si_i$ and the norm maps are multiplicative, 
the $w_{ij}$ and $d_i$ also satisfy \eqref{urels} and \eqref{brel}.
Therefore $A(M/K,\bsi,\bf w, \bf d)$ 
is a well-defined abelian crossed product.

We have the  $2$-cycle $ f\in Z^2(H,M^*)$ representing $A$ 
defined as above by,
 $f(\si ^{\bold i}, \si ^{\bold j})  \ = 
\ z^{\bold i} z^{\bold j} (z^{ 
\bold i \bold *  \bold j})\inv$.
The relations $z_i^{r_i}=b_i$ and $[z_i,z_j]=u_{ij}$ are  
encoded in $f$ by
\begin{equation}\label{cocycle}
f(\sigma_i^{\ell},\sigma_i)=
\begin{cases}
1, &\text{if $0\leq \ell \leq r_i-2$}\\
b_i, &\text{if $\ell=r_i-1$}
 \end{cases}
 \text{\,\,\,\,\,\,\,\, and \,\,\,\,\,      } 
f(\sigma_i,\sigma_j)=
\begin{cases}1 &\text{if }i< j\\
u_{ij} &\text{if $i>j$}.  
\end{cases}
\end{equation} 
We likewise build a cocycle $g\in Z^2(H,M^*)$ for
$B=A(M/K,\bsi,\bf v,\bf c)$.  Then, the cocycle $f\ccdot g$ satisfies 
conditions corresponding to those for $f$ in \eqref{cocycle},  so 
 $f\ccdot g$ is 
a cocycle
for ${C=A(M/K,\bsi, \bf w, \bf d)}$ where ${w_{ij}=u_{ij}v_{ij}}$ 
and
$d_i=b_ic_i$. From the group isomorphism 
${H^2(H,M^*)\cong\Br(M/K)}$
it follows that ${A\otimes_K B \sim C}$.
\end{proof}

In Tignol's terminology  in \cite{tignol}, a central simple 
$K$-algebra containing $M$ as a strictly maximal subfield 
{\it decomposes according to $M$} if 
$A\cong (L_1/K, \si_1, b_1)\otimes_K \ldots \otimes _K 
(L_k/K, \si_k, b_k)$ for some $b_1, \ldots, b_k \in K^*$.
Clearly then, $A \cong A(M/K,\bsi,\bold 1,\bb)$, i.e., each 
$u_{ij}= 1$.  Conversely, for any algebra 
$A(M/K,\bsi,\bold 1,\bb)$ 
 (i.e., the $z_i$ commute with each 
other),  each $z_j$ centralizes 
$b_i = z_i^{r_i}$, so $b_i\in M^H = K$ and the algebra 
decomposes 
according to $M$.  The collection of such algebras yields an 
important 
distinguished subgroup $\Dec(M/K)$ of $\Br(M/K)$, i.e.
 \begin{align}\label{decdef}
\begin{split}
\Dec(M/K) \ &= \ \{\,[A] \in \Br(M/K)\,|\ A\text{ decomposes according to 
$M$}\}\\  &= \ 
\{\,[  A(M/K,\bsi,\bu,\bb)] \ |\ 
\text{every }u_{ij} = 1 \text{ and every } b_i\in K^*\,\}.
\end{split}
\end{align} 
Since $\Br(L_i/K) = 
\{ \,[(L_i/K,\si_i, b)] \, |\ b\in K^*\,\}$, 
we have also $\Dec(M/K) = \prod_{i = 1}^k\Br(L_i/K)
\subseteq\Br(M/K)
$.  Tignol 
also 
also points out in \cite[p.~426]{tignol} a homological 
characterization:  
From the short exact sequence 
of trivial $H$-modules $0 \to \zz \to \qq \to \qq/\zz \to 0$
the long exact cohomology sequence yields the connecting 
homomorphism 
$\delta\colon H^1(H, \qq/\zz) \to H^2(H, \zz)$, which is an 
isomorphism since 
$H^i(H, \qq) = 1$ for $i\ge 1$ as $\qq$ is uniquely divisible.  For any 
$\chi\in X(M/K) = H^1(H, \qq/\zz)$ and any $c\in K^* = 
H^0(H, M^*)$
it is known (cf. ~\cite [p.~204, Prop.~2]{serre}) that under the cup 
product pairing 
${\cupp\colon H^2(H,\zz) \times H^0(H, M^*) \to H^2(H,M^*) = 
\Br(M/K)}$,\break
we have ${\delta(\chi)\cupp c = [(N/K,\rho|_N,c)]}$, where $N$ 
is the fixed field of 
$\ker(\chi)$ and $\rho \in H$ is determined by\break ${\chi(\rho) = 
(1/|\chi|) +\zz\in \qq/\zz}$. Thus, the algebra class
$[(L_1/K, \si_1, b_1)\otimes_K \ldots \otimes _K 
(L_k/K, \si_k, b_k)]$ in $\Br(M/K)$ corresponds to 
${(\delta(\chi_1) \cupp b_1) +}$ $\ldots +(\delta(\chi_k) 
\cupp b_k)$ in 
$H^2(H, M^*)$. Since the cup product is bimultiplicative and 
$X(M/K) = \langle \chi_1, \ldots , \chi_k\rangle$, we have
\begin{equation}\label{dec} 
\Dec(M/K) \ = \ \big\langle \im\big(\cupp\colon H^2(H, \zz) \times 
H^0(H, M^*)
\to H^2(H, M^*)\,\big)\big\rangle \ = \ \textstyle \prod\limits_
{\substack{K \subseteq L \subseteq M \\ \Gal(L/K) 
\text{ cyclic}}} 
\Br(L/K),
\end{equation} 
showing that $\Dec(M/K)$ is independent of the choice of the 
$\si_i$ 
and the $L_i$. (Actually, Tignol uses \eqref{dec} as his definition
of $\Dec(M/F)$, and proves in \cite[Cor.~1.4]{tignol} that this is 
equivalent to the definition given here in \eqref{decdef}.)

The case when  $H$ is  bicyclic is of particular interest, i.e.,
 $H=\langle\sigma_1\rangle\times
\langle\sigma_2\rangle$ and $M=L_1\otimes_K L_2$. Then, for 
any
algebra $A=A(M/K,\bsi, \bu,\bb)$, if we set $u=u_{12}$, 
then
$u$ determines all the $u_{ij}$ as $u_{21}=u_{12}^{-1}$ and
$u_{11}=u_{22}=1$. We write, for short, $A=A(u,b_1,b_2)$. 
The conditions in \eqref{brel} can then be restated:
\begin{equation}\label{bicyclicbrel}
b_1\in M^{\lan\si_1\ran}\,  =\,  L_2, 
\quad b_2\in M^{\lan\si_2\ran} \, =\,  L_1,
\quad N_{M/L_2}(u) \, =\,  b_1/\si_2(b_1), 
\quad N_{M/L_1}(u) \, =\,  \si_1(b_2)/b_2.
\end{equation}
Note that   $N_{M/K}(u) = N_{L_2/K}(b_1/\si_2(b_1)) =1$.
An 
easy
calculation~(cf.~\cite[Th.~1.4]{as}) shows that
\begin{equation} \label{ghgt}
\begin{split}
A(u,b_1,b_2)\, \cong \,  A(u',b_1',b_2') \text{ if and only if there exist 
$c_1,c_2 \in M^*$ such that}\qquad\qquad\qquad\qquad\\
u'\, =\, \big[c_1/\sigma_2(c_1)\big]\big[\sigma_1(c_2)/c_2\big]u, 
\quad
b_1'\, =\, N_{M/L_2}(c_1)b_1, \quad \text{and \, } 
b_2'\, =\, N_{M/L_1}(c_2)b_2. \  \ \quad\qquad\qquad
\end{split}
\end{equation}

These observations can be formulated homologically:
Recall that $\widehat{H}^{-1}(H,M^*)=\ker(N_{M/K})/I_H(M^*)$, 
where
$\ker(N_{M/K})=\{m\in M^* \mid N_{M/K}(m)=1 \}$ and, as
$H = \lan \si_1 \ran \times \lan \si_2\ran$,
$$
I_H(M^*) \ = \ \big\{\,[a/\sigma_1(a)]\,[b/\sigma_2(b)]\mid 
a,b \in M^* \big\}.
$$
 We define a map
\begin{equation}
\eta\colon\Br(M/K) \longrightarrow \widehat{H}^{-1}(H,M^*) 
\text{ \ \ given by  \ \ } \big[A(u,b_1,b_2)\big]\mapsto uI_H(M^*).
\end{equation}
By~(\ref{ghgt}) above $\eta$ is well-defined, and 
Lemma~\ref{onen}
shows that $\eta$ is a group homomorphism.  
Given any $u\in M^*$ with $N_{M/K}(u) = 1$, Hilbert 90 
gives $b_1 \in L_2^*$ and $b_2\in L_1^*$ so that the
conditions in \eqref{bicyclicbrel} are satisfied and the 
algebra $A(u,b_1,b_2)$ exists.  Therefore $\eta$
is surjective.
By~(\ref{ghgt}),
$$
\ker(\eta) \ = \ \big\{[A(u,b_1,b_2)] \mid u=1 \big \} 
 \ = \ \Dec(M/K),
$$ 
so $\eta$ yields an isomorphism
\begin{equation}\label{njnj}
\Br(M/K)\big / \Dec(M/K)  \ \cong \  \widehat{H}^{-1}(\Gal(M/K),M^*)
\qquad\text{whenever $M$ is bicyclic over $K$}.
\end{equation}
This isomorphism is known (see, e.g., \cite[Remarque, pp.~
427--428]{tignol}); indeed, it follows by comparing  
Draxl's formula \cite[Kor.~8, p.~133]{draxlSK} for $\SK$
of the division algebras considered by Platonov 
in \cite{platonov} with Platonov's formula
in \cite[Th.~4.11, Th.~4.17]{platonov}.  I learned
of this description of the isomorphism from Tignol. 
 Its relevance for $\SK$ calculations is shown 
in the next proposition, which is the graded
version of \eqref{DSRSK1} and \eqref{platformula}
above.

\begin{proposition}\label{NSRSK}
Suppose $\sN$ is a DSR central graded division algebra
over the graded field $\sT$.
Then, 
\begin{enumerate}[\upshape (i)]
\item
$\SK(\sN) \ \cong \ \wh H\inv(H,\sN_0^*)\ \text{ where }
H = \Gal(\sN_0/\sT_0)$.
\item  If 
$\sN_0 \cong L_1 \otimes_{\sT_0} L_2$ with each 
$L_i$ cyclic Galois over $\sT_0$, then
$$
\SK(\sN)  \ \cong \  \Br(\sN_0/\sT_0) \big / 
\Dec(\sN_0/\sT_0).$$
\end{enumerate}
\end{proposition}

\begin{proof}
(i) was given in \cite[Cor.~3.6(iv)]{hazwadsworth}, and (ii) follows from 
(i) and \eqref{njnj} above.
\end{proof}

We will generalize  Prop.~\ref{NSRSK}  in  Th.~\ref{goth} below 
by giving  formulas for $\SK(\sE)$    when
$\sE$ is semiramified but not necessarily DSR.  For this we need,
first, a graded version of the abelian crossed products 
described at the beginning of this section. 
Second, we need a graded version of the $I\otimes N$
decomposition for semiramified division algebras over a 
Henselian valued field. Here $I$ is inertial 
and $N$ is DSR. (See~\cite[Lemma~5.14, Th.~5.15]{jw} for the valued 
$I\otimes N$ decomposition.)

Here is the graded version of abelian crossed products. 
Let $\sB$ be a 
central simple graded algebra over a graded field $\sT$. 
Assume that $\sB$ contains a maximal graded subfield 
$\sS$ with ${[\sS \col\sT] = \deg(\sB) \ (=\sqrt{[\sB\col\sT]}\, )}$
such that $\sS$ is Galois over $\sT$ and $H = \sGal(\sS/\sT)$
is abelian.  We have $C_\sB(\sS) = \sS$ by the graded Double Centralizer
Theorem.  For any cyclic decomposition $H= \lan \si_1\ran \times 
\ldots \times \lan \si_k\ran$, the graded Skolem-Noether Theorem,
Prop.~\ref{grSN}, is available  
as $C_\sB(\sS) = \sS = Z(\sS) $;  it shows that for each 
$i$ there is $y_i\in \sB^*$ with 
$y_i$~homogeneous and $\intt(y_i)|_\sS = \si_i$.  Set 
$c_i = y_i^{r_i}$ where $r_i$ is the order of $\si_i$ in $H$, and 
set $v_{ij} = y_iy_jy_i\inv y_j\inv$.    Then, each 
$c_i \in C_\sB(\sS)^* = \sS^*$
with $\deg(y_i) = \frac 1{r_i}\deg(c_i)$, 
and each $v_{ij} \in \sS_0^*$.  For each ${\bold i = (i_1, \ldots, i_k)\in
\mathcal I=\prod_{j=1}^k
\{0,1,2,\dots,r_j-1\}}$, set $y^{\bold i} = y_1^{i_1}\ldots y_k^{i_k}$.
Then, $\intt(y^{\bold i})|_\sS = \si^{\bold i}$, and we have
\begin{equation}\label{grabcrprod}
\sB \ = \ \textstyle\bigoplus\limits_{\bold i \in \mathcal I}
 \, \sS\, y^{\bold i}.  
\end{equation} 
For, the sum in the equation is direct since $\sB \otimes _\sT
q(\sT)  = \bigoplus_{ \bold i \in \mathcal I}
(\sS\otimes _\sT q(\sT))\, y^\bold i$ by the ungraded case.
Then equality holds in \eqref{grabcrprod} by dimension count.
Note that $\sB$ is a left graded  $\sS$-vector space with homogeneous
base $\big(y^{\bold i}\big)_{\bold i \in \mathcal I}$, and 
\begin{equation}
\deg(y^{\bold i}) \ =\ \textstyle \sum\limits _{j= 1}^k\frac {i_j}{r_j} \deg(c_j).
\end{equation}
So,
\begin{equation}\label{gradeinfo}
\textstyle \Gamma_\sB \, = \, \big \lan \frac1{r_1}\deg(c_1), \ldots, 
\frac1{r_k}\deg(c_k)\big\ran + \Ga_\sS \qquad \text{and}\qquad
\text{each}\ \ B_\delta \, = \,  \bigoplus\limits_{\bold i \in \mathcal I}\,
S_ {(\de - \deg(y^{\bold i}))}y^{\bold i}.
\end{equation}
Since $\sB$ is determined as a graded $\sT$-algebra by $\sS$,
the~$\si_i$, the~$v_{ij}$, and the $c_i$, we write $\sB = 
\sA(\sS/\sT, \bsi, \bold v, \bold c)$, where $\bsi = (\si_1, \ldots,
\si_k)$, $\bold v = (v_{ij})_{i = 1,j=1}^{k \ \ \ k}$, and 
${\bold c = (c_1, \ldots, c_k)}$.
Note that the $v_{ij}$ and the $c_i$ satisfy the identities 
corresponding  to \eqref{urels} and \eqref{brel}.  Conversely, given any 
$v_{ij} \in \sS_0^*$ and $c_i \in \sS^*$ satisfying those identities
there is a central simple graded $\sT$ algebra 
$\sA(\sS/\sT, \bsi, \bold v,\bold c)$. This is obtainable as 
$\sB = \bigoplus _{\bold i \in \mathcal I}\sS\,y^{\bold i}$  within
the ungraded abelian crossed product $A = A(q(\sS)/q(\sT), \bsi, 
\bold v, \bold c)$, with the grading on $\sB$ determined by 
that on $\sS$ and  $\deg(y_i) = \frac 1{r_i}\deg(c_i)$, as described above.
To see that $\sB$ is a graded ring, one uses that each $\si\in H$ is a 
(degree-preserving) graded automorphism of 
$\sS$ and that $\deg(y^{\bold i}\cdot y^{\bold j}) = \deg(y^{\bold i})
+\deg(y^{\bold j})$ for all $\bold i, \bold j\in \mathcal I$, since all 
the $v_{ij}$ have degree $0$.   
This $\sB$ is graded simple, since any nontrivial proper homogeneous ideal 
would localize to a nontrivial proper ideal of the simple $q(\sT)$-algebra $A$.

\begin{remark}\label{abeliancpprod}
The graded analogue to Lemma~\ref{onen} holds, with the same
proof, since for $\sS$~Galois over $\sT$, we have 
$\sBr(\sS/\sT) \cong H^2(\sGal(\sS/\sT), \sS^*)$.
\end{remark}

The graded abelian crossed products we work with here will 
have $\sS$ inertial over $\sT$ and will be semiramified, 
as described in the next lemma.

\begin{lemma} \label{incp} 
Let $\sS$ be an inertial graded field extension of $\sT$ with 
$\sS$ abelian Galois over $\sT$.  Let\break ${H = \sGal(\sS/\sT)=
 \lan \si_1\ran \times \ldots \times \lan \si_k\ran}$ as above 
with $r_i$ the order of $\si_i$, and let 
$\sB = \sA(\sS/\sT, \bsi, \bold v, \bold c)$ be a graded abelian
crossed product.  Let $\delta_i = \frac 1{r_i}\deg(c_i) \in \Gamma_\sB$
and $\ov {\delta_i} = \delta_i + \Gamma_\sT \in \Gamma_\sB/ \Gamma_\sT$.
Then, $\sB$ is a semiramified graded division algebra 
if and only if each $\ov{\delta_i}$ has order $r_i$ and $\ov{\delta_1}, 
\ldots, \ov{\delta_k}$ are independent in $\Gamma_\sB /\Gamma_\sT$.
When this occurs, $\sB_0= \sS_0$ and $\Gamma_\sB/\Gamma_\sT
= \lan\ov{\de_1}\ran\times \ldots \times \lan\ov{\de_k}\ran
\cong H$.   
\end{lemma}

\begin{proof}
Since $\sS$ is inertial and Galois over $\sT$, $\sS_0$ is Galois over 
$\sT_0$ with $\Gal(\sS_0/\sT_0) \cong \sGal(\sS/\sT) = H$.  We identify
$H$ with $\Gal(\sS_0/\sT_0)$.
We have $\sS_0\subseteq \sB_0$ and $[\sS_0\col \sT_0] = 
[\sS\col\sT] = \deg(\sB)$. 

Suppose $\sB$ is a semiramified graded division algebra.
Then, ${[\sB_0\col \sT_0] = \deg(\sB) = [\sS_0\col\sT_0]}$,
so  ${\sB_0 = \sS_0}$. 
   Since $\sB$ is semiramified, 
the epimorphism $\ov\Theta_\sB\colon \Gamma_\sB/\sT \to H$ is an 
isomorphism, as noted in \S2.
When we represent ${\sB = \bigoplus_{\bold i \in \mathcal I}\sS y^{\bold i}}$
as above, since $\intt(y_i) = \si_i$ and 
$\deg(y_i) = \frac 1{r_i}\deg(c_i) = \delta_i$, 
we have $\ov \Theta_\sB(\ov{\delta _i}) = \si_i$. 
Hence, $\ov{\delta_i}$ has the same order $r_i$ as $\si_i$, and
$$
 \Gamma_\sB/\Gamma_\sT \ = \ \ov \Theta_\sB^{\, -1} (H) 
\ = \ \ov \Theta_\sB^{\, -1}(\lan\si_1\ran) \times\ldots \times 
\ov \Theta_\sB^{\, -1}(\lan\si_k\ran)\ = \ \lan\ov{\de_i}\ran
\times \ldots \times \lan \ov{\de_k}\ran,
$$
so the $\ov{\de_i}$ are independent in $\Gamma_\sB/\Ga_\sT$.  

Conversely, suppose each $\ov{\de_i}$ has order $r_i$ and 
the $\ov{\de_i}$ are independent in $\Ga_\sB/\Ga_\sT$.  Then,
$$
|\Ga_\sB\col\Ga_\sT|\!\hsp  \ \ge \  \textstyle\prod\limits_{i = 1} ^ k
|\lan\ov{\de_i}\ran| \ = \ r_1\ldots r_k  \ = \ |H| \ = \  \deg(\sB).
$$
Hence,
\begin{equation}\label{B0}
[\sB_0\col \sT_0] \ = \ [\sB\col\sT] \big/ |\Ga_\sB\col\Ga_\sT| \ 
\le \,\hsp \deg(\sB)^2/\deg(\sB) \, = \ [\sS_0\col\sT_0].
\end{equation}
Since $\sS_0 \subseteq \sB_0$, \eqref{B0} shows that $\sB_0 = \sS_0$, 
so equality holds in \eqref{B0}.  Since $\sB_0$ is a field, $\sB$
is a graded division algebra by Lemma~\ref{zerosimple}(ii), and it is 
semiramified by the equality in \eqref{B0}.
\end{proof}

Observe  that if $\sE$ is any semiramified graded $\sT$-central 
division algebra, then $\sE$ is a graded abelian crossed product
as described in Lemma~\ref{incp}.
For,  $\sEz\sT$ is a  maximal graded subfield 
of $\sE$ which  is inertial and Galois over $\sT$ with $\sGal(\sEz\sT/\sT)
\cong \Gal(\sEz/\sTz)$, which is abelian.

\bigskip

\begin{proposition}\label{INdecomp}
Let $\sE$ be a semiramified central graded division
algebra over the graded field $\sT$.  Then,
\begin{enumerate}[\upshape (i)]
\item
There exist graded $\sT$-central division algebras 
$\sI$ and $\sN$ such that 
$\sI$ is inertial, $\sN$ is DSR, and 
$\sE \sim_g \sI\otimes_\sT \sN$ in $\sBr(\sT)$.  When this occurs,
$\sN_0\cong \sE_0$, $\Gamma_N = \Gamma_E$, $\Theta_\sN
= \Theta_\sE$, and $\sE_0$ splits $\sI_0$.
\item For any other decomposition 
$\sE \sim_g \sI'\otimes_\sT \sN'$ with $\sI'$ inertial
and $\sN'$ DSR, we have $\sI_0' \equiv \sI_0\ 
(\modd\ 
\Dec(\sE_0/\sT_0)\,)$.  
\end{enumerate}
\end{proposition}

We do not give a proof of Prop.~\ref{INdecomp} because it is a simpler
version of the proof of the analogous unitary result, which is 
Prop.~\ref{uINdecomp} below.  Also, Prop.~\ref{INdecomp} is  the 
graded analogue of a known result for semiramified division algebras
over Henselian valued fields, \cite[Lemma~5.14, Th.~5.15]{jw}, and the graded result
given here is deducible from the Henselian one.

\begin{lemma}\label{morit}
For the semiramified graded division algebra
$\sE=\sA(\sE_0\sT/\sT,\bsi,\bf v,\bf c)$ as above, write\break 
${\sE
\sim_{g} \sI \otimes_T \sN}$ with $\sI$ inertial and $\sN$ 
DSR; so
$[\sI_0] \in \Br(\sE_0/\sT_0)$. If $\sI_0 \sim
A(\sE_0/\sT_0,\bsi,\bu,\bb)$, then 
by changing the chioce of 
the $y_i\in \sE^*$ inducing $\sigma_i$ on $\sE_0\sT$ 
we have $\sE = \sA(\sE_0\sT/\sT,\bsi,\bf u,\bf e)$
with the same $\bf u$ as for~$\sI_0$.
\end{lemma}

\begin{proof}
Let $\sJ$ be a maximal graded subfield of $\sN$ which is 
totally
ramified over $\sT$, so $\Gamma_\sN=\Gamma_\sJ$. Because 
$\sN$~is
semiramified, the map $\Theta_\sN\colon\Gamma_\sN/\Gamma_\sT
\rightarrow \Gal(\sN_0/\sT_0)$ is an isomorphism. But also
$\sN_0=\sE_0$. Thus, for each $i$, we can choose $x_i \in 
\sJ^*$
with $\Theta_\sN(\deg(x_i))=\sigma_i|_{\sE_0}$. 
Let
$d_i=x_i^{r_i} \in (\sN_0\sT)^*=(\sE_0\sT)^*$. Then,
$\sN\cong_g\sA(\sE_0\sT/\sT,\bsi,\bold w,\bold d)$, where 
each
$w_{ij} =x_ix_jx_i\inv x_j\inv=1$, as all the $x_i$ lie in the graded 
field
$\sJ$. Let $\sI_0'=A(\sE_0/\sT_0,\bsi,\bu,\bold b)$, 
which is
Brauer equivalent to $\sI_0$. Then set
$\sI'=\sI_0'\otimes_{\sT_0}\sT$, which is an inertial 
$\sT$-algebra
with $\sI' \sim_g \sI$. Since $\sI' \otimes_\sT \sN 
\sim_g
\sI\otimes_\sT \sN\sim_g \sE$, we may without any loss 
replace
$\sI$ by $\sI'$. Then, as ${\sI_0\cong
A(\sE_0/\sT_0,\bsi,\bu,\bold b)}$, clearly  
$\sI\cong_g\sI_0\otimes_{\sT_0}\sT
\cong_g\sA(\sE_0\sT/\sT,\bsi,\bu,\bold b)$. 
Let $\sE' = \sA(\sE_0\sT/\sT,\bsi,\bold u,\bold e)$, 
where each $e_i = b_id_i$, and let $y_1', \ldots ,
y_k'$ be the associated generators of $\sE'$ over 
$\sEz\sT$. Then,
$\sE\sim_g\sI \otimes_\sT \sN \sim_g\sE'$, 
by Remark~\ref{abeliancpprod}, as $u_{ij}w_{ij} = u_{ij}$.
 Note that for each~$i$, ${\deg(e_i) = \deg(d_i)}$, as
$\deg(b_i) = 0$. Hence, $\Ga_{\sE'} = \Ga_\sN$ by 
\eqref{gradeinfo}.  Furthermore, $\sE'$ is a semiramified 
graded division algebra since $\sN$ is, because Lemma~\ref{incp} 
shows that this is determined by the $\deg(e_i)$, resp.~
$\deg(d_i)$.  Because  $\sE'$ is a graded division algebra (not just a 
graded simple
algebra), as is $\sE$, from $\sE\sim_g\sE'$  
the
uniqueness in the graded Wedderburn 
Theorem~\cite[Prop.~1.3]{hwcor} yields
a graded $\sT$-isomorphism $\eta\colon 
\sE \to \sE'$.  By the graded Skolem-Noether Theorem, 
Prop.~\ref{grSN},  $\eta$ can be chosen so that $\eta|_{\sEz\sT}
=\id$.  Then replacing the $y_i$ by $\eta\inv(y_i')$ in the 
presentation of $\sE$ changes each~$v_{ij}$ to $u_{ij}$.
\end{proof}

\begin{theorem}\label{goth}
Suppose $\sE$ is a semiramified $\sT$-central graded division
algebra, and take any decomposition 
${\sE\sim_g\sI\otimes_\sT
\sN}$ where $\sI$ is an inertial graded $\sT$-algebra and 
$\sN$ is DSR. Then,
\begin{enumerate}[\upshape (i)]
\item  Since $\sI_0  \in \Br(\sE_0/\sT_0)$ with 
$\sEz$ abelian Galois over $\sTz$,
 we can write
$\sI_0 \sim 
A(\sE_0/\sT_0,\bsi,\bu,\bold b)$ in $\Br(\sTz)$.
 Then,
$$
\SK(\sE) \ \cong \  \widehat{H}^{-1}(H,\sE_0^*)\big / 
\big \langle \im \{ u_{ij} \mid 1\leq i,j \leq k \} 
\big \rangle, \ \text{ where } H = \Gal(\sE_0/\sT_0).
$$

\item  If  
$\sE_0 \cong L_1 \otimes_{\sT_0} L_2$ with each 
$L_i$ cyclic Galois over $\sT_0$, then
$$
\SK(\sE)  \ \cong \  \Br(\sE_0/\sT_0) \big / 
\big [\Dec(\sE_0/\sT_0)\cdot\langle[\sI_0]\rangle\big ],$$
where $\Dec(\sE_0/\sT_0) = \Br(L_1/\sTz) \cdot \Br(L_2/\sTz)$. 
\end{enumerate}
\end{theorem}

\begin{proof} The definition of $\SK$ for graded division algebras
is given in \eqref{grsk} above. 
(i) We have\break ${H = \Gal(\sEz/\sTz) \cong \sGal(\sEz \sT/\sT)}$.
Since $\sE$ is semiramified,  
$H\cong \Ga_\sE/\Ga_\sT$ via $\ov \Theta_\sE\inv$ (see~\S 2). 
By \cite[Cor.~3.6(ii)]{hazwadsworth}
 there is an exact sequence
\begin{equation}\label{ppr}
0\ \longrightarrow  \ H  \textstyle\wedge H \ 
\stackrel{\Phi}{\longrightarrow} \ \widehat{H}^{-1}(G,\sE_0^*) \ 
\stackrel{\Psi}{\longrightarrow} \ \SK(\sE) \ \longrightarrow \ 0.
\end{equation}
The maps in~(\ref{ppr}) are given as follows: Let
$\ker(\Nrd_\sE)=\{a\in \sE^* \mid \Nrd_\sE(a) =1 \} 
\subseteq
\sE_0^*$, and let $\ker(N_{\sE_0/\sT_0})=\{a \in \sE_0^* 
\mid
N_{\sE_0/\sT_0}(a) =1 \}$. Because $\sE$ is semiramified, 
by~\cite[Remark~2.1(iii), Lemma~2.2]{I}, 
${\ker(\Nrd_\sE)=\ker(N_{\sE_0/\sT_0})}$. For every $\rho \in  
H$, choose any $y_\rho\in \sE^*$ with $\intt(y_\rho)|_{\sE_0}=\rho$. 
The map
$\Phi$ is given by: for $\rho,\pi \in H$, 
$$
\Phi(\rho\wedge \pi) \ = \ y_\rho y_\pi y_\rho\inv y_\pi \inv I_H(\sE_0^*)  \ \in 
\ker(N_{\sE_0/\sT_0})\big /I_H(\sE_0^*) \ = \ 
\widehat{H}^{-1}(H,\sE_0^*).
$$
The map $\Psi$ is given by: for $a\in \ker(N_{\sE_0/\sT_0})$,
$$
\Psi\big(a\,I_H(\sE_0)^*\big) \ = \ 
a\,[\sE^*,\sE^*]  \ \in \ker(\Nrd_\sE)\big / [\sE^*,\sE^*]
 \ = \ \SK(\sE).
$$

 By Lemma~\ref{morit}, we can assume 
$\sE= \sA(\sE_0\sT/\sT,\bsi, \bu, \bf c)$ (with
the same $u_{ij}$ as  for $\!\sI_0$). Since\break ${H \cong
\sGal(\sE_0\sT/\sT)=
\langle\sigma_1\rangle \times \ldots \times
\langle\sigma_k \rangle}$, we have 
$H\wedge H =\langle \sigma_i
\wedge \sigma_j \mid 1\leq i,j \leq k \rangle$. There are
$y_1,\dots,y_k \in \sE^*$, with 
$\intt(y_i)|_{\sE_0}=\sigma_i$ and 
$y_iy_j y_i\inv y_j\inv=u_{ij}$. So we can take $y_{\sigma_i}=y_i, 
1\leq i \leq
k$, yielding for the $\Phi$ in~(\ref{ppr}), 
$\Phi(\sigma_i\wedge
\sigma_j)=u_{ij}I_H(\sE_0^*)\in \widehat{H}^{-1}(H,\sE_0^*)$. 
Thus,
$\im(\Phi)=\langle \im(u_{ij})\mid 1\leq i,j \leq k \rangle$, 
and
part (i) follows from the exact 
sequence~(\ref{ppr}).

(ii) When $\sE_0=L_1\otimes_{\sT_0}L_2$, $H=\Gal(\sE_0/\sT_0)$ 
has
rank $2$, say $H=\langle\sigma_1\rangle \times 
\langle \sigma_2
\rangle$. So, $H\wedge H=\langle \sigma_1\wedge \sigma_2 
\rangle$
and $\im (\Phi)=\langle u_{12}  I_H(\sE_0^*)\ran$. As we 
saw in
discussion of~(\ref{njnj}) above, the isomorphism 
$$
\Br(\sE_0/\sT_0)\big
/\Dec(\sE_0/\sT_0) \,\longrightarrow \, 
\widehat{H}^{-1}(H,\sE_0^*),
$$
maps $[\,\sI_0]+\Dec(\sE_0/\sT_0)$ to $u_{12}I_H(\sE_0^*)$. 
Thus using part~(i),
\begin{equation*}
\SK(\sE) \ \cong \  \widehat{H}^{-1}(H,\sE_0^*) \big / \big
\langle \im(u_{12}) \big \rangle  \ \cong \   \Br(\sE_0/\sT_0) 
\big / \big
[\Dec(\sE_0/\sT_0)\cdot\langle[\,\sI_0]\rangle\big ].\qedhere
\end{equation*}
\end{proof}

For any division algebra $D$ over a Henselian valued field 
$F$, the valuation on $F$ extends uniquely to a valuation 
on $D$, and we write $\ov D$ for its  residue division algebra
and $\Gamma_D$ for its value group.
Recall the isomorphism $\SK(D) \cong \SK(\gr(D))$ for a tame 
such $D$,
proved in \cite[Th.~4.8]{hazwadsworth}.  By using this isomorphism,
Th.~\ref{goth} yields the following: 

\begin{corollary}\label{henselcor}
 Let $F$ be  field with Henselian 
valuation $v$, and let $D$ be an $F$-central division algebra
which $($with respect to the unique extension of $v$ to 
$D$$)$ is tame and semiramified.  Take any decomposition
$D\sim I\otimes_F N$, where $I$ and $N$ are $F$-central 
division algebras with $I$ inertial and $N$  DSR. 
\begin{enumerate}[\upshape (i)]
\item 
Since $\ov I  \in \Br(\ov D/\ov F)$ with 
$\ov D$ abelian Galois over $\ov F$,
 we can write
$\ov I \sim A(\ov D/\ov F, \bsi, \bu, \bb)$ in 
$\Br(\ov F)$.  Then, 
$$
\SK(D) \ \cong \  \wh{H}^{-1}(H,\ov D^*)\big / 
\big \langle \im \{ u_{ij} \mid 1\leq i,j \leq k \} 
\big \rangle, \  \ \text{ where }  \  H \,=\, \Gal(\ov D/\ov F).
$$

\item   If  
$\ov D \cong L_1 \otimes_{\ov F} L_2$ with each 
$L_i$ cyclic Galois over $\ov F$, then
$$
\SK(D)  \ \cong \  \Br(\ov D/\ov F) \big / 
\big [\Dec(\ov D/\ov F)\cdot\langle[
\ov I]\rangle\big ],$$
where $\Dec(\ov D/\ov F) = \Br(L_1/\ov F) \cdot \Br(L_2/\ov F)$.
\end{enumerate}  
\end{corollary}

\begin{proof}  
(That $D$ is tame and semiramified means  
$[\ov D\col \ov F] = |\Ga_D\col \Ga_F| = \sqrt{[D\col F]}$ and 
$\ov D$ is a field separable over $\ov F$.)
Let $\sT = \gr(F)$, the associated graded ring of $F$  with 
respect to the filtration on it induced by the valuation
(cf.~\cite{hwcor} or \cite{hazwadsworth}).   Since $F$ is a field,  
$\sT$ is a graded field with $\sTz = \ov F$ and ${\Gamma_\sT = \Gamma_F}$.  
Since $v$ is Henselian, it has unique extensions
to valuations on $D$, $I$, and $N$; with respect to these valuations, 
let ${\sE = \gr(D)}$, $\sI = \gr(I)$, and $\sN = \gr(N)$.  These are graded
division rings, with $\sEz = \ov D$, ${\sI_0 = \ov I \sim 
A(\sEz/ \sTz, \bsi, \bu, \bb)}$, and $\sN_0 = \ov N \cong \ov D = \sEz$.
Moreover, as $D$, $I$, and $N$ are each tame over~$F$, 
it follows by \cite[Prop.~4.3]{hwcor} that 
$\sT$ is the 
center of $\sE$, $\sI$, and $\sN$, and 
$[\sE\col \sT] = [D\col F]$, $[\,\sI\col \sT] = [I\col F]$, and 
$[\sN\col \sT] = [N\col F]$.
Since $I$ is inertial over $F$, we have $\sI$ is inertial over $\sT$. 
That $N$ is DSR means (cf.~\cite[p.~149]{jw}, where the term NSR is used)
that $N$ has  maximal subfields $S$ and $J$ with $S$ inertial over~$F$ 
and $J$ totally ramified of radical type over $F$.  Then,
$\gr(S)$ and $\gr(J)$ are  maximal graded subfields of 
$\sN$ with $\gr(S)$ inertial over~$\sT$ and $\gr(J)$ totally
ramified over $\sT$.  So, $\sN$ is DSR.  Similarly, $\sE$ is 
semiramified since $D$~is tame and semiramified.  
Let $\Br_t(F)$ be the tame part of the 
Brauer group $\Br(F)$.  From the isomorphism $\Br_t(F) \cong
\sBr(\sT)$ given by \cite[Th.~5.3]{hwcor}, we obtain 
$\sE \sim_g \sI \otimes_\sT \sN$ from $D\sim I\otimes_FN$. 
Thus, Th.~\ref{goth} applies to $\sE$ with the decomposition
$\sE \sim_g \sI \otimes _\sT \sN$, and the assertions of 
Cor.~\ref{henselcor} follow immediately as $\SK(D)
\cong \SK(\sE)$ by \cite[Th.~4.8]{hazwadsworth}.
\end{proof} 

\begin{example}\label{cyclicex} 
Take any integer $n\geq 2$ and let $K$ be any 
field containing a primitive $n^2$-root of unity
$\omega$. Let $\sT=K[x,x^{-1},y,y^{-1}]$, the 
Laurent polynomial 
ring,
graded as usual by $\mathbb Z \times \mathbb Z$ with
$\sT_{(k,\ell )}=Kx^ky^\ell$; in particular, 
$\sTz = K$. (So~$\sT\cong_{g} 
\gr\big(K((x))((y))\big)$
where the iterated Laurent power series ring $K((x))(y))$ is 
given
its usual rank~$2$ Henselian valuation.) Take any 
$a,b\in K^*$ such
that $\big [K(\sqrt[n]{a},\sqrt[n]{b}\,)\col K\big]=n^2$, and let 
$\sE$ be
the graded symbol algebra  $\sE=(ax^n,by^n,\sT)_\omega$, of 
degree $n^2$.
That is, $\sE$ is the graded central simple $\sT$-algebra with
homogenous generators $i$ and $j$ such that $i^{n^2}=ax^{n}$,
$j^{n^2}=by^{n}$, and $ij=\omega ji$, and  
$\deg(i)=(\frac 1n,0)$,
$\deg(j)=(0,\frac 1n)$. Then, $\Gamma_\sE=(\frac 1n\mathbb Z)\times 
(\frac 1n \mathbb Z)$,
and $\sE_0=K(i^nx^{-1},j^ny^{-1})\cong K(\sqrt[n]{a},
\sqrt[n]{b}\,)$.
Since $\sE_0$ is a field, 
by Lemma~\ref{zerosimple}(ii) $\sE$~is a graded division ring, 
which is
clearly semiramified. We can write $\sE_0=L_1
\otimes_K
L_2$ where $L_1=K(\sqrt[n]{a}\,)$ and $L_2=K(\sqrt[n]{b}\,)$, and
$H=\Gal(\sE_0/K)=\langle \sigma_1 \rangle \times 
\langle \sigma_2
\rangle$ where $\sigma_1(\sqrt[n]{a})=\omega^n\sqrt[n]{a}$,
$\sigma_1(\sqrt[n]{b})=\sqrt[n]{b}$ and
$\sigma_2(\sqrt[n]{b})=\omega^n\sqrt[n]{b}$,
$\sigma_2(\sqrt[n]{a})=\sqrt[n]{a}$. 
Since $\intt(j\inv)|_{\sEz} = \si_1$ and $\intt(i)|_{\sEz} = \si_2$, 
we can express $\sE$ as a graded abelian crossed product with 
$y_1 = j\inv$ and $y_2 = i$, 
obtaining $\sE = \sA(\sT(\sqrt[n]{a},\sqrt[n]{b})/\sT, \bsi, \bold u,
\bold d)$, where $u_{11}=u_{22}=1$, $u_{12}=\omega$, $u_{21}=\omega^{-1}$,
and 
$d_1=1/(y\sqrt[n]{b})$,  $d_2 = x\sqrt[n]{a}$. 
Graded symbol algebras 
satisfy
the same multiplicative rules in the graded Brauer group as 
do the
usual ungraded symbol algebras in the Brauer group. 
(This follows, e.g., by the injectivity of the scalar extension map
$\sBr(\sT) \to \Br(q(\sT))$, cf.~\cite[p.~90]{hwcor}.) Thus, in
$\Br(\sT)$, we have
\begin{align*}
\sE \, & \sim_{g}  (a,b,\sT)_\omega \otimes_\sT (x^n,b,\sT)_\omega 
\otimes_\sT (a,y^n,\sT)_\omega \otimes_\sT(x^n,y^n,\sT)_\omega\\
& \sim_{g}(a,b,\sT)_\omega \otimes_\sT (x,b,\sT)_{\omega^n} \otimes
(a,y,\sT)_{\omega^n}.
\end{align*} 
(The last two terms are symbol algebras of degree $n$.) Thus, 
$\sE\sim_{g}\sI \otimes_\sT \sN$
where $\sI=(a,b,\sT)_\omega$ and $\sN=(x,b,\sT)_{\omega^n}\otimes_\sT
(a,y,\sT)_{\omega^n}$. Then, 
$\sI \cong_{g}\sI_0\otimes_{\sT_0}\sT$,
where
$\sI_0=(a,b,\sT_0)_\omega=A(K(\sqrt[n]{a},\sqrt[n]{b}\,)/K,
\bsi,\bu,\bold b)$,
with the same $\bold u$ as for $\sE$ and
$b_1=1/\sqrt[n]{b}$, 
$b_2=\sqrt[n]{a}$. So, $\sI$ is an inertial central simple 
graded
$\sT$-algebra. We have $\sN_0$ is the field
$K(\sqrt[n]{a},\sqrt[n]{b}\,)$, so $\sN$ is a graded division 
algebra by Lemma~\ref{zerosimple}(ii). 
$\sN$ is DSR since it has the inertial maximal graded 
subfield
$\sT(\sqrt[n]{a},\sqrt[n]{b}\,)=\sN_0\sT$ and the totally 
ramified
maximal graded subfield $\sT(\sqrt[n]{x},\sqrt[n]{y}\, )$. 
As a graded abelian crossed product, 
$\sN \cong_g \sA(\sT(\sqrt[n]{a},\sqrt[n]{b})/\sT, \bsi, 
\bold 1, \bold c)$, where $c_1 = 1/y$, $c_2 =x$.
Let
$M=K(\sqrt[n]{a},\sqrt[n]{b}\,)$. By Prop.~\ref{NSRSK},
$$
\SK(\sN) \ \cong \  \widehat{H}^{-1}(H,M^*) \ \cong \ 
 \Br(M/K) \big / \Dec(M/K),
$$
where $H=\Gal(M/K)$  and $\Dec(M/K) = \Br(K(\sqrt[n] a\,)/K)\cdot \Br
(K(\sqrt[n] b\,)/K)$; but, by Th.~\ref{goth},
\begin{equation}\label{cyclicSK1}
\SK(\sE) \ \cong \  \widehat{H}^{-1}(H,M^*)\big / 
\langle \im(\omega) \rangle  \ \cong  \  \Br(M/K)\big/ 
\big [\Dec(M/K)\cdot\langle[(a,b,K)_\omega]\rangle \big ]. 
\end{equation}

This example is the graded version of Platonov's example  
in~\cite{platcyclic} and \cite{plat76}
of  a cyclic algebra with nontrivial $\SK$, where 
$K$ is a suitably chosen global field. (Platonov worked 
with the Henselian valued 
ground field ${K'=K((x))((y))}$ in place of the graded field $\sT=\gr(K')$ 
considered
here.) In~\cite[Th.~2]{plat76} the added term distingushing 
$\SK(\sE)$ from
$\SK(\sN)$ is omitted. This error is corrected in~\cite[p.~536, 
footnote~1]{yy} and in \cite[p.~70]{ershov}, 
giving the first isomorphism of \eqref{cyclicSK1} but not the second. 
\end{example}


\section{Unitary graded $I\otimes N$ decomposition}\label{unitaryIN}

The goal for  \S\S 4--7 is to give a unitary version of the formulas for
$\SK$ in Prop.~\ref{NSRSK} and Th.~\ref{goth} for semiramified graded 
division algebras with graded unitary involution.  In this section we 
consider abelian crossed products with unitary involution and 
prove  a unitary analogue to the $I\otimes N$ decomposition of 
Prop.~\ref{INdecomp}. 

A {\it unitary involution}  on a central simple algebra $A$ over a 
field $K$ is a ring antiautomorphism $\tau$ of~$A$ such that 
$\tau^2 = \id_A$ and $\tau|_K \ne \id$.  (Such a $\tau$ is also called 
an involution on $A$ of the second kind.)  Let ${F = K^\tau = \{ c\in K\mid
\tau(c) = c\}}$, which is a subfield of $K$ with $[K\col F] = 2$ and 
$K$ Galois over~$F$ with $\Gal(K/F) = \{\tau|_K, \id_K\}$.  Our $\tau$
is also called a {\it unitary $K/F$-involution}. 
The unitary  
$\SK(A,\tau)$ is defined just as for $\SK(D, \tau)$ in \eqref{unitarydef}.
Recall  
(see \cite[Prop.~(17.24)(2)]{kmrt}) that if $\tau'$ is another unitary 
$K/F$-involution
on $A$, then $\SK(A,\tau') = \SK(A,\tau)$.  Thus, we will freely pass 
from one unitary $K/F$-involution on $A$ to another when convenient.

In the unitary setting generalized dihedral Galois groups often arise
where abelian Galois groups appear in the nonunitary setting.  
A group $G$ is said to be {\it generalized dihedral} with respect to a 
subgroup $H$ if $|G\col H| = 2$ and for some $\theta \in G\setminus H$,
$\theta^2 = 1$ and $\theta h \theta\inv = h\inv$ for every $h\in H$.  
Equivalently,  every element of $G\setminus H$ has order $2$.  
See \cite[\S2.4]{I} for some remarks on such groups.  Note that $H$ is 
necessarily abelian.  If $H$ is cyclic, we say that $G$ is {\it dihedral}.
(This includes the trivial cases where $|H| = 1$ or $2$.)  
For fields $F \subseteq K\subseteq M$, we say that $M$ is 
{\it $K/F$-generalized dihedral} if $[M\col F] <\infty$, 
$M$ is Galois over $F$, 
and $G = \Gal(M/F)$ is generalized dihedral with respect to its
subgroup $H = \Gal(M/K)$.

\begin{lemma}\label{unitarycp} Let $F \subseteq K \subseteq M$ be
fields, and suppose $M$ is $K/F$-generalized dihedral.
Let $A$ be a central simple $K$-algebra containing 
$M$ as a strictly maximal subfield. Let $G = 
\Gal(M/F)$ and $H = \Gal(M/K)$, and fix any 
$\theta \in G\setminus H$
$($so $\theta^2 = \id_M$$)$. Then,    
the following conditions are equivalent:\begin{enumerate} 
  \item[{\rm(i)}]
$A$ has a unitary $K/F$-involution.
  \item[{\rm(ii)}]
$A$ has a unitary $K/F$-involution $\tau$ such that 
$\tau|_M = \theta$.
  \item[{\rm(iii)}]
$A \cong A(M/K, \bsi, \bu, \bb)$ where $($in addition to 
conditions 
\eqref{urels} and \eqref{brel}$)$ 
\begin{equation}\label{unitaryconds} 
u_{ij} \cdot \si_i\si_j\theta(u_{ij}) \ = \ 1  \ \ \ 
\text{and}  \  \ \ \ b_i = \theta(b_i)
 \ \ \  \  \text{for all} \ i,j.
\end{equation}
\end{enumerate}
The  $A$   in {\rm (iii)}, has a unitary $K/F$-involution $\tau$ 
with $\tau|_M = \theta$ and $\tau(z_i) = z_i$ for each of the 
standard generators $z_i$ of $A$.  
\end{lemma}
\begin{proof} 
Note that as $\theta\notin H$ and $K$ is Galois over $F$, we 
have $\theta(K) = K$ and
$\Gal(K/F) = \{ \id_K, \theta|_K\}$. 

(i) $\Rightarrow$ (ii)   This is a special case of a 
substantial result  
\cite[Th.~4.14]{kmrt} on simple subalgebras with compatible 
involutions.  For the 
convenience of the reader we give a short direct proof.
Let $\rho$ be a unitary $K/F$-involution on~ $A$, so  
$\rho |_K=\theta|_K$. 
Since $\rho\theta$ is a $K$-linear homomorphism 
$M\rightarrow A$, by the 
Skolem-Noether 
Theorem, there is $y \in A^*$ with $\intt(y)|_M = \rho\theta$.    
For any $a\in M$, as $\rho^2 = \theta^2 = \id|_M$, 
we have
$$
\rho(y) a\rho(y)\inv \ = \ \rho(y\inv \rho(a) y) \ = 
\ \rho(\rho \theta)\inv \rho(a) \ =  \ \rho\theta(a)
 \ = \ yay\inv.
$$
Therefore,  letting  $c = y\inv \rho(y)$, we have $c\in
 C_A(M)^* = M^*$ and 
$\rho(y) = yc$.  Hence,   
$$
y  \, =  \, \rho^2(y)  \, = \,  \rho(yc) \,  = \,  \rho(c) yc 
 \, = \,  \rho(c) \rho\theta(c) y \, = \, \rho(c\hsp\theta(c))y;
$$
so, $c\hsp\theta(c) = 1$.  Since $\theta^2 = \id|_M$, by Hilbert 90 
applied to the 
quadratic extension  $M/M^\theta$ there is $d\in M^*$ with 
$c = d\hsp \theta(d)\inv$. 
Let $z = yd$.  Then, as $\theta(c) = \theta(d)d\inv$,
$$
\rho(z) \, =  \, \rho(d)yc \,  = \,  
\rho(d)\rho\theta(c)\, y
 \, = \,  \rho\theta(d)\, y \,  = \,  yd  \, = \,  z.
$$
Let $\tau = \rho\circ\intt(z)$, which is an involution on $A$, as 
$\rho(z) = z$.
Then, $\tau|_M = \rho\intt(z)|_M = \rho\intt(y)|_M = \rho^2
\theta = \theta$, as desired. 

(ii) $\Rightarrow$ (iii) Let $\tau$ be a unitary 
$K/F$-involution on $A$ such that 
$\tau|_M =\theta$. For any $\sigma \in H$, we claim that 
there is 
$z\in A^*$ with $\intt(z)|_M = \si$ and $\tau(z) = z$.  
For this, first apply 
Skolem-Noether to obtain $y\in A^*$ with $\intt(y)|_M = \si$.  
For any $a\in M$ we have,
as $\tau \si \inv \tau = \si$ on $M$ since  $\tau\sigma\inv|_M \in G \setminus H$,
$$
\tau(y) a \tau(y)\inv \, = \ \tau(y\inv \tau(a) y) \, = 
\ \tau\si\inv\tau(a) \, = \ \si(a) \, = \ yay\inv.
$$ 
Hence, $\tau(y) = cy$, where $c \in C_A(M)^* = M^*$. Now,
$$
y \, = \, \tau^2(y)  \ =\  \tau(cy) \ = \ \tau(y) \tau(c)  
\ =  \ 
cy\theta(c)  \ = \  c\hsp\si \theta(c)\, y, 
$$
so $c\,\si\theta(c) = 1$.  Since $\si\theta$ has order $2$, 
Hilbert 90 
applied to the quadratic extension $M/M^{\si\theta}$ shows 
that there is 
$d\in M^*$ with $c = d\,\si\theta(d)\inv$.  Let $z = dy$.  Then, 
$\intt(z)|_M = \intt(y)|_M = \si$ and 
$$
\tau(z)  \, = \,  cy\theta(d)  \, = \, [d\,\si\theta(d)\inv]\,
\si\theta(d) \,y \, =
\, z,
$$
proving the claim.  Thus, with  our cyclic decomposition 
$H = \langle \si_1\rangle \times\ldots \times \langle \si_k
\rangle$,
we can choose $z_1, \ldots, z_k\in A^*$ with $\intt(z_i)|_M = 
\si_i$
and $\tau(z_i) = z_i$.  Then, for $b_i = z_i^{r_i} \in M^*$, 
we have 
$\theta(b_i) = \tau(b_i) = \tau(z_i^{r_i}) = b_i$.  Also, for 
$u_{ij} = z_iz_jz_i\inv z_j\inv$, we have 
$$
\si_i\si_j\theta(u_{ij})  \, = \ z_iz_j\,
\tau(z_i z_j 
z_i\inv z_j\inv)z_j\inv z_i\inv \, =  
\ z_iz_j(z_j\inv z_i\inv 
z_jz_i)z_j\inv z_i\inv
 \, = \ z_jz_iz_j\inv z_i\inv \, = \ u_{ij}\inv,
$$ 
so $u_{ij}\,\si_i\si_j\theta(u_{ij}) = 1$.  Thus, $A\cong 
A(M/K, \bsi, \bu, \bold b)$
with the $\uij$ and $b_i$ satisfying the equations in 
\eqref{unitaryconds}. 

(iii) $\Rightarrow$ (i)  Assume $A = A(M/K, \bsi, \bu, \bold b)$ 
where the 
$\uij$ and $b_i$ satisfy the conditions in 
\eqref{unitaryconds}.  Take 
$z_1, \ldots , z_k\in A^*$ with $\intt(z_i)|_M = \si$, 
$z_i^{r_i} = b_i$ and 
$z_iz_jz_i\inv z_j\inv = \uij$.  We show that there is a  unitary 
$K/F$-involution 
$\tau$ on $A$ satisfying (and determined by) 
$\tau|_M = \theta$ and 
$\tau(z_i) = z_i$ for each $i$.  Basically, this is a matter 
of checking that the $\tau$ just described is compatible 
with the defining 
relations of $A$.  Here is a more complete argument, based 
on the 
description of $A(M/K, \si_i, \uij, b_i)$ given in the proof of
\cite[Th.~1.3]{as}.    
First, take any ring $B$ with an automorphism $\si$, 
 and let $B[y;\si]$ be the twisted polynomial ring 
$\{\sum c_iy^i\,|\ c_i\in B\}$ with the multiplication 
determined by 
$yc = \si(c) y$ for all $c\in B$. It is easy to check that 
an involution 
$\rho$ on $B$ extends to an involution $\rho'$ on $B[y;\si]$ 
with 
$\rho'(y) = y$ iff $\si\rho\si = \rho$.  Also, for 
$d\in B^*$, an 
automorphism $\eta$ of $B$ extends to an automorphism 
$\eta'$ of 
$B[y;\si]$ with $\eta'(y) = dy$ iff $\intt(d) \si\eta = 
\eta \si$.  
Here, let $B_0 = M$, 
$B_1 = B_0[y_1;\si_1^*]$, \ldots, $B_\ell = 
B_{\ell-1}[y_\ell;\si_\ell^*]$,
\ldots, $B_k = B_{k-1}[y_k;\si_k^*]$, where $\si_1^* = 
\si_1$ and for~$\ell>1$,
the automorphism 
$\si_\ell^*$ of $B_{\ell-1}$ is defined by $\si_\ell^*|_M = 
\si_\ell$ and $\si_\ell^*(y_i)
= u_{\ell i} y_i$ for $1\le i<\ell$.  (One checks inductively 
using the identities in  \eqref{urels} that
for $1\le i \le \ell-1$,   
$\si_\ell^*$
satisfies $\intt(u_{\ell i})\si_i^*\si_\ell^* = \si_\ell^* 
\si_i^*$ on~$B_{i-1}$, hence
 $\si_\ell^*$ extends from $B_{i-1}$ to $B_i$; 
 thus, $\si_\ell^*$ is an automorphism of $B_{\ell-1}$.)    
Define inductively involutions
$\tau_i$ on $B_i$ by $\tau_0 = \theta$ and for $\ell >0$, 
$\tau_\ell |_{B_{\ell-1}}
= \tau_{\ell-1}$ and $\tau_\ell(y_\ell) = y_\ell$. Given 
$\tau_{\ell-1}$, the 
condition for the existence of~$\tau_\ell$ is that 
$\si_\ell^* \tau_{\ell-1}\si_\ell^* = \tau_{\ell-1}$.  
For this, note first that 
$\si_\ell^* \tau_{\ell-1}\si_\ell^*|_M = \si_\ell 
\theta \si_\ell = \theta =
\tau_{\ell-1}|_M$ as $G$ is generalized dihedral.  
Furthermore, for $1\le i<\ell$, 
$$
\si_\ell^* \tau_{\ell-1}\si_\ell^*(y_i)  \, = \ 
\si_\ell^*\tau_{\ell-1}(u_{\ell i} y_i)
 \, =  \ \si_\ell^*\big(y_i\, \theta(u_{\ell i})\big)  \, = \ 
\si_\ell^*[\si_i\theta(u_{\ell i})\, y_i] 
 \, = \,  [\si_\ell\si_i\theta(u_{\ell i})] u_{\ell i}\, y_i \, =  \ y_i   
 \, = \ \tau_{\ell-1}(y_i).
$$
Thus, $\si_\ell^* \tau_{\ell-1}\si_\ell^*$ agrees with 
$\tau_{\ell-1}$
throughout $B_{\ell-1}$, as needed.  By induction, we have 
the involution 
$\tau_k$ on $B_k$.  As pointed out in \cite[p.~79]{as}, 
$A\cong B_k/I$, 
where $I$ is the two-sided ideal of $B_k$ generated by 
$\{ y_i^{r_i}-b_i \, |\ 1\le i\le k\}$.  Since 
$\tau_k(b_i) = \theta(b_i) = b_i$, 
$\tau_k$ maps each generator of $I$ to itself. Therefore, 
$\tau_k$ induces an involution 
$\tau$ on $A \cong B_k/I$ which clearly restricts to 
$\theta$ on $M$; so 
$\tau$ is a unitary $K/F$-involution on $A$.   
\end{proof}

We write $\Br(M/K;F)$ for the subgroup of $\Br(M/K)$ of 
algebra classes $[A]$
such that $A$ has a unitary $K/F$-involution.  By Albert's 
theorem 
\cite[Th.~3.1(2)]{kmrt}, $\Br(M/K;F)$ is the kernel of the 
corestriction map
$\text{cor}_{K\to F}\colon \Br(M/K) \to \Br(M/F)$.  
For $M$ a $K/F$-generalized 
dihedral extension of $F$, as above, there is 
in addition a 
corresponding subgroup of 
$\Dec(M/K)$.  For this, note first that for any field $L$ 
with 
$K \subseteq L\subseteq M$ and $L$ cyclic 
Galois over $K$, say $\Gal(L/K) = \langle \si \rangle$, 
$L$ is $K/F$-dihedral,
so Lemma~\ref{unitarycp} (with $k = 1$) implies that
$\Br(L/K;F) = \{[(L/K, \si, b)]\mid b\in F^*\}$.  For 
$H = \Gal(M/K)
= \langle \si_1\rangle \times \ldots \times\langle 
\si_k \rangle$ and  $(\chi_1, \ldots, \chi_k)$ the base
of $X(M/K)$ dual to $(\si_1, \ldots, \si_k)$, and $L_i$ the
fixed field of $\ker(\chi_i)$, 
as at the beginning of \S\ref{abcp}, define 
\begin{equation}
\Dec(M/K;F) \ = \ \{ [(L_1/K, \si_1, b_1)\otimes_K\ldots 
\otimes _K
(L_k/K,\si_k, b_k)]\mid \text {each } b_i\in F^*\} \ 
\subseteq  \ \Br(M/K;F)\}.
\end{equation}
Note that $\Dec(M/K;F)$ is generated as a group by the 
image under the cup 
product of $H^2(H,\zz)\times  F^*$.  Thus $\Dec(M/K;F)$ is 
independent of 
the choice of cyclic decomposition of $H$, and we have
analogously to~\eqref{dec}, 
\begin{equation}\label{Decformula}
\Dec(M/K;F) \ =  \ \textstyle\prod\limits _{i = 1}^k 
\Br(L_i/K;F)
 \ = \ \prod \limits_
{\substack{K \subseteq L \subseteq M \\ \Gal(L/K) 
\text{ cyclic}}}  
\Br(L/K;F). 
\end{equation}

For the rest of this section we fix a graded field 
$\sT$ and a graded subfield $\sR\subseteq \sT$
such that  $[\sT\col\sR] =2$ and $\sT$ is inertial and 
Galois over $\sR$.  Let $\psi$ be the nonidentity graded 
$\sR$-automorphism of $\sT$, and let~$\psi_0$~be the 
restriction $\psi|_{\sT_0}$.
Thus, $\Gamma_\sT = \Gamma_\sR$, 
$[\sT_0\col\sR_0]= 2$, $\sT \cong_g \sT_0\otimes_{\sR_0} \sR$,
$\sT_0$ is Galois over $\sR_0$, and $\psi$ on $\sT$ corresponds 
to  $\psi_0 \otimes \id_\sR$ on $ \sT_0
\otimes_{\sR_0} \sR$.  We are interested in central simple 
graded $\sT$-algebras $\sA$ with graded unitary 
$\sT/\sR$-involutions $\tau$.  This means that $\tau$ is a 
degree-preserving ring antiautomorphism of $\sA$ with 
$\tau^2 = \id_\sA$ and the ring of invariants $\sT^\tau=
\sR$; the last condition is equivalent to $\tau|_\sT = \psi$. Suppose now 
that $\sA$ is a graded division algebra.  Set $\tau_0
= \tau|_{\sA_0}$, which is a unitary involution on 
$\sA_0$, as $\tau_0|_{\sT_0} = \psi_0 \ne \id$ and 
$\sT_0 \subseteq Z(\sA_0)$.  Just as for any graded division 
algebra, $Z(\sA_0)$ is abelian Galois over $\sT_0$.  But the 
presence of the involution~$\tau$ implies further that 
$Z(\sA_0)$ is actually $\sT_0/\sR_0$-generalized dihedral,
by \cite[Lemma~4.6(ii)]{I}.                                   

A central graded division algebra $\sN$ over $\sT$ is said to
be {\it decomposably semiramified for $\sT/\sR$} 
(abbreviated DSR for $\sT/\sR$) if $\sN$ has a 
unitary graded $\sT/\sR$-involution $\tau$ and a
maximal graded subfield $\sM$ inertial over~$\sT$ and another 
maximal graded subfield $\sJ$  with $\sJ$  totally ramified 
over~$\sT$ and $\tau(\sJ) = \sJ$.  When this occurs, 
$\sN$~is semiramified with $\sN_0 = \sM_0$, a field, 
which as just noted is $\sT_0/\sR_0$-generalized dihedral.
Also, $\Gamma_\sN = \Gamma_\sJ$ and $\Theta_\sN$ induces an
isomorphism $\Gamma_\sN/\Gamma_\sT \cong \Gal(\sN_0/
\sT_0)$. Furthermore, as $\sM = \sM_0\sT = \sN_0\sT$, we have
$\tau(\sM) = \sM$.

\begin{example}\label{DSRex}
Let $L$ be any cyclic Galois field extension of $\sTz$ with 
$L$ dihedral over $\sRz$.  (That is, $L$ is Galois over $\sRz$
and there is $\theta \in \Gal(L/\sRz)\mi\Gal(L/\sTz)$ with 
$\theta^2 = \id_\sL$ and
$\theta h\theta\inv = h\inv$ for every $h\in \Gal(L/\sTz)$.
Thus, the group $\Gal(L/\sRz)$ is either dihedral or isomorphic to 
$\zz/2\zz$ or $\zz/2\zz\times \zz/2\zz$.)  Let $r = 
[L\col\sTz]$, and take any $b\in \sR^*$ with the image 
of $\deg(b)$ having order $r$ in $\Ga_\sT/r\Ga_\sT$.  Take 
any generator $\si$ of $\Gal(L/\sTz)$, and let 
$\si$ denote also its canonical extension $\si\otimes \id_\sT$
 in $\sGal((L\otimes _\sTz \sT)/\sT)$.  Let 
$$
\sN \, = \,((L\otimes _\sTz \!\sT)/\sT, \si, b), 
\text{ a cyclic graded algebra over $\sT$}.
$$ 
We show that  $\sN$ is a central graded division algebra over $\sT$ 
of degree $r$, and $\sN$ is DSR for $\sTR$. For, letting
$L\sT$ denote $L\otimes _\sTz \!\sT$, note that $L\sT$ is a 
graded field which is inertial over $\sT$  and is Galois over 
$\sT$ with ${\sGal(L\sT/\sT) = \langle \si\rangle}$.  Our $\sN$
is $\bigoplus_{i=0}^{r-1}L\sT z^i$, where $zcz\inv = 
\si(c)$ for all $c\in L\sT$, and $z^r = b$, with the grading 
on~$\sN$ extending that on $L\sT$ by setting $\deg(z)=
\frac 1 r \deg(b)$.  A graded cyclic $\sT$-algebra is 
always graded simple with center~$\sT$. 
Note that for $j\in \zz$, 
if $j\deg(b)/r\in \Gamma_{L\sT} = \Gamma_\sT$,
 then $j\deg(b) \in r\Gamma_\sT$,
so by hypothesis $r\,|\, j$. 
Hence,
$$
\sN_0 \,=\,\textstyle \sum\limits_{i=0}^{r-1}(L\sT)_{-i\deg(b)/r}\,z^i 
\ =\, 
(L\sT)_0 \,=\, L.
$$ Since $\sN_0$~is a division ring, the 
simple graded algebra~$N$~is a graded division ring,
by Lemma~\ref{zerosimple}(ii).   
 Also, as
${[L\sT\col \sT] = [L\col \sTz] = r = \deg(\sN)}$,
$L\sT$~is a maximal graded subfield of 
$\sN$ which is inertial over $\sT$.  Take any 
$\theta\in \Gal(L/\sRz)$ with ${\theta|_\sTz = \psi_0}$, 
and let $\theta$ denote    
also its canonical extension $\theta \otimes \id|_\sR$
to $\sGal(L\sT/\sR)$.  Define a map $\tau\colon \sN\to \sN$
by 
$$
{\tau(\textstyle\sum\limits_{i=0}^{r-1} c_i z^i) \,
=\, \sum\limits_{i=0}^{r-1}  z^i\theta (c_i) \,=\, 
\sum\limits_{i=0}^{r-1} \si^i \theta(c_i) z^i}.
$$  Since $\theta|_\sT = \psi$,
$\theta^2 = \id$, 
and $\theta \si \theta \inv = \si\inv$ (as $L$ is 
$\sTz/\sRz$-dihedral), it is easy to check 
that $\tau$ is a graded $\sTR$-involution of $\sN$.  Moreover, 
if we let $\sJ = \bigoplus_{i=0}^{r-1} \sT z^i= \sT[z]$, then
$\sJ$ is a maximal graded subfield of $\sN$, and the hypothesis
on $\deg(b)$ assures that $\sJ$ is totally ramified over $\sT$;
also $\tau(\sJ) = \sJ$.  This verifies that $\sN$~
is DSR for $\sTR$.  Note that $\sN_0 = L$ and
$\Ga_\sN = \langle \frac 1r \deg(b)\rangle + \Ga_\sT$.  
\end{example}

\begin{lemma}\label{DSRprod}
Let $\sN$ and $\sN'$ be graded division algebras which 
are each  DSR for $\sT/\sR$.  Suppose  
 $\sN_0$ and $\sN'_0$ are linearly disjoint over 
$\sT_0$ and $\Gamma_\sN \cap\Gamma_{\sN'} = \Gamma_\sT$.
Then, $\sN\otimes _\sT \sN'$ is a graded division algebra
which is DSR for~$\sTR$.  Also, $(\sN\otimes_\sT\sN')_0
\cong \sN_0 \otimes _\sTz \sN_0'$ and 
$\Ga_{\sN\otimes _\sT\sN'} = \Ga_\sN + \Ga_{\sN'}$.
\end{lemma}

\begin{proof} 
Let $\sB = \sN\otimes_\sT \sN'$, which is a 
central simple graded $\sT$-algebra, since this is true 
for $\sN$ and $\sN'$ by \cite[Prop.~1.1]{hwcor}.  
For each $\gamma\in \Ga_\sT$ choose a 
nonzero $t_\gamma\in \sT_\gamma$.
Then, 
$$
\sB_0  \ = \textstyle\sum \limits_{\gamma\in \Ga_\sN\cap 
\Ga_{\sN'}}\sN_\ga\otimes 
_\sTz\! \sN'_{-\ga} \ = \ \sum\limits_{\ga\in \Ga_\sT}
\sN_0 \,t_\ga\otimes _\sTz \!\sN'_0\,t_\ga\inv \ = 
\ \sN_0 \otimes_\sTz\!\sN'_0.
$$
The linear disjointness hypothesis assures that 
$\sB_0$ is a field, and hence $\sB$ is a 
graded division ring, by Lemma~\ref{zerosimple}(ii).
Moreover, by dimension count $\sB_0 \sT$
is a graded maximal subfield of $\sB$ which is inertial over~ 
$\sT$. Let $\tau$ be a graded $\sTR$-involution of
$\sN$, and let $\sJ$ be a graded maximal subfield of
$\sN$ with $\tau(\sJ) = \sJ$.  Take $\tau'$ and $\sJ'$ 
correspondingly for $\sN'$.  Then, 
$\sJ\sJ' = \sJ\otimes_\sT\sJ'$ and $\tau \otimes \tau'$
is a graded $\sTR$-involution on $\sB$ with
$(\tau\otimes \tau')(\sJ\sJ') = \sJ\sJ'$.  Moreover, 
$\sJ\sJ'$ is a maximal graded subfield of $\sB$ by dimension count,
and, as $\Ga_\sJ\cap \Ga_{\sJ'} = \Ga_\sN\cap \Ga_{\sN'} =
\Ga_\sT$, we have 
$$
|\Ga_{\sJ\sJ'}\col \Ga_\sT|
\, \ge\, |\Ga_{\sJ} +\Ga_{\sJ'}\col\Ga_\sT|\,  = \, |\Ga_\sJ\col\Ga_\sT|
\cdot |\Ga_{\sJ'}\col \Ga_\sT|\, =\,  [\sJ\col\sT]\cdot [\sJ'\col\sT]
\, =\,  [\sJ\sJ'\col\sT]. 
$$  
Hence, $\sJ\sJ'$ is totally
ramified over~$\sT$.  Thus, $\sB$ is DSR for $\sTR$.  
\end{proof}

The next proposition shows that all graded division algebras
$\sN$
which are DSR for $\sTR$ are obtainable from those in 
Ex.~\ref{DSRex} by iterated application of Prop.~\ref{DSRprod}.
This justifies the term \lq\lq decomposably semiramified" for such $\sN$.

\begin{proposition}\label{DSRdecomp} 
Let $\sN$ be a graded division algebra which is DSR for 
$\sTR$.  Take any decomposition ${\sN_0 = L_1 \otimes _\sTz
\ldots \otimes _\sTz\!L_k}$ with each 
$L_i$ cyclic Galois over $\sTz$, and choose correspondingly\break 
${\si_1 \ldots, \si_k\in \sGal(\sN_0\sT/\sT) \cong
\Gal(\sN_0/\sTz)}$ such that $\si_i|_{L_j}
= \id$ whenever $j\ne i$ and $\Gal(L_i/\sTz) = \langle 
\si_i|_{L_i}\rangle$ for each $i$.  $($So ${\sGal(\sN_0\sT/\sT)=
\lan \si_1\ran \times \ldots \times \lan \si_k\ran}$.$)$
Let $r_i$ be the order of $\si_i$.  
For each $i$ choose $\ga_i\in \Ga_\sN$ with $\Theta_\sN(\ga_i)
= \si_i$.  Then, there exist $b_1,\ldots, b_k \in \sR^*$ 
such that $\deg(b_i) = r_i \gamma_i$ and 
$$
\sN \ \cong_g (L_1\sT/\sT, \si_1, b_1) \otimes _\sT
\ldots \otimes _\sT (L_k\sT/\sT, \si_k, b_k)
\ \cong_g \ \sA(\sN_0\sT/\sT, \bsi,\boldsymbol 1, \bb).
$$
\end{proposition}

\begin{proof}
Since $\sN$ is DSR for $\sTR$, there is a graded 
$\sTR$-involution $\tau$ of $\sN$ and a maximal graded 
subfield $\sJ$ of $\sN$ with $\sJ$ totally ramified 
over $\sT$ and $\tau(\sJ) = \sJ$.  As noted earlier, 
we have $\Ga_\sJ = \Ga_\sN$.  Since $\tau$ is a graded
automorphism of $\sJ$ of order $2$, the fixed set
$\sS = \sJ^\tau = \{a\in \sJ\mid \tau(a) = a\}$ is 
a graded subfield of $\sJ$ with $2 = [\sJ\col\sS] = 
[\sJ_0\col\sS_0]\,|\Ga_\sJ\col\Ga_\sS|$.
Since $\sS_0 \cap\sTz = \sRz \subsetneqq \sTz =
\sJ_0\cap \sTz$ we have  $\sS_0\subsetneqq\sJ_0$, so
$[\sJ_0\col\sS_0] = 2$, and hence $\Ga_\sS = \Ga_\sJ
\  (= \Ga_\sN)$.  
Thus, for each $i$ there is a nonzero $x_i \in \sS_{\ga_i}$,
and for any such $x_i$, 
 $\intt(x_i)|_{\sN_0\sT} = \si_i$ as 
$\Theta_\sN(\gamma_i) = \si_i$.  Let $b_i = x_i^{r_i} \in \sS^*$.
Then, $\Theta_\sN(\deg(b_i)) = \si_i^{r_i} = \id$, 
so $\deg(b_i) \in \ker(\Theta_\sN) = \Ga_\sT$;
hence, ${b_i \in \sJ_{\deg(b_i)} = \sT_{\deg(b_i)}}$ as 
$\sJ$ is totally ramified over $\sT$.  Therefore,
$b\in \sS^* \cap \sT= \sR^*$.  Let $\sC_i$ be the 
graded $\sT$-subalgebra of~$\sN$ generated by 
$L_i$ and $x_i$.  Since $\intt(x_i)|_{L_i\sT} = \si_i|_{L_i\sT}$,
 there is  a graded $\sT$-algebra epimorphism
$(L_i\sT/\sT,\si_i,b_i) \to \sC_i$, which is a graded 
isomorphism as the domain is graded simple.  Since the 
$x_i$ all lie in the graded field $\sS$
and $\si_i|_{L_j\sT} = \id$ for $j\ne i$, the distinct 
$\sC_i$ centralize each other.  Hence, there is a graded
$\sT$-algebra homomorphism 
${(L_1\sT/\sT,\si_1,b_1) \otimes _\sT \ldots \otimes _\sT 
(L_k\sT/\sT,\si_k,b_k)\to \sN}$ which is injective as 
the domain is graded simple, and surjective by dimension
count.  Clearly also, ${(L_1\sT/\sT,\si_1,b_1) \otimes _\sT \ldots \otimes _\sT 
(L_k\sT/\sT,\si_k,b_k)\cong_g \sA(\sN_0\sT/\sT, \bsi,
\bone, \bb)}$.
\end{proof} 

\begin{proposition}\label{uINdecomp} 
Let $\sE$ be a semiramified central graded division 
algebra over $\sT$, and suppose $\sE$ has a graded
$\sTR$-involution, where $\sT$ is inertial over $\sR$.
  Then, $\sE_0$ is $\sTz/\sRz$-generalized dihedral
and 
\begin{enumerate} 
  \item[{\rm(i)}]
$\sE \sim_g \sI \otimes _\sT \sN$ in $\sBr(\sT)$
for some $\sT$-central graded
division algebras $\sI$ and $\sN$ with $\sI$
inertial and $\sN$~DSR~for~$\sTR$. 
  \item[{\rm(ii)}]
Take any decomposition 
$\sT \sim_g \sI' \otimes _\sT \sN'$ in $\sBr(\sT)$
with graded $\sT$-central division algebras $\sI'$
and $\sN'$ with $\sI'$~inertial and 
$\sN'$ DSR for $\sTR$.  Then, 
$\sN_0'\cong \sE_0$, $\Ga_{\sN'}= \Ga_\sE$, $\Theta_{\sN'} = 
\Theta_\sE$, and ${[\,\sI_0'] \in \Br(\sE_0/\sTz;\sRz)}$.
Furthermore, $\sI_0'$ is uniquely determined modulo
$\Dec(\sE_0/\sTz;\sRz)$.
\end{enumerate}
\end{proposition}

\begin{proof} (i) 
Since $\sE$ is semiramified, $\sE_0\sT$ is an inertial 
 maximal graded subfield of $\sE$.  Moreover, as 
$\sE$~has a graded $\sTR$-involution, $\sE_0$
is $\sTz/\sRz$-generalized dihedral, by \cite[Lemma~4.6(ii)]{I}.   
Because $\sE$ has an inertial graded maximal subfield,
it is a graded abelian crossed product:  Say 
$\sE_0 = L_1\otimes_\sTz\ldots\otimes _\sTz L_k$, where 
each field  $L_i$~is cyclic Galois over $\sTz$ (so 
dihedral over $\sRz$).  Then $G = \sGal(\sE_0\sT/\sT)
\cong \Gal(\sE_0/\sTz)$ has a corresponding 
cyclic decomposition
$G = \lan \si_1\ran\times \ldots \times \lan\si_k\ran$,
where each $\si_i|_{L_j\sT} = \id$ for $j\ne i$, 
and $\si_i|_{L_i\sT}$ generates $\sGal(L_i\sT/\sT)$.
Let $r_i = |\lan\si_i\ran|=[L_i\col\sTz]$.
By Lemma~\ref{incp}, $\sE = \sA(\sE_0\sT/\sT, \bsi, \bu,\bb)$
where each $u_{ij} \in \sEz^*$, $b_i \in \sEz\sT^*$,
 $\frac 1{r_i}\deg(b_i) + \Ga_\sT$ has order $r_i$ in 
$\Ga_\sE/\Ga_\sT$, and
\begin{equation}\label{relGamma}
\textstyle\Ga_\sE/\Ga_\sT  \, = \  
\lan\frac 1{r_1}\deg(b_1) + \Ga_\sT \ran\times \ldots \times
\lan\frac 1{r_k}\deg(b_k) + \Ga_\sT \ran.
\end{equation}
So, $\deg(b_i) \in \Ga_{\sEz\sT} = \Ga_\sT = \Ga_\sR$ and 
the image of  $\deg(b_i)$ has order $r_i$ in $\Ga_\sT/r_i\Ga_\sT$. 
For each $i$, choose $c_i\in \sR^*$ with $\deg(c_i) = 
\deg(b_i)$.  Let 
$$
\sN \, = \  \sC_1\otimes _\sT \ldots \otimes_\sT \sC_k,
\ \ \text{where each}\ \ \sC_i = (L_i\sT/\sT, \si_i, c_i).
$$ 
By Ex.~\ref{DSRex} each $\sC_i$ is DSR for $\sTR$ with 
$(\sC_i)_0 \cong L_i$ and $\Ga_{\sC_i} = \lan\frac1{r_i}
\deg( c_i)\ran +\Ga_\sT =\lan\frac1{r_i}
\deg(b_i)\ran +\Ga_\sT$.  It follows by
induction on~$k$ using  Lemma~\ref{DSRprod} 
and \eqref{relGamma} that 
$\sN$ is a graded division algebra which  is DSR for $\sTR$. 
Choose $z_i\in \sC_i^*$ with $\intt(z_i)|_{L_i\sT} = \si_i$
and $z_i^{r_i} = c_i$.  
Then, when we view $z_i\in \sN^*$, we have $\intt(z_i) = \si_i$
on all of $\sN_0\sT$.  Since further $z_iz_j= z_jz_i$ for all 
$i,j$, our
$\sN$ is the graded abelian crossed product
$\sN = \sA(\sE_0\sT/\sT, \bsi, \bone, \bf c)$.  For its 
opposite algebra $\sN\op$ we then have 
$\sN\op \cong_g \sA(\sE_0\sT/\sT, \bsi, \bone, \bf d)$
where each $d_i = c_i\inv$.  
Let $\wh \sI = \sA(\sE_0\sT/\sT, \bsi, \bu, \bf e)$
where each $e_i =b_id_i=b_i c_i\inv\in\sEz^*$.  The $u_{ij}$ and $b_i$
satisfy conditions \eqref{urels} and \eqref{brel}, as do the 
$c_i$ with the corresponding $u_{ij}=1$; hence the $u_{ij}$ here and~ 
$e_i$ satisfy \eqref{urels} and \eqref{brel};  also,  $\deg(u_{ij}) = 0$
for all $i,j$.
So, $\wh \sI$ is a well-defined graded abelian crossed product. 
By Remark~\ref{abeliancpprod}, 
we have $\wh \sI \sim_g\sE \otimes_\sT \sN\op$.
There are homogeneous $x_1, \ldots, x_k\in \wh \sI^{\, *}$ 
such that $\intt(x_i)|_{\sEz\sT} = \si_i$, $x_i^{r_i}= e_i$, and 
$x_ix_jx_i\inv x_j\inv = \uij$ for all $i,j$.  Then, 
$\deg(x_i) = \frac 1 {r_i}\deg(e_i) = 0$; hence,
$\deg(x^{\bold i}) = 0$ for each $\bold i \in 
\mathcal I=\prod_{i=1}^k
\{0,1,2,\dots,r_i-1\}$. Thus, in $\wh\sI  = \bigoplus
_{\bold i \in \mathcal I}\sEz\sT x^{\bold i}$ we have
$\wh \sI_{\,0}  =\bigoplus_{\bold i \in \mathcal I}
\sEz x^{\bold  i}  \cong A(\sEz/\sTz, \bsi, \bu, \bold e)$,
which is a central simple $\sTz$-algebra with 
$\dim_\sTz(\,\,\wh \sI_{\,0}) = [\sEz\col\sTz]^2 = 
\dim_\sT\big(\,\wh \sI\,\big)$. Hence, $\wh \sI$
is inertial over~$\sT$. Since $\wh\sI$ is simple, 
by Lemma~\ref{zerosimple} $\wh\sI\cong_g M_\ell(\sI)$
for a graded division algebra $\sI$ with $\wh\sI_{\,0} \cong
M_\ell(\sI_0)$.  Then, $[\sI_0\col \sTz] = 
\frac 1{\ell^2}\dim_\sTz(\,\,\wh \sI_{\,0})= 
\frac 1{\ell^2}\dim_\sT\big(\,\wh \sI\,\big) =
[\sI\col\sT]$, showing that $\sI$ is inertial over $\sT$.
Since $\sI \sim_g \wh \sI$, we have in~$\sBr(\sT)$,
$$
[\sE] \, = \, [\sE]\,[\sN]\inv[\sN] \, = \,
[\sE \otimes_\sT\sN\op] \, [\sN] \, = \,[\,\wh\sI\,]\,[\sN]
\, = \,[\,\sI\,]\,[\sN] \,=\, [\,\sI \otimes_\sT \sN\,],  
$$
i.e., $\sE \sim_g \sI \otimes _\sT \sN$, proving 
(i).  Also, $\sN$ has a graded $\sTR$-involution $\tau_\sN$,
which is also a graded involution for~$\sN\op$, and 
 $\sE$ has a graded $\sTR$-involution $\tau_\sE$.  So,
$\tau = \tau_\sE \otimes \tau_\sN$ is a graded $\sTR$-involution
on $\wh\sI\,$, and $\tau_0 = \tau|_{\,\wh \sI_{\,0}}$ is a  
$\sTz/\sRz$-involution on $\wh\sI_{\,0}$.  
 So, in  $\Br(\sTz)$ we have 
 $[\,\sI_0]  = [\,\wh \sI_{\,0}] \in \Br(\sEz/\sTz;\sRz)$. 

(ii)  Take any decomposition $\sE \sim_g \sI' \otimes
\sN'$ as in (ii).  
Since $\sI'$ is inertial and 
$\sE$ is the graded division algebra with $\sE\sim_g
\sI'\otimes _\sT \sN'$, Cor.~\ref{ItensorE} yields 
$\sEz\sim \sI'_0 \otimes_\sTz \sN'_0$ and $ \sEz =Z(\sEz)
\cong Z(\sN'_0) = \sN'_0$, so $\sEz$~splits~
$\sI'_0$; furthermore, $\Ga_\sE = \Ga_{\sN'}$
and $\Theta_\sE = \Theta_{\sN'}$.  
We now use the $b_i$, $c_i$, $\sN$, and $\sI$ of part (i).
Because $\sN'$ is DSR with $\sN_0'\cong \sEz$ and 
$\Theta_{\sN'}(\frac 1 {r_i}\deg(c_i)) = \Theta_{\sE}
(\frac 1 {r_i}\deg(b_i)) = \si_i$, by Prop.~\ref{DSRdecomp}
there exist $c_1', \ldots, c_k'\in \sR^*$ with 
${\deg(c_i') = \deg(c_i)}$ such that 
$\sN' \cong_g\sA(\sEz\sT/\sT, \bsi, \bone, \bold c')$.  
Let $\sB = \sA(\sEz\sT/\sT, \bsi, \bone,  \bold f)$ where
each $f_i = c_i c_i^{\prime -1} \in \sR_0^*$.  So, in 
$\sBr(\sT)$, $\sB \sim_g \sN\otimes _\sT 
\sN^{\prime \,\text{op}}
\sim_g\sI'\otimes _\sT \sI\,\op$.  Because $\deg(f_i) = 0$ for 
each $i$, the argument for $\wh\sI$ in (i) shows that $\sB$ 
is inertial over $\sT$ with 
$$
\sB_0 \, \cong \ A(\sEz/\sTz, \bsi, \bone, \bold f) \ 
\cong \ (L_1/\sTz, \si_1,f_1)\otimes _\sTz \ldots 
\otimes _\sTz(L_k/\sTz, \si_k, f_k).  
$$
Thus, $[\sB_0] \in \Dec(\sEz/\sTz;\sR_0)$, as each 
$f_i \in \sR_0^*$ (see Ex.~\ref{DSRex}).
  Let 
$\sC$ be the graded division algebra with 
$\sC \sim_g \sB \sim_g \sI' \otimes _\sT\sI\op$.
Since $\sB_0$ is simple and $\sI'$ is inertial, 
Lemma~\ref{zerosimple} and Cor.~\ref{ItensorE}
yield $\sC_0 \sim \sB_0$ and 
$\sC_0 \sim (\sI'\otimes_\sT\sI\op)_0\cong\sI'_0\otimes _\sTz \sI_0\op$;
so, in $\Br(\sTz)$, 
$$
[\,\sI'_0]\, = \, [\sC_0] \, [\,\sI_0] \, =
 \, [\sB_0] \, [\,\sI_0]  \, = \,
[\sB_0]\,[\,\wh\sI_{\,0}] \, \in \Br(\sEz/\sTz;\sRz).
$$
Since $[\sB_0] \in 
\Dec(\sEz/\sTz;\sRz)$, we have ${\sI'_0 \equiv \sI_0\ 
(\modd \ \Dec(\sEz/\sTz;\sRz)\hsp)}$.  This yields the uniqueness 
of $\sI'_0$ modulo $\Dec(\sEz/\sTz;\sRz)$ independent of 
the choice of decomposition of $\sE$ as $\sI' \otimes _\sT
\sN'$.
\end{proof}

\begin{remark}
The $\sI \otimes \sN$ decomposition described in Prop.~\ref{uINdecomp}
for $\sE$ semiramified actually holds more generally for $\sE$ inertially
split (with graded $\sTR$-involution), i.e., when $\sE$ has a 
maximal graded subfield inertial over~$\sT$.  One then has 
$\sN_0 \cong Z(\sEz)$ and $\sI_0 \otimes _\sTz Z(\sEz) \sim \sEz$.
See \cite[Lemma~5.14, Th.~5.15]{jw} for the nonunitary nongraded Henselian valued 
analogue of this.    
\end{remark}


\section{Galois cohomology with twisted coefficients}\label{twist}

Where $\wh H\inv(H,M^*)$ occurs in formulas for $\SK$ as in \S\ref{abcp}, 
analogous formulas for the unitary $\SK$ involve $\wh H\inv(G,
\wi {M^*})$ for a twisted action of $G$ on the multiplicative group 
$M^*$. In this section, we recall the relevant twisted action, and give some 
calculations concerning $\wh H\inv$ which will be used later.  
The cohomology with twisted action also allows us to give a new 
interpretation of Albert's corestriction condition for an algebra 
to have a unitary involution, see Prop.~\ref{relbriso} below. 

Let $G$ be a profinite group with a closed subgroup $H$ with 
$|G\!\col \!H| = 2$.  From the mappping\break 
${G/H \xrightarrow{\sim} \zz/2\zz \xrightarrow{\sim} \Aut(\zz)}$ 
we obtain a nontrivial discrete $G$-module structure on 
$\zz$ for which for $g\in G$, $j\in \zz$, 
$$
g*j  \ = \  \begin{cases}  \  \ j, &\text {if } g\in H,
\\-j, &\text{if } g\notin H.
\end{cases}
$$
Let $\wi \zz$ denote $\zz$ with this new $G$-action.  
Then, for any discrete $G$-module $A$ we have an 
associated discrete $G$-module 
$\wi A = A \otimes _\zz \wi\zz$.  That is, $\wi A = A$
as an abelian group, but the $G$-action 
on $\wi A$ (denoted by $*$, while $\cdot$~denotes the 
$G$-action on $A$) is given by 
\begin{equation}\label{tildeseq}
g*a  \ = \  \begin{cases}  \  \ g\cdot a, &\text {if } g\in H,
\\-g\cdot a, &\text{if } g\notin H,
\end{cases}
\quad \text{for all $g\in G$, $a\in A$}. 
\end{equation}
So, the actions of $H$ on $\wi A$ and on $A$ coincide, and  
$\wi{\wi A\,} = A$ as $G$-modules.  The cohomology of 
such modules is discussed in \cite[Appendix]{ae},
\cite[\S30.B]{kmrt}, \cite[\S5]{hkrt}.  Notably,
there is a canonical short exact sequence of $G$-modules
$$
0  \ \longrightarrow \  \wi A \  \longrightarrow  \ 
\Ind_{H\to G}(A) \ \longrightarrow  \ A \ 
\longrightarrow  \ 0
$$
Since Shapiro's Lemma says that 
$\widehat H^i(G, \Ind_{H\to G}(A))
\cong  \wh H^i(H,A)$ for all $i\in \zz$, this yields  
a long exact sequence of Tate cohomology groups:
\begin{equation}\label{longexact}
\ldots  \ \lra \ \wh H^{i-1}(G, A) \ \lra \wh H^i(G, \wi  A) 
 \ \lra \ \wh H^i(H, A)  \ \lra \ \wh H^i(G,A) \ 
\lra \ \wh H^{i+1}(G, \wi  A) \  
\lra \ \ldots
\end{equation} 
(This is  stated in \cite[(30.10)]{kmrt} and \cite{ae} for nonnegative 
indices, but it is valid for $i<0$ as well.)  For the trivial 
$G$-module $\zz$ we have $|H^1(G, \wi \zz)| = 2$, as
\eqref{longexact} shows, and each connecting homomorphism
$\delta\colon \wh H^{i-1}(G,A) \to\wh  H^i(G, \wi A)$ is 
given by the cup product with the nontrivial element of 
$H^1(G, \wi \zz)$.

We will invoke the twisted cohomology typically in the following
setting:  Let $F \subseteq K \subseteq M$ be fields with 
$[K\col F] = 2$, and $M$ Galois over $F$.  Let $G = \Gal(M/F)$
and $H = \Gal(M/K)$, which is a closed subgroup of 
$G$ of index $2$.  Then, $M^*$ is a discrete $G$-module, and 
$\wi{M^*}$ denotes $M^*$ with the twisted $G$-action 
relative to $H$ 
described above.  Recall that $\Br(M/K;F)$ denotes the subgroup of 
$\Br(M/K)$ consisting of classes of central simple $K$-algebras
split by $M$ and having a unitary $K/F$-involution.

\begin{prop}\label{relbriso} 
$H^2(G,  \wi{M^*}) \cong \Br(M/K; F)$.
\end{prop}

\begin{proof}
Part of the long exact sequence \eqref{longexact} is 
\begin{equation}\label{H2seq}
H^1(G, M^*)\ \lra \ H^2(G,\wi{M^*})\ \ \lra \ \ H^2(H,M^*)
\  \ \overset \cors\lra \  \ H^2(G,M^*)
\end{equation}  
By Albert's theorem \cite[Th.~3.1(2)]{kmrt}, for $[A] \in \Br(M/K)$,
the algebra $A$ has a $K/F$-involution iff $\cors_{K\to F}(A)$
is split.  Thus, in the isomorphism $\Br(M/K)\cong
H^2(H,M^*)$, $\Br(M/K;F)$ maps isomorphically to 
${\ker\big(H^2(H,M^*) \overset \cors \lra H^2(G,M^*)\big)}$.
Because $H^1(G, M^*) = 0$ by the homological Hilbert 90, 
the exact sequence \eqref{H2seq} above yields the desired 
isomorphism.
\end{proof}

\begin{remark}\label{Hiformulas}
Here are formulas for $\wh H^i(G, \wi{M^*})$ for small
$i$, which are easily derived from 
standard group cohomology formulas and 
\eqref{longexact} above.  We assume $[M\col K]< \infty$, and 
let $\theta$ be any element of $G\setminus H$.  So, 
$\Gal(K/F) = \{\id_K,\theta|_K\}$. We write $b^{1-\theta}$ for 
$b/\theta(b)$.
\begin{align*}
\quad\text{(i)}& & H^1(G, \wi{M^*})  \ &\cong \  
F^*\big/N_{K/F}(K^*) 
 \ \cong \  \wh H^0(\Gal(K/F), K^*).\\
\text{(ii)} & &
H^0(G, \wi{M^*}) \ &\cong \  \{c\in K^*\mid N_{K/F}(c) = 1\}.\\
\text{(iii)}& &
\wh H^0(G,\wi{M^*}) \ &\cong \ 
\{ c\in K^*\mid N_{K/F}(c) = 1\}
\, \big/ \, \{N_{M/K}(m)^{1-\theta}\mid m \in M^*\} 
\qquad\qquad\qquad\qquad\qquad\qquad\\
& & & = \ 
\{ b^{1-\theta}\mid b\in K^*\} \, \big/ \, 
\{N_{M/K}(m)^{1-\theta}\mid m \in M^*\}. 
\end{align*}
We will be working particularly with $\wh H^{-1}(G, \wi{M^*})$.
For this, let $\wi N\colon\wi{M^*} \to K^*$ be given by 
$$
\wi N(m) \,= \ \textstyle \prod \limits_{g\in G}g*m 
\ = \  \prod\limits_{h\in H}h(m) \cdot (\theta h)(m)\inv \ 
= \ N_{M/K}(m)\big/\theta(N_{M/K}(m)).
$$
So, $\wi N$ is the norm map for $\wi{M^*}$ as a $G$-module. 
Note that 
\begin{equation}\label{kerNtilde}
\ker(\wi N) \, = \ \{m\in M^*\mid  N_{M/K}(m) \in F^*\}.
\end{equation}
Also, let 
\begin{equation}\label{IsubG}
I_G(\wi {M^*})\, = \ \big\lan \,(g*m)m\inv\mid m\in M^*, 
g\in G\big\ran \ = \ \big\lan\,h(m)/m, \, h\theta(m) \,m\mid
m\in M^*, h\in H\big\ran. 
\end{equation}
Then, by definition, 
\begin{equation}\label{Hhat-1}
\wh H^{-1}(G, \wi{M^*}) \,\cong \,
\ker(\wi N) \big/ I_G(\wi {M^*}).
\end{equation}

\end{remark}

In the following useful lemma, part (ii) is an abstraction of an 
argument of Yanchevski\u\i \ \cite[proof of~Cor.~4.13]{y}. 

\begin{lemma}\label{Hdihedral}
Let $D$ be a finite dihedral group,
i.e., $D = \langle h,\theta\rangle$ 
where $\theta^2 = 1$, $\theta\ne 1$, and 
$\,\theta h\theta\inv = h\inv$,
and $h$ has finite order. 
 Let~$H = \langle h \rangle$.  Let $A$ be a $D$-module 
such that $H^1(H,A) = 0$ and $H^1(\langle \theta \rangle,
A^H) = 0$.  Let ${A^\theta = \{ a\in A \mid 
\theta\cdot a = a\}}$
and $N_H(a) = \sum_{h\in H} h\ccdot a$. Then, 
\begin{enumerate}[\upshape (i)]
\item 
$A^H + A^{\theta} \,=\, \{ a\in A\mid a - \theta \ccdot a \in A^H\}$.
\item
$A^\theta + A^{h\theta} \,=\,
 \{a\in A\mid N_H(a) \in A^\theta\}\, = \, A^\theta + A^{\theta h}$.
\item 
The map $\cors_{\langle \theta \rangle \to D} \times 
\cors_{\langle h\theta \rangle \to D}\colon
\wh H^{-1}(\langle \theta \rangle, \wi A) \times
\wh H^{-1}(\langle h\theta\rangle, \wi A) 
\to \wh H^{-1}(D, \wi A)$ is surjective.
\end{enumerate}
\end{lemma}

\begin{proof} (i)
We have the short exact sequence of $\lan \theta \ran$-modules
$0 \to A^H \to A \to A/A^H \to 0$.  Since 
$H^1(\lan\theta\ran,A^H) = 1$, the long exact cohomology sequence
shows that $A^\theta$ maps onto $(A/A^H)^\theta$, which yields (i).

(ii) Note that for $a\in A$, $N_H(\theta\!\ccdot\! a) = 
\sum_{k\in H}(k\theta)\ccdot a
= \sum _{k\in H}(\theta k\inv) \ccdot a = \theta \ccdot N_H(a)$. 
 The left inclusion $\subseteq$ in~(ii) follows immediately.  For the 
inverse inclusion, take $a\in A$ with $N_H(a) \in A^\theta$.  
Then, ${N_H(a-\theta\ccdot a) = 
N_H(a) - \theta\ccdot N_H(a) = 0}$.  
Since $H^1(H,A) = 0$, with $H = \langle h \rangle$, there is 
$c\in A$ with\break 
${a-\theta\ccdot a = c- h\ccdot c}$.  So,
\begin{align*}
0 \  &=  \ a - \theta\ccdot a + \theta\ccdot(a - \theta\ccdot a) \ = \ 
c- h\ccdot c + \theta\ccdot c - (\theta h) \ccdot c \\
& =  \ c-h\ccdot c + (h\theta h)\ccdot c -(\theta h)\ccdot c  \ 
= \ [c- (\theta h)\ccdot c] - h\ccdot [c- (\theta h)\ccdot c],
\end{align*}
i.e., $c- (\theta h)\ccdot c\in A^H$. Since the group action
of $\lan \theta h \ran$ on $A^H$ coincides with the 
action of $\lan \theta  \ran$ on $A^H$, we have 
$H^1(\lan \theta h \ran, A^H) \cong H^1(\lan \theta  \ran,
A^H) = 0$. Therefore, part (i) applies, with $\theta h$
replacing $\theta$. Thus, we can write $c = d+e$
with $d\in A^H$ and $e\in A^{\theta h}$, hence 
$\theta \ccdot e = h \ccdot e= (h\theta)\ccdot(\theta\ccdot e)$. 
 Now, as $d = h\ccdot d$, 
$$
a - \theta \ccdot a \ = \ c- h\ccdot c  \ = \ e- h\ccdot e 
\ = \ e - \theta \ccdot e,
$$
showing that $a+ \theta\ccdot e \in A^\theta$.  Thus, 
$a = [a +\theta \ccdot e] -\theta \ccdot e \in 
A^\theta + A^{h\theta}$,
completing the proof of  the first equality in (ii).  Since 
$\theta h = h\inv \theta$, the second equality in
(ii) follows from the first by replacing $h$ by $h\inv$.

(iii) We have $\wh H\inv(\langle \theta \rangle, \wi A)
\cong A^\theta\big/ \{a+ \theta \ccdot a\mid a\in A\}$,
$\wh H\inv(\langle h\theta \rangle, \wi A)
\cong A^{h\theta}\big/ \{a+ (h\theta) \ccdot a\mid a\in A\}$, 
and 
$$
\wh H\inv(D, \wi A) \ \cong \ \{a\in A \mid 
N_H(a) \in A^\theta\} \,\big/\, \langle a-k\ccdot a, \ 
a+(k\theta)\ccdot a \mid a\in A, \ k\in H\rangle.
$$ 
The map $\cors_{\langle \theta \rangle \to D}\colon 
\wh H\inv (\langle \theta\rangle, \wi A) 
\to \wh H\inv(D, \wi A)$ arises from the inclusion 
$A^\theta  \hookrightarrow \{ a\in A\mid N_H(a) \in A^\theta\}$; 
likewise for $\cors_{\langle h\theta \rangle \to D}\colon 
\wh H\inv (\langle h\theta\rangle, \wi A) 
\to \wh H\inv(D, \wi A)$.  Thus, the surjectivity 
asserted in part (iii) is immediate from part (ii).
\end{proof}
\begin{proposition}\label{gendihedralprop}
Let $F\subseteq K \subseteq M$ be fields with $[M\col F] <\infty$
and $M$ a $K/F$-generalized dihedral exten-sion.  Let
$G = \Gal(M/F)$ and $H = \Gal(M/K)$.  Take any 
$\theta \in G\setminus H$.  Then there is an exact sequence:
\begin{equation}\label{gendihedralexact}
\textstyle\prod\limits_{h\in H}\wh H\inv(\langle h\theta
\rangle , \wi {M^*})  \ \lra \ \wh H\inv (G, \wi {M^*})  \ 
\lra
 \ \ker(\wi N) \, \big/ \,\Pi  \ \lra \ 1
\end{equation}
where $\ker(\wi N) = \{ m\in M^*\mid N_{M/K}(m) \in F^*\}$
and $\Pi = \prod_{h\in H}M^{*h\theta}$.  In particular, if 
$M/K$ is cyclic Galois, then $\ker(\wi N) /\Pi = 1$. 
\end{proposition}

\begin{proof}
Here, $M^{*h\theta} = \{ m\in M^*\mid h\theta(m) = m\}$.    
We have 
 $\wh H\inv(G, \wi{M^*}) \cong \ker(\wi N) 
\big/I_G(\wi {M^*})$ as in \eqref{kerNtilde}--\eqref {Hhat-1}.
For any $h\in H$ and $m \in M^*$,
$$
m/h(m)\, = \, [m\cdot \theta(m)]\big/[\theta(m) \cdot h(m)] \, = \,
[m\cdot \theta(m)]\big/[\theta(m) \cdot h\theta(\theta(m))] \, \in 
M^{*\theta} M^{*h \theta}
$$
and $m\ccdot h\theta(m)\in M^{*h\theta}$.  Hence, by \eqref{IsubG},
\begin{equation} \label{IGeq}
I_G(\wi{M^*}) \,\subseteq \, 
\textstyle\prod\limits_{h\in H}M^{*h\theta} \, = \,\Pi.
\end{equation}
  Thus, 
there is a well-defined epimorphism $\zeta\colon 
\wh H\inv(G,\wi{M^*})\to \ker(\wi N) 
\, \big/ \, \Pi$,  with
${\ker(\zeta) = \Pi\big/I_G(\wi{M^*})}$.
Now, for $h\in H$, we have $\wh H\inv(\langle h\theta\rangle,
\wi{M^*}) \cong M^{*h\theta}\big/N_{M/M^{\langle h\theta
\rangle}}(M^*)$. So, $\prod_{h\in H} 
\wh H\inv(\langle h\theta \rangle, \wi{M^*})$ clearly maps
onto $\ker(\zeta)$, proving the exactness of 
\eqref{gendihedralexact}. If $H$ is cyclic, then $G$ is dihedral, and 
$\ker(\wi N) = \Pi$ by Lemma~\ref{Hdihedral}(ii). 
\end{proof}

\begin{remark}
In the context of Prop.~\ref{gendihedralprop}, suppose 
$H = \langle h_1, \ldots, h_m \rangle$. Then, the following
lemma shows that  
\begin{equation}\label{productongens}
\textstyle \prod\limits_{h\in H}M^{*h \theta} \ = \ 
\prod\limits_{(\varepsilon_1, \ldots , \varepsilon_m)
\in \{0,1\}^m} M^{*\,h_1^{\varepsilon_1}\ldots h_m^{\varepsilon_m} 
\theta}, 
\end{equation}
so the left term in \eqref{gendihedralexact} could be 
replaced by $\prod\limits_{(\varepsilon_1, \ldots , \varepsilon_m)
\in \{0,1\}^m} \wh H\inv(\langle h_1^{\varepsilon_1}\ldots 
h_m^{\varepsilon_m}\theta \rangle, \wi {M^*})$. One can see by looking
at examples that the product on the right in \eqref{productongens}
is minimal in that if we delete any of the terms in that product, then
the equality no longer holds in general.

\end{remark}

\begin{lemma}\label{generators}
Let $G = \lan H, \theta \ran$ be a generalized dihedral 
group, where $H$ is an abelian subgroup of $G$ with 
$|G\col H| = 2$, $\theta$ has order $2$, and 
$\theta h \theta = h\inv$ for all $h\in H$. 
Let $A$ be any $G$-module.  Suppose $H = 
\lan h_1, \ldots, h_m\ran$.  Then, 
$$
\textstyle \sum\limits_{h\in H} A^{h\theta} \ = \ 
\sum\limits_{(\varepsilon_1, \ldots \varepsilon_m) \in 
\{0,1\}^m} A^{h_1^{\varepsilon_1}\ldots h_m^{\varepsilon_m}\theta}.
$$
\end{lemma}

\begin{proof}
This follows from \cite[Lemma~4.9]{I} (with $A$ for $U$, 
$H$ for the abelian group $A$ and $W_h = A^{h\theta}$
for all $h\in H$), once we establish that 
$A^{h\theta} \subseteq A^{k\theta} + A^{k^2h\inv \theta}$
for all $h,k\in H$. For this, take any $a\in A^{h\theta}$.  
Then $\theta(a) = h\inv(a)$.  Hence, 
$k^2h\inv\theta(k\theta(a)) = k^2h\inv k\inv(a) = k\theta(a)$,
showing that $k\theta(a) \in A^{k^2h\inv\theta}$.  
Thus $a = [a+ k\theta(a)] - k\theta(a) \in A^{k\theta}
+  A^{k^2h\inv \theta}$, proving the required inclusion. 
\end{proof}


\section{Unitary relative Brauer Groups, bicylic case}\label{ubicyclic}

In  this section we prove a unitary version of   
the formula 
$\Br(M/K)\big/\Dec(M/K) \cong \widehat{H}^{-1}(\Gal(M/K),M^*)$,
for $M$  a bicyclic Galois extension of $K$, see \eqref{njnj} above. The unitary version was 
inspired by the result of Yanchevski\u\i~ \cite[Prop.~5.5]{y}, which was a 
key part of his proof in~\cite[Th.~A]{yinverse} that any finite abelian 
group can be realized as the unitary $\SK$ of some division algebra with 
involution of the second kind.  

Let $F\subseteq K \subseteq M$ be fields with 
$[K\col F] = 2$ and $K$ Galois over $F$, and 
$M = L_1\otimes_KL_2$ with each $L_i$ cyclic
Galois over $F$.  
 Assume $M$ is $K/F$-generalized
dihedral, 
as described at the beginning of \S \ref{unitaryIN}.
Let $G = \Gal(M/F)$ and $H = \Gal(M/K)$, and
choose and fix an element $\theta\in G\setminus H$.  So,
$\Gal(K/F) = \{\theta|_K, \id_K\}$.
 To simplify notation, 
let $\si$ (not $\si_1$) be a fixed generator of $\Gal(M/L_2)$,
and  $\rho$ (not $\si_2$) a fixed generator of $\Gal(M/L_1)$;
so, $H = \lan\si\ran \times \lan\rho\ran$.
Let $n = [L_1\col K]$, which is the order of $\si$ in $H$,
and let $\ell = [L_2\col K]$, which is the order of $\rho$.
 As in Prop.~\ref{gendihedralprop}, let 
$$
\ker(\widetilde{N}) \,= \,\{ a \in M^*\mid N_{M/K}(a)\in F^* \}
$$
and
\begin{equation}\label{nott}
\textstyle{\Pi}\ = \ \textstyle{\prod}_{h\in H}M^{*h\theta}
 \, = \ M^{*\theta} M^{*\rho\theta} M^{*\sigma\theta} M^{*\rho\sigma\theta}.
\end{equation}
(See \eqref{productongens} for the second equality.)

\begin{proposition}\label{thmainm}
We have  
$$
 \Br(M/K;F)\big/
\Dec(M/K;F) \,\cong \,\ker(\widetilde{N})/\Pi.
$$
\end{proposition}

\begin{proof} This follows by combining the formulas for unitary $\SK$
given in \cite[Prop.~5.5]{y} with the Henselian version of 
the formula in \cite[Cor.~4.11]{I}.
However, we give a direct proof avoiding the use of 
Yanchevski\u\i's  special unitary conorms, since we 
will later need
an explicit description of the isomorphism.

Define a map 
$$
\Psi\colon\Br(M/K;F) \longrightarrow \ker(\widetilde N)/\Pi
$$ 
as follows: By 
Lemma~\ref{unitarycp}, a Brauer class in $\Br(M/K;F)$ is 
represented by an algebra $A=A(u,b_1,b_2)$, where
$u,b_1,b_2$ satisfy the conditions in \eqref{bicyclicbrel} and  
$b_1 \in L_2^{*\theta}, b_2 \in L_1^{*\theta}$, and 
$u\,\rho\sigma\theta(u)=1$. By Hilbert 90 (for the group 
$\langle \rho \sigma \theta \rangle$), there is $q \in M^*$ 
with $u=q /\rho \sigma \theta(q)$. Define
$$
\Psi\big(A(u,b_1,b_2)\big)\, =\ q\,\Pi\,\in \,
\ker(\widetilde N)/\Pi.
$$ 
We will show that $\Psi$ is a well-defined, surjective 
homomorphism with kernel $\Br(L_1/K;F) \Br(L_2/K;F)$, which equals
$\Dec(M/K;F)$ (see \eqref{Decformula}).

For the well-definition of $\Psi$, first note that 
$$ 
1\, =\, N_{M/K}(u)\, =\, N_{M/K}(q/\rho\sigma \theta(q))\, =\, 
N_{M/K}(q)\big/N_{M/K}(\theta(q))\, =\,  
N_{M/K}(q)\big/\theta(N_{M/K}(q)),
$$ 
so, $q\in \ker(\widetilde N)$. Also, given $u$, the choice 
of $q$ with $q/\rho \sigma \theta(q)=u$ is unique up to a 
multiple in~$M^{*\rho \sigma \theta}$. Since  
$M^{*\rho \sigma \theta}\subseteq \Pi$, 
$\ \Psi\big(A(u,b_1,b_2)\big)$ is independent of the choice 
of $q$ from $u$. Now, suppose 
${A(u,b_1,b_2)\cong A(u',b_1',b_2')}$, with $u,b_1,b_2$ and 
$u',b_1',b_2'$ each satisfying the conditions of 
Lemma~\ref{unitarycp}(iii). We have the presentation
 $A(u,b_1,b_2)=\bigoplus_{i=0}^{n-1}\bigoplus_{j=0}^{\ell-1}Mx^iy^j$, 
where 
$\intt(x)|_M=\si$, $x^n = b_1$,
 $\intt(y)|_M=\rho$, $y^\ell = b_2$, and $xyx^{-1}y^{-1}=u$, so,  
 (see \eqref{bicyclicbrel})
\begin{equation}\label{bicyclicbrels}
b_1\in M^{\lan\si\ran}\,  =\,  L_2, 
\quad b_2\in M^{\lan\rho\ran} \, =\,  L_1,
\quad N_{M/L_2}(u) \, =\,  b_1/\rho(b_1), 
\quad N_{M/L_1}(u) \, =\,  \si(b_2)/b_2.
\end{equation}
The conditions of  Lemma~\ref{unitarycp}(iii)  
we are also assuming
are that   
\begin{equation}\label{unitaryeqs}  
b_1 \,\in \, L_2^{\theta},\qquad b_2\in L_1^{\theta}, 
\qquad\text{and} \qquad u\,\rho\si\theta(u) \,=\, 1.
\end{equation}
The corresponding conditions in \eqref{bicyclicbrels} and  
\eqref{unitaryeqs} hold for $b_1'$, $b_2'$ and $u'$.
By 
Lemma~\ref{unitarycp}, 
there is a $K/F$-involution $\tau$ of $A=A(u,b_1,b_2)$ with 
$\tau|_M=\theta, \tau(x)=x, \tau(y)=y$. We have an
isomorphism 
${A(u,b_1,b_2)\cong A(u',b_1',b_2')}$, and by Skolem-Noether 
there is such an isomorphism which restricts to the identity
on $M$. Therefore, 
 there exist $x'$ and 
$y'$ in $A^*$  such that
$\intt(x')|_M = \si$, 
$x'^n=b_1'$, $\intt(y')|_M = \rho$, $y'^\ell=b_2'$, 
and $x'y'x'^{-1}y'^{-1}=u'$.  Since $\intt(x')|_M = \intt(x)|_M$
there is $c_1\in C_A(M)^* = M^*$ with $x' = c_1 x$, and likewise
$c_2\in M^*$ with $y'= c_2 y$. By simplifying the expressions
$b_1' = (c_1x)^n$, $b_2' = (c_2y)^\ell$, and $u'=(c_1x)(c_2y) (c_1x)\inv
(c_2y)\inv$, we find that
\begin{equation}\label{prime}
b_1'\, = \, N_{M/L_2}(c_1) \,b_1, \qquad b_2' \,=\, N_{M/L_1}(c_2) \,b_2, 
\qquad u' \, =\, \big(c_1/\rho(c_1)\big) \big(\si(c_2)/c_2\big) \,u. 
\end{equation}
By Lemma~\ref{unitarycp}, there is a 
$K/F$-involution $\tau'$ on $A$ with $\tau'(x')=x', \tau'(y')=y'$, and 
$\tau'|_M=\theta$. Since $\tau'\tau\inv$ is a $K$-automorphism of $A$, 
there exists  $e\in A^*$ with $\tau'=\intt(e)\tau$. Because 
$\tau'|_M=\tau|_M$, $e \in C_A(M)=M$. The condition that 
$\tau'^2=\id_A$ implies that $e/\theta(e)\in K^*$. Since 
$e/\theta(e)\big(\theta (e/\theta(e))\big)=1$, Hilbert 90 
for $K/F$ shows 
that there is $d\in K^*$ with $d/\theta(d)=e/\theta(e)$. By replacing 
$e$ by $e/d$, we may assume that $\theta(e)=e$. The conditions that 
$c_1 x=\tau'(c_1 x)=\intt(e)\tau(c_1x)$ and 
$c_2 y=\tau'(c_2 y)=\intt(e)\tau(c_2 y)$ yield
 $$
c_1\,=\,\si\theta(c_1)\,e/\si(e) \qquad\text{and}\qquad 
c_2\,=\,\rho\theta(c_2)\,e/\rho(e),
$$  
hence,
\begin{equation}\label{rhoc1} 
\rho(c_1)\,=\,\rho\si\theta(c_1)\,\rho (e)/\rho\si(e) \qquad\text{and}\qquad 
\si(c_2)\,=\,\rho\si\theta(c_2)\,\si(e)/\rho\si(e).
\end{equation}
The equations \eqref{rhoc1} yield 
 \begin{equation}\label{t6}
 c_1/\rho(c_1)\, =\,
\big(c_1/\rho\sigma\theta(c_1)\big)\big(\rho\si(e)/\rho(e)\big) 
\qquad\text{and} \qquad
 \si(c_2)/c_2\,=\,
\big(\rho\sigma\theta(c_2)/c_2\big)\big(\si(e)/\rho\sigma(e)\big).
  \end{equation}
  Let $\widetilde{q}=(c_1/c_2)\si(e)$. Then, 
using (\ref{t6}),  \eqref{prime} and $\theta(e) = e$,
\begin{align}\label{t7}
\begin{split}
\widetilde{q}/\rho\sigma\theta(\widetilde{q}) \ &= \ 
\big(c_1/\rho\si\theta(c_1)\big)\big(\rho\si\theta(c_2)/c_2\big)
\big(\si(e)/\rho\sigma\theta\si(e)\big)\\
&= \ 
\big(c_1/\rho(c_1)\big)\,\big(\rho(e)/\rho\si(e)\big)\,
\big(\si(c_2)/c_2\big)\, \big(\rho\sigma(e)/\si(e)\big)\, 
\big(\si(e)/\rho\theta(e)\big)\\ 
&= \ 
\big(c_1/\rho(c_1)\big)\,\big(\si(c_2)/c_2\big) \ = \ u'/u.
\end{split}
\end{align}
When $q \in M^*$ is chosen so that $q/\rho\sigma\theta(q)=u$, 
set $q'=\widetilde{q}q$; then (\ref{t7}) shows 
that $q'/\rho\sigma\theta(q')=u'$. We check that $\wi q \in \Pi$: 
We have (see~(\ref{prime}) and \eqref{unitaryeqs})   
$N_{M/L_2}(c_1)=b_1'/b_1 \in L_2^{*\theta}$. Therefore, by 
Lemma~\ref{Hdihedral}(ii) applied to the dihedral group 
$\lan \sigma, \theta\ran = \Gal(M/L_2^\theta)$,  
$c_1 \in M^{*\theta} M^{*\sigma \theta}\subseteq \Pi$. 
Likewise, $c_2 \in M^{*\theta} M^{*\rho \theta}\subseteq \Pi$ 
as 
$N_{M/L_1} (c_2)=b_2'/b_2 \in L_1^{*\theta}$. Finally, since 
$\theta(e)=e$, we have $\si(e)=\si\theta(e)=
\si\theta\si^{-1}(\si(e))=\si^2\theta(\si(e))$. So, 
$\si(e)\in M^{*\si^2\theta} \subseteq \Pi$. 
Thus, ${q' \equiv q \pmod{\Pi}}$,   
which shows that $\Psi$ is well-defined independent of the 
choice of presentation of $A$ as $A(u,b_1,b_2)$ with 
$u,b_1,b_2$ as in Lemma~\ref{unitarycp}(iii).  

For the surjectivity of  $\Psi$, take any 
$q\in \ker(\widetilde{N})$ and set $u=q/\rho\sigma\theta(q)$. 
So, $u\,\rho\si\theta(u) = 1$.  Furthermore,
 as $N_{M/K}(q)\in F^*$,
$$
N_{M/K}(u)\,=\,N_{M/K}(q)/N_{M/K}(\rho\sigma\theta(q))\,=\,
N_{M/K}(q)/\theta(N_{M/K}(q))\,=\,1.
$$
Since $N_{L_2/K}(N_{M/L_2}(u))=N_{M/K}(u)=1$, by Hilbert 90 
for $L_2/K$ there is $b_1 \in L_2^*$ with 
${b_1/\rho(b_1)=N_{M/L_2}(u)}$. Then, 
$$
b_1/\rho(b_1)\,=\,N_{M/L_2}(q)\big/N_{M/L_2}(\rho\sigma\theta(q))\,=\,
N_{M/L_2}(q)/\rho\theta(N_{M/L_2}(q)).
$$ 
Hence, 
$$
1\,=\,(b_1/\rho(b_1))\,\rho\theta(b_1/\rho(b_1))\,=\,
(b_1/\theta(b_1))/\rho(b_1/\theta(b_1)),
$$
which shows that $b_1/\theta(b_1) \in L_2^\rho = K$. 
By Lemma~\ref{Hdihedral}(i) applied to the dihedral group
${\Gal(L_2/F) = \lan\rho|_{L_2},\theta|_{L_2}\ran}$, it follows that 
 $b_1=k\widehat{b_1}$ with $k\in K^*$ and 
$\widehat{b_1} \in L_2^{*\theta}$. By replacing $b_1$ with 
$\widehat{b_1}$, we may assume that $b_1 \in L_2^{*\theta}$. 
Likewise, there is $b_2 \in L_1^{*\theta}$ with 
$N_{M/L_1}(u^{-1})=b_2/\si(b_2)$. Then, as $u,b_1,b_2$ 
satisfy the conditions of \eqref{bicyclicbrel} (where $\si_1 = \si$ 
and $\si_2=\rho$) the algebra $A(u,b_1,b_2)$ exists, and 
by Lemma~\ref{unitarycp}
$[A(u,b_1,b_2)] \in \Br(M/K;F)$.  Clearly, 
$\Psi[A(u,b_1,b_2)]=q\,\Pi$. 

Finally, we determine $\ker(\Psi)$: 
If $[B] \in \Br(L_1/K;F)$ then we can assume that $B$ has 
$L_1$ as a maximal subfield. Then, by Lemma~\ref{unitarycp},
 $B\cong (L_1/K,\sigma,b_1)$, where $b_1 \in K^{*\theta} = F^*$. 
Likewise, for 
any ${[C] \in \Br(L_2/K;F)}$, we have 
$C \thicksim (L_2/K,\rho,b_2)$ for some $b_2\in F^*$. Then, 
$$
\big[B\otimes_K C\big]=\big[(L_1/K,\sigma,b_1)\otimes_K 
(L_2/K,\rho,b_2)\big]=\big[A(1,b_1,b_2)\big] \in \ker(\Psi),
$$
since when $u = 1$ we can take $q = 1$.
So $\Br(L_1/K;F)\Br(L_2/K;F) \subseteq \ker(\Psi)$. For the 
reverse inclusion, take any 
$A=A(u,b_1,b_2)$ with $[A] \in \ker(\Psi)$. Since 
$[A]\in \Br(M/K;F)$, by Lemma~\ref{unitarycp} we may assume that 
$b_1 \in L_2^{*\theta}, b_2 \in L_1^{*\theta}$ and 
$u\,\rho\sigma\theta(u)=1$. Since, $[A]\in \ker(\Psi)$, we have 
$u=q/\rho\sigma\theta(q)$ with $q\in \Pi$, so
$q=q_\theta q_{\rho\theta} q_{\si\theta}q_{\rho\sigma\theta}$, where 
$q_\theta \in M^{*\theta}, q_{\rho\theta} \in M^{*\rho \theta}$, 
$q_{\si\theta}\in M^{*\si\theta}$, and $q_{\rho\si\theta}\in 
M^{*\rho\si\theta}$. Thus, 
\begin{align*} 
u\ &= \ q/\rho\sigma\theta(q)\ = \ 
\big(q_\theta/\rho\sigma(q_\theta)\big)
\big(q_{\rho\theta}/\sigma(q_{\rho\theta})\big)
\big(q_{\sigma\theta}/\rho(q_{\sigma\theta})\big)\\
&= \ \big(q_\theta q_{\rho\theta}/\si(q_\theta q_{\rho \theta})
\big)\big(q_{\sigma\theta}\sigma(q_\theta)/
\rho(q_{\sigma\theta}\sigma(q_\theta)\big)
 \ = \ \big(c_2/\si(c_2)\big)\big(\rho(c_1)/c_1\big),
\end{align*}
where $c_2=q_\theta q_{\rho\theta}$ and 
$c_1=(q_{\sigma\theta}\sigma(q_\theta))\inv$. Then 
by~(\ref{ghgt}), $A =A(u,b_1,b_2) \cong A(u',b_1',b_2')$ 
where\break 
${u'=(c_1/\rho(c_1))(\si(c_2)/c_2)u=1}$, and 
$b_1'=N_{M/L_2}(c_1)b_1$ and $b_2'=N_{M/L_1}(c_2)b_2$.
Since $c_2\in M^{*\theta}M^{*\rho\theta}$, 
an easy calculation or an application of 
Lemma~\ref{Hdihedral}(ii)
for the dihedral group $\Gal(M/L_1^\theta) = \lan \rho, \theta\ran$
shows\break that $N_{M/L_1}(c_2) \in L_1^{*\theta}$. 
Therefore, $b_2'= N_{M/L_1}(c_2)b_2
\in L^*_1{^\theta}$, as 
$b_2\in L_1^{*\theta}$. 
But also, as in \eqref{bicyclicbrels},\break 
${\sigma(b_2')/b_2'=N_{M/L_1}(u')=N_{M/L_1}(1)=1}$. Hence, 
$b_2' \in L_1^{*\theta}\cap L_1^{*\sigma}=K^{*\theta}=F^*$. 
Likewise, as $q_{\si\theta} \in M^{*\theta}$ and $\si(q_\theta)
\in M^{*\si^2\theta} \subseteq M^{*\theta}M^{*\si\theta}$ (see  
\eqref{productongens}), we have $c_1 \in M^{*\theta}M^{*\si\theta}$.
Therefore, an
 easy calculation or Lemma~\ref{Hdihedral}(ii)
for the dihedral group $\Gal(M/L_2^\theta) = \lan\si,\theta\ran$
shows that  $N_{M/L_2}(c_1)\in L_2^{*\theta}$.  
So, arguing just as for $b_2'$, we find that $b_1' \in F^*$. Thus, 
$$
A \ \cong \  A(1,b_1',b_2') \ \cong \ 
(L_1/K,\sigma,b_1')\otimes_K(L_2/K,\rho,b_2'),
$$ 
  and since the $b_i' \in F^*$, 
$\big[(L_1/K,\sigma,b_1')\big] \in \Br(L_1/K;F)$ and 
$\big[(L_2/K,\rho,b_2')\big] \in \Br(L_2/K;F)$, by 
Lemma~\ref{unitarycp}. Thus, $\ker(\Psi)=\Br(L_1/K;F)\Br(L_2/K;F)
=\Dec(M/K;F)$. 
\end{proof}

This yields our unitary analogue to \eqref{njnj} above.

\begin{proposition} \label{unitarybicyclic}
 For $M$ bicyclic Galois over $K$ with $M$ 
$K/F$-generalized dihedral, setting $G = \Gal(M/F)$, $H=\Gal(M/K)$,
and $\theta$ any element of $G\setminus H$ as above, 
there is an exact sequence
\begin{equation}
\textstyle \prod\limits_{h\in H} \wh H\inv(\lan h\theta\ran, 
\wi{M^*}) \ 
 \lra \  \wh H\inv(G, \wi {M^*})  \ \lra  \ 
\Br(M/K;F)/\Dec(M/K;F) \  \lra\, 0
\end{equation}
\end{proposition}

\begin{proof}
This follows  from Prop.~\ref{thmainm} and Prop.~\ref{gendihedralprop}.
\end{proof}


\section{Semiramfied algebras}\label{semiram} 

We now apply the results of the preceding sections 
to the calculation of unitary $\SK$ for semiramified 
graded division algebras with graded $\sTR$-involution
Throughout this section, fix a graded field $\sT$
and a graded subfield $\sR$ of $\sT$ with 
$[\sT\col\sR] = 2$ and $\sT$ Galois over $\sR$, say with 
$\sGal(\sT/\sR) = \{\id,\psi\}$.  Assume further that 
$\sT$ is inertial over $\sR$.  Thus, $\Ga_\sT = \Ga_\sR$, 
$[\sTz\col\sRz] = 2$, $\sTz$ is Galois over with $\Gal
(\sTz/\sRz) = \{ \id , \psi_0\}$, where $\psi_0 = \psi|_\sTz$,
and $\psi = \psi_0\otimes \id_\sR$ when we identify 
 $\sT$ with $\sTz \otimes_\sRz \sR$.  
By definition, for a central simple graded division algebra 
$\sB$ over $\sT$ with a graded unitary $\sTR$-involution $\tau$,
 the unitary 
$\SK$ is given by 
$$
\SK(\sB, \tau) \ = \ \Sigma_\tau'(\sB) \,\big / \, 
\Sigma_\tau(\sB), 
$$
where 
$$
\Sigma_\tau'(\sB) \ = \ \{b\in \sB^*\mid \Nrd_\sB(b) \in \sR\}\qquad\text{and}\qquad
\Sigma_\tau(\sB)\ = \ \big\lan \{ b\in \sB^*|\ \tau(b) = b\}\big\ran
$$
We are assuming that $\sTR$ is 
inertial because otherwise $\sTR$ is totally ramified and  $\SK(\sB,\tau) = 1$, 
by \cite[Th.~4.5]{I}. It is known by \cite[Lemma~2.3(iii)]{I} that 
$[\sB^*,\sB^*] \subseteq \Sigma _\tau(\sB)$, so $\SK(\sB,\tau)$
is an abelian group.  Also, if $\tau'$ is another graded 
$\sTR$-involution on $\sB$, then 
$\Sigma_{\tau'}'(\sB) = \Sigma_\tau'(\sB)$
and $\Sigma_{\tau'}(\sB) = \Sigma_\tau(\sB)$, 
so $\SK(\sB,\tau') =\SK(\sB,\tau)$. The easy proof is analogous
to the ungraded proof given in \cite[Lemma~1]{yin}.  

Let $\sE$ be a semiramified 
$\sT$-central graded division algebra.  So, as we have seen,
$\sEz$ is a field abelian Galois over $\sTz$, and 
$\ov\Theta_\sE \colon \Ga_\sE/\Ga_\sT \to 
\Gal(\sEz/\sTz)$ is a canonical isomorphism.
Suppose $\sE$ has a graded $\sTR$-involution $\tau$;
so $\tau|_\sTz = \psi_0$.  We have seen in Prop.~\ref{uINdecomp} 
that $\sEz$ is then a $\sTz/\sRz$-generalized dihedral Galois
extension.  Let $H = \Gal(\sEz/\sTz)$ and $G = \Gal
(\sEz /\sRz)$, and let $\ov \tau = \tau|_{\sEz} 
\in G\setminus H$.  For each $\ga\in \Ga_\sE$ choose and
fix $x_\ga \in \sE_\ga$ with $x_\ga\ne 0$ and $\tau(x_\ga)
= x_\ga$.  (Such $\xg$ exist, by \cite[Lemma~4.6(i)]{I}.)  Our starting
point is the formula proved in \cite[Th.~4.7]{I}
\begin{equation}\label{USK1PiX}
\SK(\sE, \tau) \ \cong \ \big(\Sigma_\tau(\sE)' \cap \sEz^*\big)\, \big/
\big(\Sigma_\tau(\sE) \cap \sEz^*\big) \ = \ 
\ker(\wi N) \big / 
\big(\, \Pi \cdot X\,\big),
\end{equation}
 where  
\begin{align*}
\ker(\wi N) \ &= \ \{ a\in \sEz^* 
\mid N_{\sEz/\sTz}(a) \in \sRz\}; \\
\Pi \ & = \textstyle \prod \limits _{h\in H}\sEz^{*h\ov \tau},
\quad\text{where} \ \ \sEz^{*h\ov \tau} \ = \ 
\{ a\in \sEz^*\mid h\ov \tau(a) = a\};\qquad\qquad\qquad\qquad\qquad\\
X \ & = \ \big\lan \xg\xd \xs\inv \mid 
\ga, \de \in \Ga_\sE\big\ran \ \subseteq  \ \sEz^*.
\end{align*}
 Note that 
$H$ maps $\ker(\wi N)$ (resp.~$\Pi$) to itself, so $H$
acts on $\ker(\wi N) /\Pi$.  But this action is trivial
since $I_H(\ker(\wi N)) \subseteq I_G(\wi \sEz^*) 
\subseteq \Pi$ (see \eqref{IGeq} above).  

\begin{theorem}\label{unitaryDSR}
Suppose $\sE$ is DSR for $\sTR$, i.e.,
in addition to the hypotheses above, $\sE$
  has a maximal 
graded subfield $\sJ$ with $\tau(\sJ) = \sJ$.  Then,
\begin{enumerate}[\upshape (i)]
\item 
$\SK (\sE, \tau)\cong \ker(\wi N)/\Pi$, and there is an 
exact sequence
$$
\textstyle \prod\limits _{h\in H}
\wh H\inv(\lan h\ov \tau\ran, \wi \sEz^*) \ \lra \ 
\wh H\inv(G, \wi\sEz^*) \ \lra \  \SK(\sE, \tau)  \ \lra
\ 1.
$$
\item 
If $\sEz = L_1 \otimes _\sTz L_2$ with each $L_i$ cyclic Galois
over $\sTz$, then 
$$
\SK(\sE, \tau) \ \cong \ \Br(\sEz/\sTz;\sRz) \big/ 
\Dec(\sEz/\sTz;\sRz).
$$
 \end{enumerate}
\end{theorem}

\begin{proof}
(i)  The first formula for $\SK(\sE, \tau)$ was given  
in \cite[Cor.~4.11]{I}.  The point is that the $\xg$ can all 
be chosen in~$\sJ$; then $X \subseteq \sJ_0^{*\tau}
= \sRz^* \subseteq \Pi$, so the $X$ term in \eqref{USK1PiX}
drops out. 
  The exact sequence in (i) 
then follows by Prop.~\ref{gendihedralprop}. Part (ii)
is immediate from~(i) and Prop.~\ref{thmainm}. 
\end{proof}

Note that Th.~\ref{unitaryDSR} is the unitary analogue to 
Prop.~\ref{NSRSK} for nonunitary $\SK$ in the DSR case.

To improve the formula \eqref{USK1PiX} 
in the manner of Th.~\ref{unitaryDSR} for $\sE$ semiramified but not DSR
we need more information  on the contribution of
the $X$ term.  This contribution is measured by $(\Pi\cdot X )/\Pi$.
For $\ga \in \Ga_\sE$ we write 
$\ov \ga$ for $\ga + \Ga_\sT\in \Ga_\sE/\Ga_\sT$.  

\begin{proposition}\label{Gammafn}
There is a well-defined $2$-cocycle $g\in Z^2(
\Ga_\sE/\Ga_\sT , \ker(\wi N) / \Pi)$ given by 
\begin{equation}
g(\ov \ga, \ov \de) \ = \ x_\ga x_\de x_{\ga+\de}\inv \, \Pi.
\end{equation} 
This $g$ is independent of the choice of nonzero symmetric elements
$x_\ga, x_\de , x_{\ga + \de}$ in $\sE_\ga, \sE_\de, \sE_{\ga+\de}$.
Furthermore, 
for all  $\ov \ga,\ov \de \in \Ga_\sE/\Ga_\sT$
and $i,j, k, \ell\in \zz$, we have
\begin{equation}\label{gformula} 
g(i\ov \ga+ j\ov \de, k\ov \ga + \ell\ov \de)  \ = \ g(\ov \ga,
\ov \de)^\Delta\quad \text{where}\ \ \Delta = 
\det\left(\begin{smallmatrix}
i&j\\ k&\ell
\end{smallmatrix}\right).
\end{equation}
$($In particular, $g(\ov \ga, \ov \ga) = 1 \, \Pi$ and $g(\ov \de , \ov \ga)
 = g(\ov \ga, \ov\de) \inv$.$)$  Moreover, $\lan\im(g) \ran =
\big(\Pi\cdot X\big) /\Pi$, 
which is a finite group.
\end{proposition}

\begin{proof}
For $\ga, \de\in \Ga_\sE$, set 
$$
c_{\ga,\de} \, = \, x_\ga x_\de x_{\ga +\de}\inv  \, \in \sEz^*.
$$
Note that $c_{\ga, \de} \in\ker(\wi N)$, since it is a product 
of $\tau$-symmetric elements of $\sE^*$.
 For notational convenience we work with the
function 
$$
f\colon \Ga_\sE \times \Ga_\sE \,\lra \, \ker(\wi N) /\Pi\quad 
\text{given by} \ \ f(\ga, \de)  \, = \,  c_{\ga, \de}\, \Pi.
$$
Thus, $g(\ov \ga, \ov \de) = f(\ga, \de)$
We first show that the definition  of $f$ is independent of 
the choices made of $x_\ga, x_\de, x_{\ga + \de}$. 
Fix $\ga$ and $\de$  in $\Ga_\sE$ for the moment. Take any 
$a\in \sEz^*$ with 
$\tau(ax_\ga ) = a x_\ga$. Then,\break 
${a\xg = \tau(a\xg) = \xg\ov \tau(a)
= \Theta_\sE(\ga)( \ov \tau(a)) \xg}$;
so ${a = \Theta_\sE(\ga)(\ov\tau(a))}$, i.e. 
$a\in \sEz^{*\Theta_\sE(\ga)\ov \tau}\subseteq \Pi$.
Hence, if we let $\xg' =a \xg$, then 
$\xg'\xd\xs\inv \equiv \xg\xd \xs\inv\ (\modd\ \Pi)$.
Likewise, if we take any $b\in \sEz^*$ with 
${\tau(b\xd) = b\xd}$, then $\Theta_\sE(\ga)(b) \in 
\sEz^{*\Theta_E(2\ga+\de) \ov \tau}\subseteq \Pi$
so ${\xg\xd'\xs\inv = \Theta_\sE(\ga)(b) \xg\xd\xs\inv\equiv 
\xg\xd\xs\inv \ (\modd \ \Pi)}$.  Again, for\break
 $d\in \sEz^*$
with $\tau(d\xs) = d\xs$, we have $d\in \sEz^{*\Theta_\sE(\ga + \de)
\ov \tau}\subseteq \Pi$, so for $\xs' = d\xs$, 
we have\break
 ${\xg\xd\xs^{\prime-1} \equiv
\xg\xd \xs\inv\ (\modd \ \Pi)}$.  Thus, each such change does not affect
the value of $f(\ga, \de)$, and we are free to make such 
changes when convenient.

We prove further identities for the function $f$ which hold for all 
$\ga, \de, \ep \in \Ga_\sE$ and $i,j,k,\ell \in \zz$:
$$
\qquad (\text{i})\qquad f(\ga+\be, \de) \, = \, f(\ga, \de) \,=\, f(\ga, \de + \be)
\quad\text{for any}\ \  \be \in \Ga_\sT. \qquad\qquad\qquad\qquad\qquad\qquad
\qquad\qquad\ \, 
$$
For, as $\Ga_\sR = \Ga_\sT$, there is a nonzero $a \in \sR_\be$.
Since $a\in Z(\sE)$ and $\tau(a) = a$, we could have chosen 
$x_{\ga+\be} = a\xg$, $x_{\de+\be} = a\xd$, and 
$x_{\ga+\de +\be}= a\xs$.  Then,
$$
f(\ga+\be,\de) \, = \,(a\xg) \xd (a\xs)\inv\,\Pi \, = \, 
\xg\xd\xs\inv \, \Pi \, =\, f(\ga, \de),  
$$
and likewise $f(\ga, \de+\be) = f(\ga, \de)$.  This proves (i), 
which shows that the $g$ of the Prop. is well-defined.
$$
\qquad\,(\text{ii})\qquad f(i\ga,j\ga) \, = \, 1\,\Pi.\qquad\qquad\qquad
\qquad\qquad\qquad\qquad\qquad\qquad\qquad\qquad\qquad\qquad\qquad\qquad
\qquad\qquad
$$
For, we can choose $x_{i\ga} = \xg^i$, $x_{j\ga} = \xg^j$, and 
$x_{i\ga +j\ga} = \xg^{i+j}$. 
  Then, $c_{i\ga,j\ga} = 1$.
$$
\qquad(\text{iii})\qquad f(\de,\ga)\ = \ f(\ga,\de)\inv.
\qquad\qquad\qquad\qquad\qquad\qquad\qquad\qquad\qquad\qquad\qquad
\qquad\qquad\qquad\qquad\qquad
$$
For, by applying $\tau$ to the equation $\xg\xd = c_{\ga,\de} \xs$,
we obtain
$$
\xd \xg\, =\, \xs\ov\tau(c_{\ga, \de})\, =\, \ThetaE(\ga +\de)
(\ov \tau(c_{\ga, \de}))x_{\de+\ga}, 
$$
yielding $c_{\de,\ga} = \ThetaE(\ga+\de)(\ov \tau(c_{\ga,\de}))$, 
so $c_{\de,\ga} c_{\ga,\de}\in \sEz^{*\ThetaE(\ga+\de)\ov \tau}
\subseteq \Pi$.  Formula (iii) then follows.
$$
\qquad(\text{iv})\qquad f(\ga,\de) f(\ga+\de, \ep)  \,= \,
 f(\ga, \de+\ep) f(\de, \ep), \qquad\qquad\qquad\qquad\qquad\qquad
\qquad\qquad\qquad\qquad\qquad\quad \ \ 
$$
i.e., $f\in Z^2(\GaE, \ker(\wi N) /\Pi)$, since 
$\GaE$ (acting via $\ThetaE(\GaE) = H$) acts trivially on 
$\ker(\wi N) /\Pi$. This identity follows
from $(\xg\xd)x_\ep = x_\ga(\xd x_\ep)$, which yields 
$c_{\ga, \de}\, c_{\ga+ \de, \ep} = 
\ThetaE(\ga)(c_{\de,\ep}) \,c_{\ga, \de +\ep}$.
Then (iv) follows, given the trivial action of 
$H$ on  $\ker(\wi N) /\Pi$.  
$$
\qquad(\text{v})\qquad f(\ga + \de, \de) \,=\, f(\ga, \de)
\quad \text {and} \quad f(\ga, \ga + \de) \,=\, f(\ga, \de).
\qquad\qquad\qquad\qquad\qquad\qquad \qquad\qquad\qquad
$$
For, as $\tau(\xd\xg\xd) = \xd\xg\xd$, we can take 
$x_{\ga+2\de} = \xd\xg\xd$.  Then,
$$
\xd\xg\xd\, = \, c_{\de,\ga} \,c_{\de + \ga, \de} \,x_{\ga + 2\de}
\, = \, c_{\de,\ga} \,c_{\de + \ga, \de} \,\xd\xg\xd.
$$
Hence, $1\,\Pi = f(\de,\ga) f(\de+\ga, \de)$, so 
$f(\de+\ga, \de) = f(\de, \ga) \inv = f(\ga,\de)$, using (iii).
This proves the first formula in (v), and the second formula 
follows analogously, or from the first by using (iii).
$$
\qquad(\text{vi})\qquad f(\ga + j\de,\de) \, =\, f(\ga,\de)
\, = \, f(\ga, j\ga +\de) \quad\text{for all} \,\, j\in \zz.
\qquad\qquad\qquad\qquad\qquad\qquad\qquad\qquad\qquad
$$
This follows from (v) by  induction on $j$.
$$
\quad\ \ \ \, (\text{vii})\qquad f(i\ga, j\de) \, = \, f(\ga, \de) ^{ij} .
\qquad\qquad\qquad\qquad\qquad\qquad\qquad\qquad\qquad\qquad\qquad
\qquad\qquad\qquad\qquad\ \ 
\quad
$$
For, by (iv) with $j\de$ for $\de$ and $\de$ for $\ep$,
$$
f(\ga, j\de) f(\ga+ j\de,\de) \, = \, 
f(\ga,(j+1)\de)f(j\de,\de),
$$
which by (vi) and (ii) reduces to 
$ f(\ga, j\de) f(\ga, \de) = f(\ga, (j+1) \de)$. Then (vii) for 
$i = 1$ follows by induction on $j$ with the initial 
case $j = 0$ given by (ii).  From the $i = 1$ case the result for 
arbitrary $i$ follows by using~(iii).  
$$
\quad\ \ \, (\text{viii})\qquad  
f(i \ga+ j \de, k \ga + \ell \de)  \ = \ f( \ga,
 \de)^\Delta\quad \text{where}\ \ \Delta = 
\det\left(\begin{smallmatrix}i&\,j\\ k&\ell
\end{smallmatrix}\right).
\qquad\qquad\qquad\qquad\qquad\qquad\qquad\qquad\quad\ \ 
$$
For this note first that this is true if $i = 0$, as
$$
f(j\de,k\ga+\ell\de)  \, = \,  f(\de,k\ga+\ell\de)^j
\, = \, f(\de, k\ga)^j \, =\,  f(\ga,\de)^{-jk},
$$  
by (vii), (vi), (vii), and (iii).
Analogously, (viii) is true if $k = 0$.  To verify 
(viii) in general, we argue by induction on 
$|i| + |k|$.  By invoking (iii) and interchanging
$i\ga + j \de$ with $k\ga + \ell\de$ if necessary, 
we can assume $|i| \le |k|$. We can assume 
$|i|\ge 1$, since the case $|i| = 0$ is already done.
Let $\eta = \pm1$, with the sign chosen so that $|k - \eta i| 
= |k| - |i|$. Since $|i|+|k-\eta i|= |k|<|i|+|k|$, we have
by (vi) and induction,
\begin{align*}
f(i\ga+ j\de, k\ga + \ell \de)  \, &= \,  
f\big(i\ga + j\de, (k\ga + \ell\de\big)
- \eta(i\ga + j\de))\, 
=\,f\big(i\ga + j\de, (k- \eta i)\ga + (\ell - \eta j)\de\big)
\\ &= \, f(\ga,\de)^{\Delta'}\quad\text{where} \ \ \Delta' = 
\det\left(\begin{smallmatrix}i &j\\
 k-\eta i &  \ \ell-\eta j\end{smallmatrix}\right) = 
\det\left(\begin{smallmatrix}i &j\\
 k & \,\ell\end{smallmatrix}\right).
 \end{align*} 
Thus, (viii) is proved, and when  (viii) is restated in terms of 
$g$, it is formula~\eqref{gformula}. It is clear from the definition and 
well-definition of $g$ that $\lan \im(g) \ran= (\Pi\cdot X)/\Pi$.  This 
abelian group is finite since the domain of $g$  
is finite, and each $g(\ov \ga,
\ov \de)$ has finite order by formula~\eqref{gformula}. Identity (iv) above 
shows that $f$ is a $2$-cocycle, so $g$ is also a $2$-cocycle.   
\end{proof}

Remark. If the finite abelian group $\Ga_\sE/\Ga_\sT$ has exponent $e$,
then formula \eqref{gformula} shows that $\lan\im(g)\ran$ has exponent
dividing $e$.  So, we have the crude upper bound 
$|\lan\im(g)\ran| \le e^{|\Ga_\sE/\Ga_\sT|^2}$.

We can now prove a formula for unitary $\SK$ of semiramified graded 
algebras.  This is a unitary analogue to Th.~\ref{goth} above. 

\begin{theorem}\label{main}
Let  $\sE$ be a semiramified $\sT$-central graded division algebra
with a unitary graded $\sTR$-involution $\tau$, where
$\sT$ is unramified over $\sR$.  Take any decomposition 
$\sE \sim_g \sI \otimes _\sT \sN$ where $\sI$ is inertial 
with ${[\sI_0]\in \Br(\sEz/\sTz;\sRz)}$ and $\sN$ is DSR
for $\sTR$, as in Prop.~\ref{uINdecomp}  above. Then, 
\begin{enumerate} 
  \item[{\rm(i)}]  $\SK(\sE,\tau)  \cong
\big(\ker(\wi N)/\Pi\big)\big/\lan\im(g)\ran$, where 
$g$ is the function of  Prop.~\ref{Gammafn}. If 
$\sI_0 \sim A(\sEz/\sTz, \bsi, \bu, \bold b)$ as in  
Lemma~\ref{unitarycp}$($iii$)$  with $\theta = \tau|_{\sEz}$, then 
$\im(g)$ is computable from the~$u_{ij}$.
  \item[{\rm(ii)}]
If $\sEz\cong L_1\otimes_\sTz L_2$ with each $L_i$ cyclic Galois over 
$\sTz$, then 
\begin{equation*} 
\SK(\sE,\sT) \ \cong \  
\Br(\sEz/\sTz;\sRz)\big/[\Dec(\sEz/\sTz;\sRz)\cdot
 \lan[\sI_0]\ran]. 
\end{equation*}
\end{enumerate}
\end{theorem}
 
\begin{proof}
(i)  From \eqref{USK1PiX} and Prop.~\ref{Gammafn}, 
we have 
$$
SK(\sE, \tau) \ \cong \ \big( \ker(\wi N) /\Pi\big) \big / 
\big[(\Pi\cdot X)/\Pi\big]
\ \cong \ \big(\ker(\wi N)/\Pi\big)\big/\lan\im(g)\ran.
$$
It remains to relate $\im(g)$ to the $u_{ij}$ describing $\sI_0$.

We have $\sI_0\sim A(\sEz/\sTz, \bsi, \bu,\bold b)$, as in 
Lemma~\ref{unitarycp}(iii), with $\theta = \ov \tau = \tau|_{\sTz}$.
Since $\sN$ is DSR for~$\sTR$ with $\sN_0\cong\sEz$ 
and $\Theta_\sN = \Theta_\sE$ by 
Prop.~\ref{uINdecomp},  Prop.~\ref{DSRdecomp} yields
$\sN\cong_g \sA(\sE_0\sT/\sT, \bsi,\boldsymbol 1, \bold c)$,
with each $c_i \in \sR^*$\break  with
$\deg(c_i)= r_i \ga_i$ for some $\ga_i \in \Ga_\sN = \Ga_\sE$
with $\Theta_\sE(\ga_i) = \si_i$.  Therefore, by 
Remark~\ref{abeliancpprod},\break
$\sE\sim_g \sE'$, where ${\sE' = \sA(\sEz\sT/\sT, \bsi, \bu, \bold d)}$, 
with the same $\bu$ as for $\sI_0$ and each $d_i = b_i c_i\in \sEz^*\sR^*$.
So,\break ${\tau(d_i) = \tau(c_i) \tau(b_i) = c_i b_i = b_i c_i = d_i}$.
Since $\sN$ is a semiramified graded division algebra and 
${\deg(d_i) = \deg(c_i)}$ for each $i$, Lemma~\ref{incp} applied 
to $\sN$ and to $\sE'$ shows that $\Gamma_{\sE'} = \Gamma_\sN$
and $\sE'$ is a semiramified graded division algebra.
Therefore, as $\sE$ and $\sE'$ are each graded division algebras 
with $\sE\sim_g \sE'$, we have 
$\sE \cong_g \sE'$ by the graded Wedderburn Theorem.   
So, we may assume $\sE = \sE' = \sA(\sEz\sT/\sT, \bsi, \bu, \bold d)$.
Take $y_1, \ldots, y_k \in \sE^*$ with $\intt(y_i)|_{\sEz\sT} = \si_i$, 
$y_i^{r_i} = d_i$, and $y_i y_j y_i\inv y_j\inv = u_{ij}$.  Now, 
the graded field $\sEz\sT$ is $\sTR$ generalized dihedral, 
and $\theta = \tau|_{\sEz\sT}$ lies in $\sGal(\sEz \sT/\sR) \setminus
\sGal(\sEz\sT/\sT)$.   Therefore, the proof of 
Lemma~\ref{unitarycp}~(iii)~$\Rightarrow$~(i) shows that there is a 
graded $\sTR$-involution $\tau'$ of $\sE$ with each $y_i =\tau'(y_i)$
and $\tau'|_{\sEz\sT} = \theta$. Since $\SK(\sE,\tau) = \SK(\sE,\tau')$
we may replace~$\tau$ by $\tau'$, so each $y_i = \tau(y_i)$, while 
$\ov \tau$ is unchanged.

Fix any 
$ \eta\in \Ga_\sE/\Ga_\sT$, and let 
$\si_{\eta} = \ov\Theta_E(\eta) \in H$.  Take the unique
$\bold i \in \mathcal I$  with $\si^{\bold i} = \si_{\eta}$
(notation as in~\S\ref{abcp}),  
 let $\ga =\deg(y^{\bold i})\in \Ga_\sE$, and set 
$y_\ga = y^{\bold i}$.  Since $\Theta_\sE(\ga) = 
\intt(y_\ga)|_{\sEz} = \ov \Theta_\sE(\eta)$ and 
$\ov \Theta_\sE\colon \Ga_\sE/\Ga_\sT \to H$ is an 
isomorphism for $\sE$ semiramified (see \S\ref{graded}), $\eta = \ov \ga$ in 
$\Ga_\sE/\Ga_\sT$.  
Since $\tau(y_i) = y_i$ for each $i$, 
$\tau(y_\ga)$ is the product of the $y_i$ appearing in $y_\ga$ but with 
the order reversed. Hence, the commutator identities show that 
$\tau(y_\ga) = a_\ga y_\ga$ where $a_\ga$ in $\sEz$ is a computable product
of the $u_{ij}$ and their conjugates under the $y_i$.  Since 
each $y_\ell u_{ij}y_\ell\inv = \si_\ell(u_{ij})$,  $a_\ga$ is a 
computable product of  terms $\si_\ell(u_{ij})$.  
(For example, ${\tau(y_1 y_2y_3) = y_3 y_2y_1 =
[u_{32}\si_2(u_{31})u_{21}] y_1y_2y_3}$.)
By applying $\tau$ to the equation 
$\tau(y_\ga) = a_\ga y_\ga$, we find
$$
a_\ga \, \si_\eta\tau(a_\ga) \, = \, 1.
$$
Therefore, from Hilbert 90 for the quadratic extension 
$\sEz/\sEz^{\si_\eta\tau}$, there is $t_\ga\in \sEz^*$ with 
$$
t_\ga \,[\si_\eta\tau(t_\ga)]\inv \, = \, a_\ga.
$$
Then, $\tau(t_\ga y_\ga) = t_\ga y_\ga$, so for the $x_\gamma$ in $X$ 
we can set $x_\ga = 
t_\ga y_\ga$. Now take any $\zeta \in 
\Ga_\sE/\Ga_\sT$ and  
carry out the same process for  $\zeta $ 
as we have just done for $\eta$, obtaining
$\delta \in \Ga$ with $\ov \de = \zeta$, and $y_\de$ with 
$\deg(y_\de) = \de$ and $\intt(y_\delta)|_{\sEz} = \si_\zeta$, 
then determining   $a_\de$, $t_\de$, $x_\de$.
Then set $y_{\ga+\de} = y_\ga y_\de$, so  
${\intt(y_{\ga+\de})|_{\sEz} = \si_\eta\si_\zeta}$.\break
  Let ${a_\gade = 
\tau(y_{\gade}) y_\gade\inv \in \sEz^*}$.  Since 
$a_\gade \si_\eta\si_\zeta \tau(a_\gade) = 1$, by Hilbert 90 there is 
$t_\gade\in \sEz^*$ with\break
 ${t_\gade[\si_\eta\si_\zeta\tau(b_\gade)]
 \inv = a_\gade}$.
Then set $x_\gade = t_\gade y_\gade$, so that $\tau(x_\gade) = x_\gade$.  
By the definition of the function $g$ of Prop.~\ref{Gammafn}, we have
in $\ker(\wi N)/\Pi$,
$$
g(\eta, \zeta) \,=\, x_\ga x_\de x_\gade\inv \,\Pi
=\, (t_\ga y_\ga) (t_\de y_\de)(t_\gade y_\ga y_\de)\inv \,\Pi\,
= \, t_\ga \si_\eta(t_\de) t_\gade\inv\,\Pi.
$$
Since the $t$'s are determined by the $a$'s, which are determined by 
the $u_{ij}$, this shows that $\im(g)$ is determined by the $u_{ij}$.

(ii)  Suppose now that $\sEz = L_1\otimes_\sTz L_2$ with each 
$L_i$ cyclic Galois over $\sTz$, and let $\si = \si_1$
and $\rho = \si_2$, as in \S\ref{ubicyclic}.  The isomorphism
\begin{equation}\label{unitarybicyc}
 \Br(M/K;F)\big/
\Dec(M/K;F) \,\cong \,\ker(\widetilde{N})/\Pi
\end{equation}
of Prop.~\ref{thmainm} maps $[\sI_0] = [A(u,b_1,b_2)]$
to $q\,\Pi$, where $q\in \sEz^*$ with $u = q[\rho\si\ov \tau(q)]\inv$.
Take standard generators $y_1,y_2$ of $A(u,b_1,b_2)$.  As noted for (i),
we can assume after modifying $\tau$ (without changing $\ov\tau$)
that $\tau(y_1) = y_1$ and $\tau(y_2) = y_2$. Let 
$\ga = \deg(y_1)$ and $\de = \deg(y_2)$ in $\Ga_\sE$, so 
$\Theta_\sE(\ga) = \intt(y_1)|_{\sEz} = \si$ and 
${\Theta_\sE(\de) = \intt(y_2)|_{\sEz} = \rho}$.   Since  
$\Ga_\sE/\Ga_\sT \cong H = \lan\si,\rho \ran$, we have
$\Ga_\sE/\Ga_\sT = \lan\ov \ga, \ov \de \ran$.  As 
$\tau(y_1) = y_1$, we can take $x_\ga = y_1$, and likewise
$x_\de = y_2$.  Because $\tau(y_2y_1) = uy_2y_1 = 
q[\rho\si\ov \tau(q)]\inv y_2y_1$, we have 
$\tau(q y_2y_1) = qy_2y_1$; thus, we can take 
$x_{\de+\ga} = qy_2y_1$.  Then,
$$
g(\ov \de, \ov \ga) \, = \, x_\de x_\ga x_{\de +\ga}\inv \,\Pi
\, = \, y_2y_1(qy_2y_1)\inv\,\Pi \, = \, q\inv \,\Pi.
$$
Since $\ov\de$ and $\ov\ga$ generate $\Ga_\sE/\Ga_\sT$ formula~\eqref{gformula}
shows that $\im(g) = \lan g(\ov\de, \ov \ga)\ran = \lan q\inv\Pi\ran
= \lan q\,\Pi \ran$. Therefore, the isomorphism of \eqref{unitarybicyc}
maps $\lan[\sI_0]\ran$ to $\lan q\,\Pi\ran = \lan\im(g) \ran$. Thus, the 
isomorphism asserted for (ii) follows from~(i).
\end{proof}

\begin{example}
Here is a unitary version of Ex.~\ref{cyclicex}. Take any integer 
$n\ge 2$, and let $F\subseteq K$ be fields with $[K\col F] = 2$, 
$K$ Galois over $F$, and $K = F(\omega)$ where $\omega$ is a primitive 
$n^2$-root of unity.  Suppose further that for the nonidentity element
$\psi_0$ of $\Gal(K/F)$ we have $\psi_0(\omega) = \omega\inv$.  
(For example, 
we could take $K = \qq(\omega)$, the $n^2$-cyclotomic extension of  
$\qq$, and $F = K\cap \rr$.)  Let $\sT = K[x,x\inv,y,y\inv]$, the Laurent
polynomial ring, with its usual grading by $\zz\times\zz$;
so, $\sT$ is a graded field.  Let
$\sR = F[x,x\inv,y,y\inv]$, which is a graded subfield of $\sT$
with $[\sT\col\sR] = 2$, $\sT$ Galois over $\sR$, and 
$\sT$ inertial over $\sR$.  Also, $\sGal(\sTR) = \{\psi, \id_\sT\}$,
where $\psi = \psi_0\otimes \id _\sR$ on $\sT = \sTz \otimes_\sRz \sR$.
Take any $a,b \in F^*$ such that $[K(\sqrt[n]a,\sqrt[n]b\,)\col K] = n^2$,
and let $M = K(\sqrt[n]a,\sqrt[n]b\,)$.  Then, it is easy to check that 
$M$ is $K/F$-generalized dihedral.  (One can think of such field extensions
$M/F$ as the generalized dihedral analogue to Kummer extensions.) 
Indeed,  $\psi_0$ on $K$ extends to $\theta\in \Gal(M/F)$
given by $\theta (\sqrt[n]a) =\sqrt[n]a$, $\theta(\sqrt[n]b) =
\sqrt[n]b$, and $\theta|_K = \psi_0$; so, $\theta^2 = \id_M$, 
and for $h\in \Gal(M/K)$, we have $\theta h\theta = h\inv$.
As in 
Ex.~\ref{cyclicex}, take the graded symbol algebra 
$\sE = (ax^n,by^n,\sT)_\omega$ of degree $n^2$, with its generators $i,j$ 
satisfying $i^{n^2} = ax^n$, $j^{n^2} = by^n$, $ij = \omega ji$.  
For $\si_1, \si_2$ as in Ex.~\ref{cyclicex}, it was noted there 
that $\sE = \sA(M\sT/\sT, \bsi, \bold u, \bold d)$ where 
$u_{12} = \omega$, and $d_1 = 1/(y\sqrt[n]b)$ and $d_2 = x\sqrt[n]a$.  
We extend $\theta$ to an element of $\sGal(M\sT/\sR)$ by setting 
$\theta|_\sR = \id$.  Since $\theta(d_1) =d_1$, $\theta(d_2) =d_2$,
and $u_{12}\, \si_1\si_2\theta(u_{12}) = \omega\omega\inv = 1$, the
graded version of Lemma~\ref{unitarycp} shows that there is a 
graded $\sT/\sR$-involution $\tau$ on $\sE$ given by $\tau(j\inv)
= j\inv$, $\tau(i) = i$, and $\tau|_{M\sE} = \theta$.     
That is, $\tau$ is 
the $\sR$-linear map $\sE \to \sE$ such that 
$\tau(c\,i^\ell j^m) = \psi(c) j^m i^\ell$ for all $c\in \sT$, 
$\ell,m\in \zz$.  We have 
the decomposition of $\sE$ noted in Ex.~\ref{cyclicex},
$$
\sE\,\sim_g\,  \sI \otimes_\sT \sN\qquad \text{where} \qquad
\sI \, = \, (a,b, \sT)_\omega \quad \text{and} \quad
\sN \, = \, (x,b,\sT)_{\omega^n}\otimes_\sT (a,y,\sT)_{\omega^n}.
$$
These $\sI$ and $\sN$ are $\sT$-central graded division algebras 
with $\sI$ inertial and $\sN$ DSR.  Furthermore, as\break
 ${a,b,x,y\in \sR^*}$,
there are unitary graded $\sTR$-involutions $\tau_\sI$ on $\sI$
and $\tau_\sN$ on $\sN$ defined analogously to $\tau$ on $\sE$.  So,
by Th.~\ref{unitaryDSR}(ii) 
$$
\SK(\sN, \tau_\sN) \ \cong  \ 
\Br\big(M/K;F\big)\big/
\Dec\big(M/K;F\big), \quad\text{where}\quad M \, =
 \,K(\sqrt[n]a, \sqrt[n]b\,),
$$
with $\Dec\big(M/K;F\big) = 
\Br\big(K(\sqrt[n]a\,)/K;F\big) \cdot
\Br\big (K(\sqrt[n]b\,)/K;F\big)$ by \eqref{Decformula}.
   Since $\sI_0 \cong(a,b,K)_\omega$, 
Th.~\ref{main}(ii) yields
$$
\SK(\sE, \tau) \ \cong \ \Br(M/K;F)\big/
\big[\Dec(M/K;F)\cdot \lan(a,b,K)_\omega \ran
\big].
$$
Note that $\sE$ is semiramified, but it may or may not be DSR.  Indeed, by 
Prop.~\ref{uINdecomp}(ii) $\sE$ is DSR if and only if  
${\sI_0 \in \Dec\big(M/K;F\big)}$; the formulas
above show that this holds if and only if  the obvious surjection\break 
${\SK(\sN, \tau_N) \to \SK(\sE, \tau)}$ is an isomorphism. 
Note also that $\Dec(M/K;F)$ may be strictly smaller than 
$\Dec(M/K) \cap \Br(M/K;F)$, i.e., there may be  an algebra in $\Br(M/K)$
which decomposes according to $M$ and has a $K/F$-involution, but in any
decomposition the factors do not have  $K/F$-involutions.  Examples of this 
are given in Remark~\ref{indecs} below.

For an  
ungraded version of this example, let $K$, $F$, $a$, and $b$ be as above;
then
let $K' = K((x))((y))$ and 
$F' = F((x))((y))$, and ${D = (ax^n,by^n,K')_\omega}$.  Then, 
with respect to the usual rank $2$ Henselian valuations $v_{K'}$ on 
$K'$ and $v_{F'}$ on $F'$, $K'$ is inertial of degree $2$ over $F'$.  
Furthermore, with respect to the valuation $v_D$ on $D$ extending
$v_{K'}$ on $K'$, 
$D$ is a semiramified $K'$-central division algebra
with a unitary $K'/F'$-involution $\tau_D$ defined just as for $\tau$
on $\sE$.  For the associated graded ring $\gr(D)$ of $D$ 
determined by $v_D$, we have
 $\gr(D) \cong_g \sE$, so by \cite[Th.~3.5]{I}
$\SK(D, \tau_D) \cong \SK(\sE, \tau)$.  
\end{example}

\section{Noninjectivity} 

For any $\sT$-central graded division algebra $\sB$ 
with unitary $\sTR$-involution $\tau$, 
there are well-defined canonical homomorphisms
\begin{equation}\label{alpha}
\al\colon \SK(\sB, \tau)  \to \SK(\sB) 
\quad\text{given by}   \ \ 
a\,\Sigma_\tau(\sB) \mapsto \tau(a) a\inv\, [\sB^*,\sB^*]  
\ \ \text{for\ } a\in \Sigma_\tau'(\sB), \qquad \ \ \, 
 \end{equation}
\vskip -0.15truein
and
\vskip-.25truein
\begin{equation*}
\be\colon \SK(\sB) \to \SK(\sB, \tau) \quad \text{given by}
 \  \ b \, [\sB^*,\sB^*] \mapsto b\, \Sigma_\tau(\sB)
\ \ \text{for\ } b\in \sB^*  \ \text {with} \ \Nrd_\sB(b) = 1.
\      
\end{equation*}
It is easy to check that $\beta \circ \alpha$ and 
$\al \circ \be$ are each the squaring map.  As
pointed out in \cite[Lemma, p.~185]{y}, since the   
exponent of the abelian group $\SK(\sB,\tau)$ divides
$\deg(\sB)$, if $\deg(\sB)$ is odd, then $\al$ must be injective.
It seems to have been an open question up to now whether 
$\alpha$ is always injective, even when $\deg(\sB)$ is even.  
We now settle this question by using some of the results above 
to give  examples of $\sB$ of  
degree $4$ with $\al$~not injective.  
We thank J.-P.~Tignol for pointing out the relevance of 
indecomposable division algebras of degree $8$ and exponent $2$, 
and for calling his paper \cite{tigtri} to our attention. 

Let $F$ be a field with $\chr(F) \ne 2$.  Let $M = F(\sqrt a, \sqrt b,
\sqrt c\,)$ 
with $a,b,c\in F^*$ and $[M\col F] = 8$.  Let ${K = F(\sqrt a\,)}$.  
We write $\Br_2(F)$ for the $2$-torsion subgroup of $\Br(F)$, 
and set ${\Br_2(M/F) = \Br(M/F)\cap \Br_2(F)}$, $\Br_2(M/K;F)
= \Br(M/K;F) \cap \Br_2(K)$, etc.  Note that as $\Gal(M/F)$ is an 
elementary 
abelian $2$-group, $M$ is  a $K/F$-generalized dihedral extension.
Also, $\res_{F\to K}$ maps $\Br_2(M/F)$ to $\Br(M/K;F)$, since for 
$[A] \in \Br_2(M/F)$, $\cors_{K\to F}  [A\otimes_F K]
= [A]^{[K: F]} = 1$ in $\Br(F)$, so by Albert's Theorem 
$A\otimes_F K$ has a 
unitary $K/F$-involution.

\begin{proposition}\label{Br2}
There is an exact sequence:
\begin{equation}\label{decexact}
0  \ \lra  \ \Br_2(M/F)\big/\Dec(M/F)  \ \lra \  
\Br(M/K;F)\big/ \Dec(M/K;F)\ \lra \ 
\Br(M/K)\big/ \Dec(M/K) 
\end{equation}
\end{proposition} 

\begin{proof} 
The kernel of the right map in \eqref{decexact} is 
$\big[\Br(M/K;F) \cap \Dec(M/K)\big]\big/
\Dec(M/K;F)$. So, the 
 exactness of \eqref{decexact} is equivalent to two assertions: 
\begin{equation*}
\qquad(\text{a}) \quad \Br(M/K;F) \, \cap \,  \Dec(M/K) \ = \ \Br_2(M/K;F).
\qquad\qquad\qquad\qquad\qquad\qquad\qquad\qquad\qquad\qquad\qquad
\end{equation*}
\vskip-7pt
\noindent and
\vskip-22pt
\begin{equation*}
\qquad(\text{b}) \quad \Br_2(M/F)\big/ \Dec(M/F) \ \cong \ 
\Br_2(M/K;F)\big/ \Dec(M/K;F)\qquad\qquad\qquad\qquad\qquad\qquad
\qquad\qquad\quad
\end{equation*}
The equality (a) is immediate from the fact that $\Dec(M/K) =\Br_2(M/K)$,
as $M$ is a biquadratic extension of $K$.  (This is well-known, and is 
deducible, e.g., by refining the argument in  \cite[Prop.~16.2]{kmrt}.   
It also appears in  
\cite[Cor.~2.8]{tigtri} as the assertion that property P$_2$(2) holds for $K$.)
The isomorphism (b) appears in \cite[Prop.2.2]{tigtri} as the 
isomorphism ${N_2(M/F) \cong M_2(M/K/F)}$, see the conmments on p.~14 of 
\cite{tigtri}.  Since the isomorphism (b) is somewhat buried in the general 
arguments of 
\cite{tigtri}, we give a short and direct proof of it:  
  If ${[A] \in \Dec(M/F)}$, then $A \sim Q_1\otimes_F Q_2
\otimes _F Q_3$, where $Q_1$ is the quaternion algebra  $\quat arF$,
$Q_2 = \quat bsF$, and $Q_3 = \quat ctF$, for some $r,s,t\in F^*$.  
So, ${A\otimes_FK \sim (Q_2 \otimes_F K)\otimes_K(Q_3 \otimes_FK)}$.\break
Here, $Q_2\otimes _F K$ has the unitary $K/F$-involution 
$\eta \otimes \psi$, where $\eta$ is any involution of the first 
kind
\break 
on $Q_2$ and $\psi$ is the nonidentity $F$-automorphism of $K$.
So $[Q_2\otimes_FK] \in \Br(K(\sqrt b)/K;F)\subseteq \Dec(M/K;F)$; 
likewise
$[Q_3\otimes_F K] \in
\Br(K(\sqrt c)/K;F)\subseteq \Dec(M/K;F)$, and hence 
$[A\otimes_F K] \in \Dec(M/K;F)$.  Thus, $\res_{F\to K}$
induces a well-defined map
${f\colon \Br_2(M/F)/\Dec(M/F) \to 
\Br_2(M/K;F)/\Dec(M/K;F)}$.  From Arason's long  exact sequence (see, e.g.,
\cite[Cor.~30.12(1)]{kmrt} or \eqref{longexact} above)  
$$
{\ldots \to H^2(F, \mu_2) \to H^2(K,\mu_2) 
\to H^2(F, \mu_2) \to \ldots \ },
$$
$f$ is surjective.  For injectivity of~$f$, take any $[A] \in 
\Br_2(M/F)$ with $\res_{F\to K}[A] \in \Dec(M/K;F)$.  We need to show
$[A]\in \Dec(M/F)$.  We have $A\otimes _FK \sim Q_2' \otimes_K  Q_3'$
where the $Q_i'$ are quaternion algebras over~$K$ with 
$Q_2' \in \Br(K(\sqrt b\, )/K;F)$ and $Q_3'\in \Br(K(\sqrt c\, )/K;F)$.  By
a result of Albert \cite[Prop.~2.22]{kmrt}, the quaternion algebra  
$Q_2'$ with $K/F$-involution has the form 
$Q_2' \cong Q_2^{\prime\prime} \otimes _F K$, where 
$Q_2^{\prime\prime}$ is a quaternion algebra over $F$.  
Then, $[Q_2^{\prime\prime}] \in \Br_2(K(\sqrt b\, )/F) =
\Dec(K(\sqrt b \,)/F)$, as noted for (a) above.
Likewise, $Q_3'\cong Q_3^{\prime\prime}\otimes_F K$, where 
$[Q_3^{\prime\prime}] \in \Dec(K(\sqrt c\, )/F)$.
Since $[A\otimes _F Q_2^{\prime\prime} \otimes_F Q_3^{\prime\prime}]
\in \Br(K/F) = \Dec(K/F)$, we have 
$$
[A] \,= \  [A\otimes _F Q_2^{\prime\prime} \otimes_F Q_3^{\prime\prime}]
\,[Q_2^{\prime\prime}]\,[Q_3^{\prime\prime}]
 \ \in \ \Dec(K/F) \cdot \Dec(K(\sqrt b\, )/F) \cdot 
\Dec(K(\sqrt c\, )/F)
 \ \subseteq\,  \Dec(M/F).
$$
Thus, $f$ is an isomorphism, proving (b).
\end{proof}

\begin{remark}\label{indecs}
The term $\Br_2(M/F)/\Dec(M/F)$  for $M/F$ triquadratic has arisen in the 
study of indecomposable algebras $A$ of degree $8$ and exponent $2$. 
Note first that for any  $A$ of degree $8$ and exponent $2$,
  by Rowen's theorem 
\cite[Th.~6.2]{rowen} there is a triquadratic field extension  $M$ of  
the center $F$ of $A$, such that $M$~is a maximal subfield of~$A$.  
If $A$ is indecomposable, then $[A]$ yields a nontrivial element of $\Br_2(M/F)/\Dec(M/F)$.
Examples of indecomposables if degree $8$ and exponent $2$ were
first given in \cite[Th.~5.1]{art}.  Subsequently, Karpenko showed 
in \cite[Cor.~5.4]{karpenko} that if $B$ is a division algebra with center $F$
of degree $8$ and exponent~$8$, and $F'$ is a field generically reducing the
exponent of $B$ to $2$, then $B\otimes_F F'$ is an indecomposable division
algebra of degree $8$ and exponent $2$.  Also, K. McKinnie  in 
her thesis (unpublished), using lattice methods, gave another example of 
indecomposables  of degree $8$ and exponent $2$. There is a kind of converse 
to this as well:  Given a division algebra $A$ with $[A] \in 
\Br_2(M/F)\setminus \Dec(M/F)$, Amitsur, Rowen, and Tignol showed in 
\cite[Th.~3.3]{art} that the associated generic abelian crossed product algebra 
$A'$ of $A$ is indecomposable of degree $8$ and exponent~$2$. 
(It is not stated this way in \cite{art}, but made explicit in 
\cite[\S~2]{tignol}.) 
This $A'$ is the ring of quotients of a semiramified graded division 
algebra $\sE$ of the type considered in previous sections:  $\sE$ is 
graded Brauer equivalent to  $\sI \otimes_\sT \sN$, where 
$\sT$ is a graded field with $\sTz \cong F$, $\sI$ is an inertial graded
division algebra over $\sT$ with $\sI_0\cong A$, and 
$\sN$ is DSR over $\sT$ with $\sN_0 \cong M$.      
\end{remark}

Using Prop.~\ref{Br2} we now construct biquaterion graded algebras where the 
map $\alpha$ of \eqref{alpha} above is not injective.

\begin{example}\label{noninjex}
Let $M$ be a triquadratic extension of a field $F$ ($\chr(F)\ne 2$) with 
 ${\Br_2(M/F) /\Dec(M/F) \ne 0}$.  (Such $F$ and $M$ exist, as noted
in Remark~\ref{indecs}.) Say $M = F(\sqrt a, \sqrt b, \sqrt c\, )$ for
$a,b,c\in F^*$.  Let $K = F(\sqrt a\,)$, and let $H = \Gal(M/K)$.  
Let $\sR = F[x,x\inv,y,y\inv]$, the Laurent polynomnial ring in indeterminates
$x$ and~$y$, with its usual grading in which  $\sR_{(k,\ell)} = Fx^ky^\ell$ for 
all $(k,\ell) \in \zz\times \zz$.  So, $\sR$ is a graded field with $R_0 = F$
and $\Gamma_\sR = \zz\times \zz$.  Let $\sT = K[x,x\inv,y,y\inv]$, a graded
field with $[\sT\col\sR] = 2$, and let $\sE = \sQ\otimes_\sT \sQ'$, where
$\sQ$ and $\sQ'$ are the following semiramified graded quaternion 
division algebras 
over $\sT$:  $\sQ = \quat b x \sT$, which is generated over $\sT$ by 
homogeneous elements $i$ and$j$ with relations $i^2= b$, $j^2 = x$, 
and $ij = -ji$, with $\deg(i) = 0$ and $\deg(j) = (\frac12,0)$. 
So, $\sQ_0 \cong K(\sqrt b\, )$ and $\Ga_\sQ = \frac12\zz\times \zz$. 
Likewise, set $\sQ' = \quat c y\sT$ with standard generators $i'$~and~$j'$,
with $\deg(i') = 0$ and $\deg(j') = (0,\frac12)$, so
$\sQ_0 \cong K(\sqrt c\, )$ and $\Ga_{\sQ'} = \zz\times \frac12\zz$.  
Since $\sQ \cong \quat bx\sR \otimes_ \sR\sT$, $\sQ$~has the 
graded $\sT/\sR$-involution $\tau_\sQ = \eta \otimes \psi$, where 
$\eta$ is the 
canonical symplectic graded involution on $\quat bx\sR$, for which  $\eta(i) = -i$
and $\eta(j) = -j$, and $\psi$ is the nonidentity graded $\sR$-automorphism
of $\sT$. Likewise $\sQ'$ has a graded $\sT/\sR$-involution $\tau_{\sQ'}$ with
$\tau_{\sQ'}(i') = -i'$ and $\tau_{\sQ'}(j') = -j'$. 
By Lemma~\ref{DSRprod}, $\sE$~is a graded division algebra which is 
DSR for $\sTR$ with $\sEz \cong \sQ_0\otimes_ \sTz \sQ_0'
\cong K(\sqrt b\, ) \otimes _K K(\sqrt c\, ) \cong M$ and 
$\Ga_\sE = \Ga_\sQ + \Ga_{\sQ'} = \frac12\zz\times \frac12 \zz$;
our     
graded $\sTR$-involution on $\sE$ is $\tau = \tau_\sQ \otimes \tau_{\sQ'}$.  
(Explicitly, $\sS = \sT[i,i'] \cong_g M[x,x\inv,y,y\inv]$ is a maximal
graded subfield of $\sE$ with $\sS$ inertial over $\sT$, and 
$\sJ = \sT[j,j']\cong_g \sT[\sqrt x, \sqrt x^{\,-1}, \sqrt y, \sqrt y^{\,-1}]$
is a  maximal graded subfield of $\sE$ which is totally ramified over 
$\sT$ with $\tau(\sJ) = \sJ$.)  We claim that the following diagram is 
commutative with all horizontal maps isomorphisms and vertical
maps described below:
\begin{equation}\label{diagram}
\begin{CD}
\Br(M/K;F)\big/\Dec(M/K;F) @>>> \ker(\wi N)/ \Pi @>>>\SK(\sE,\tau)\\
@VVV @VVV @V{\alpha}VV\\
\Br(M/K)\big/\Dec(M/K) @>>> \wh H\inv(H,M^*)@>>> \SK(\sE)
\end{CD}
\end{equation}
The left vertical map is the map in Prop.~\ref{Br2}, whose kernel
is  there shown to be isomorphic to\break 
 $\Br_2(M/F)/\Dec(M/F)$.  
Since we have assumed this kernel is nontrivial, once the claim 
is established the right vertical map $\alpha$, which is the map
of \eqref{alpha} must also have nontrivial kernel, as desired.

We now verify the claim.  In the top line of \eqref{diagram},
$\ker(\wi N) = \{ a\in M^*\mid N_{M/K}(a) \in F\}$ and 
${\Pi = \prod_{h\in H} M^{*h\ov \tau}}$, where $H = \Gal(M/K)$
and $\ov \tau = \tau|_{\sEz}$.  The middle vertical map sends 
$a\,\Pi \mapsto a/\ov\tau(a)\,I_H(M^*)$. It is well defined since
if $a\in \ker(\wi N)$, we have $N_{K/F}(a/\tau(a)) = N_{K/F}(a)/
\tau(N_{K/F}(a)) =1$, and 
if $b\in M^{*h\ov \tau}$, then 
${b/\ov \tau(b)  = h\ov \tau(b)/\ov \tau(b) \in I_H(M^*)}$.  
In the right rectangle of \eqref{diagram}, the top map sends
$a\,\Pi\mapsto a\Sigma_\tau(\sE)$, and the bottom map sends
$b\,I_H(M^*) \mapsto b\,[\sE^*,\sE^*]$, so the 
right rectangle is clearly commutative. The horizontal maps
in this rectangle are the isomorphisms given in Th.~
\ref{unitaryDSR}(i) and Prop.~\ref{NSRSK}(i).  For the left
vertical map take an arbitrary element of $\Br(M/K;F)$,
which has the form $[A]$, where $A = A(u,b_1,b_2)$ in the notation
of \S\ref{ubicyclic}, with $u, \ b_1, \ b_2$ satisfying
the relations in \eqref{urels} and \eqref{brel} and the added relations
in Lemma ~\ref{unitarycp}(iii), notably $u\, \sigma\rho \ov \tau(u)
= 1$.  The horizontal map in the left rectangle is the isomorphism
of Th.~\ref{thmainm} which sends $[A]$ mod $\Dec(M/K;F)$ to 
$q\, \Pi$ for any $q\in M^*$ with $q/\sigma\rho\ov \tau(q) = u$.  
This is mapped downward to $u \,I_H(M^*)$, since
$q/\ov \tau(q) = u\,\sigma\rho\ov \tau(q)/\ov\tau(q) \equiv 
u\ (\modd\ I_H(M^*))$.  On the other hand, $[A]$ mod $\Dec(M/K;F)$
is mapped downward to $[A]$ mod $\Dec(M/K)$, which is mapped to the 
right to $u\, I_H(M^*)$ by the isomorphism of \eqref{njnj}. Thus, 
the left rectangle of \eqref{diagram} is commutative, and its
horizontal maps are isomorphisms, completing the proof of the claim.  
\end{example}

\begin{remark}
For the preceding example with the $\alpha$ of \eqref{alpha}
noninjective, we have worked with graded division algebras.  
There are corresponding examples of division algebras over
a Henselian valued field with the corresponding $\al$ not
injective, obtainable as follows: With fields $F \subseteq K \subseteq M$
as in Ex.~\ref{noninjex}, let $F' = F((x))((y))$, $K' = K((x))((y))$, 
and $M' = M((x))((y))$, which are twice iterated Laurent power
series fields each with it standard Henselian valuation with 
value group $\zz\times \zz$ (with right-to-left lexicographic ordering)
and residue fields $\ov{F'} \cong F$, $\ov{K'}\cong K$, and $\ov{M'}
\cong M$.  Let $D = \quat bx{K'} \otimes_{K'} \quat cy{K'}$, which is 
a division algebra over $K'$, and the Henselian valuation $v_{K'}$ on 
$K'$ extends uniquely to a valuation $v_D$ on $D$, for which 
$\ov D \cong M$ and $\Ga_D = \frac12\zz\times \frac12 \zz$.
For the associated graded ring of $D$ determined by $v_D$, we have 
$\gr(D) \cong_g \sE$ and, as $D$ is tame over $K'$,
  $Z(\gr(D)) = \gr(K')\cong_g \sT$, for the $\sE$ and $\sT$
of Ex.~\ref{noninjex}.  Also, $\gr(F')\cong_g\sR$ for the $\sR$
of Ex.~\ref{noninjex}. 
This $D$ has a unitary $K'/F'$-involution $\tau_D$, since each 
constituent quaternion algebra has such an involution.  
Because the Henselian valuation $v_{F'}$ on $F'$ has a unique  
extension to $K'$, namely $v_{K'}$, and $v_D$~is the unique extension 
of $v_{K'}$ to $D$, we must have $v_D\circ \tau_D = v_D$.  
Therefore, $\tau_D$ induces a graded involution 
$\wi \tau$ on~$\sE$, which is a unitary $\sTR$-involution.
By \cite[Th.~3.5]{I} and \cite[Th.~4.8]{hazwadsworth},  
$\SK(D,\tau_D) \cong \SK(\sE, \wi \tau)$ and $\SK(D) \cong
\SK(\sE)$.  These isomorphisms are compatible with the map
$\al_{\wi\tau}\colon\SK(\sE, \wi \tau) \to \SK(\sE)$
and the corresponding map $\al_D\colon \SK(D, \tau_D) \to 
\SK(D)$.  Also, because $\wi \tau$ and the $\tau$ of 
Ex.~\ref{noninjex} are each graded $\sT/\sR$-involutions
on $\sE$, we have $\SK(\sE, \wi \tau) \cong \SK(\sE,  \tau)$,
and it is easy to check that under this isomorphism
$\alpha_{\wi\tau}$ corresponds to the $\al$ of Ex.~
\ref{noninjex}.  Since this $\al$ is not injective, 
$\al_D$ is also noninjective. 
\end{remark}

 


\end{document}